%% file: main_arxiv.tex
\documentclass[11pt,a4paper,oneside]{amsart}

\input{header}
\usepackage{xr}
\input{def}
\begin{document}

\title[PDMP and their invariant measures]{Piecewise Deterministic Markov Processes and their invariant measures}
 \author[A. Durmus, A. Guillin, P. Monmarché]{Alain Durmus,  Arnaud Guillin, Pierre Monmarché}

\address{CMLA, ENS Cachan, CNRS, Université Paris-Saclay, 94235 Cachan, France}
\email{alain.durmus@cmla.ens-cachan.fr}

\address{Laboratoire Jacques-Louis Lions and Laboratoire de Chimie Th{\'e}orique, Sorbonne Universit{\'e}}
\email{pierre.monmarche@sorbonne-universite.fr}
\urladdr{https://www.ljll.math.upmc.fr/monmarche/}

\address{Laboratoire de Math\'ematiques Blaise Pascal,
CNRS UMR 6620,
Universit\'e Clermont-Auvergne}
\email{guillin@math.univ-bpclermont.fr}
\urladdr{http://math.univ-bpclermont.fr/$\sim$guillin}

% \subjclass[2010]{Primary }

% \keywords{}

\date{}

\begin{abstract}
  Piecewise Deterministic Markov Processes (PDMPs) are studied in a
  general framework. First, different constructions are proven to be
  equivalent. Second, we introduce a coupling between two PDMPs
  following the same differential flow which implies quantitative
  bounds on the total variation between the marginal distributions of
  the two processes. Finally two results are established regarding the
  invariant measures of PDMPs. A practical condition to show that a
  probability measure is invariant for the associated PDMP
  semi-group is presented. In a second time, a bound on the invariant
  probability measures in $V$-norm of two PDMPs following the same
  differential flow is established. This last result is then applied
  to study the asymptotic bias of some non-exact PDMP MCMC methods.
\end{abstract}

\maketitle

\input{intro}

\input{cadre}
\input{constructions}

\input{superposition}

\input{explosion}
\input{synchrone}
\input{generator}
\input{regularity}

\input{invariant_1}
\input{invariant_2}

\section*{Acknowledgements}
Alain Durmus acknowledges support from Chaire BayeScale "P. Laffitte". Pierre Monmarch\'e acknowledges support from the French
ANR project ANR-12-JS01-0006 - PIECE. Arnaud Guillin and Pierre Monmarch\'e acknowledge support from the French
ANR-17-CE40-0030 - EFI - Entropy, flows, inequalities.
\bibliographystyle{plain}
\bibliography{../Bibliography/biblioBouncy}

\appendix
\input{core}

\end{document}

%% file: header.tex
 \usepackage[a4paper]{geometry}

\usepackage{amsthm,amsmath}
\allowdisplaybreaks
\usepackage[utf8]{inputenc}
\usepackage{amssymb}
\usepackage[english]{babel}
\usepackage{graphicx}
\usepackage{amsfonts,amssymb}
\usepackage{oldgerm}
\usepackage{mathrsfs}
\usepackage{pgf,tikz}
\usetikzlibrary{arrows}
\usepackage[active]{srcltx}
\usepackage{verbatim}
\usepackage{enumitem, kantlipsum} %label option for enumerate \alph* \Alph* \roman* \Roman* or other... [label=$\bullet$], [label=\alph*]usepackage{enumitem,}
\usepackage{aliascnt}
\usepackage{bbm}
\usepackage{array}
\usepackage{hyperref}
\hypersetup{
  colorlinks = true,
  citecolor = blue,
}
\usepackage[disable]{todonotes}
\usepackage{xargs}
\usepackage{cellspace}
\usepackage{xr}
\usepackage{textcomp}

\usepackage[Symbolsmallscale]{upgreek}
\usepackage{xcolor}
\definecolor{mydarkblue}{rgb}{0,0.08,0.45}
\usepackage{multicol} 
\usepackage{subcaption} 
\usepackage{wrapfig}

\usepackage{accents}
\usepackage{dsfont}
\usepackage{aliascnt}
\usepackage{cleveref}
\makeatletter
\newtheorem{theorem}{Theorem}
\crefname{theorem}{theorem}{Theorems}
\Crefname{Theorem}{Theorem}{Theorems}

\newaliascnt{lemma}{theorem}
\newtheorem{lemma}[lemma]{Lemma}
\aliascntresetthe{lemma}
\crefname{lemma}{lemma}{lemmas}
\Crefname{Lemma}{Lemma}{Lemmas}

\newaliascnt{corollary}{theorem}
\newtheorem{corollary}[corollary]{Corollary}
\aliascntresetthe{corollary}
\crefname{corollary}{corollary}{corollaries}
\Crefname{Corollary}{Corollary}{Corollaries}

\newaliascnt{proposition}{theorem}
\newtheorem{proposition}[proposition]{Proposition}
\aliascntresetthe{proposition}
\crefname{proposition}{proposition}{propositions}
\Crefname{Proposition}{Proposition}{Propositions}

\newaliascnt{definition}{theorem}
\newtheorem{definition}[definition]{Definition}
\aliascntresetthe{definition}
\crefname{definition}{definition}{definitions}
\Crefname{Definition}{Definition}{Definitions}

\newaliascnt{remark}{theorem}
\newtheorem{remark}[remark]{Remark}
\aliascntresetthe{remark}
\crefname{remark}{remark}{remarks}
\Crefname{Remark}{Remark}{Remarks}

\crefname{figure}{figure}{figures}
\Crefname{Figure}{Figure}{Figures}

\newtheorem{assumption}{\textbf{A}\hspace{-3pt}}
\Crefname{assumption}{\textbf{A}\hspace{-3pt}}{\textbf{A}\hspace{-3pt}}
\crefname{assumption}{\textbf{A}}{\textbf{A}}

\newtheorem{assumptionC}{\textbf{C}\hspace{-3pt}}
\Crefname{assumptionC}{\textbf{C}\hspace{-3pt}}{\textbf{C}\hspace{-3pt}}
\crefname{assumptionC}{\textbf{C}}{\textbf{C}}

\Crefname{assumptionG}{\textbf{G}\hspace{-3pt}}{\textbf{G}\hspace{-3pt}}
\crefname{assumptionG}{\textbf{G}}{\textbf{G}}

\Crefname{assumptionQ}{\textbf{Q}\hspace{-3pt}}{\textbf{Q}\hspace{-3pt}}
\crefname{assumptionQ}{\textbf{Q}}{\textbf{Q}}

\newcounter{hypA}
\theoremstyle{definition}
\newtheorem{construction}{Construction}
\Crefname{construction}{Construction}{Construction}
\crefname{construction}{Construction}{Construction}

\newaliascnt{example}{definition}
\newtheorem{example}[definition]{Example-Bouncy Particle Sampler}
\aliascntresetthe{example}
\crefname{example}{example}{examples}
\Crefname{Example}{Example}{Examples}

\newtheorem*{example*}{Example - Bouncy Particle Sampler}
\crefname{example-BPS}{example-BPS}{examples-BPS}
\Crefname{Example-BPS}{Example-BPS}{Examples-BPS}
\usepackage{stmaryrd}

%% file: def.tex
\definecolor{qqqqff}{rgb}{0,0,1}
\definecolor{wwwwff}{rgb}{0.4,0.4,1}
\definecolor{fftttt}{rgb}{1,0.2,0.2}

\def\mfg{\mathfrak{g}}

\def\msa{\mathsf{A}}
\def\msk{\mathsf{K}}
\def\mss{\mathsf{S}}

\def\msb{\mathsf{B}} 
\def\msc{\mathsf{C}}
\def\mse{\mathsf{E}}
\def\msl{\mathsf{L}}

\def\msm{\mathsf{M}}
\def\msu{\mathsf{U}}
\def\msv{\mathsf{V}}

\def\msy{\mathsf{Y}}
\def\msd{\mathsf{D}}

\def\mcbb{\mathcal{B}}  %%% \mcb est déjà pris
\newcommand{\mcb}[1]{\mathcal{B}(#1)}

\def\mcf{\mathcal{F}}
\def\mcg{\mathcal{G}}

\def\tmcf{\tilde{\mathcal{F}}}
\def\mcp{\mathcal{P}}

\def\rset{\mathbb{R}}

\def\nset{\mathbb{N}}
\def\nsets{\mathbb{N}^*}

\def\rmd{\mathrm{d}}

\def\mrd{\mathrm{d}}

\def\rme{\mathrm{e}}

\def\mrm{\mathrm{m}}
\def\mrc{\mathrm{C}}
\def\mrg{\mathrm{G}}
\def\mrgD{\mathrm{G(\ccint{0,T}\times \msm)}}
\def\mrl{\mathrm{L}}
\def\rmc{\mathrm{C}}

\def\rmt{\mathrm{T}}
\def\mrt{\mathrm{T}}

\newcommand{\cco}{\llbracket}
\newcommand{\ccf}{\rrbracket}
\newcommand{\po}{\left(}
\newcommand{\pf}{\right)}

\newcommand{\R}{\mathbb R}

\newcommand{\na}{\nabla}
\newcommand{\loiy}{\mu_{\mathrm{v}}}

\def\MeasFspace{\mathbb{M}}

\newcommandx{\functionspace}[2][1=+]{\mathbb{F}_{#1}(#2)}
\newcommandx{\VarDeux}[3][3=]{\operatorname{Var}^{#3}_{#1}\left\{#2 \right\}}

\newcommand{\1}{\mathbbm{1}}

\newcommand{\LeftEqNo}{\let\veqno\@@leqno}

\newcommand{\fraca}[2]{\left. \parentheseDeux{#1}  \middle / \parentheseDeux{#2} \right.}

\newcommand{\N}{\ensuremath{\mathbb{N}}}

\newcommand{\PE}{\mathbb{E}}
\newcommand{\PP}{\mathbb{P}}

\newcommand{\abs}[1]{\left\vert #1 \right\vert}
\newcommand{\absLigne}[1]{\vert #1 \vert}
\newcommand{\tvnorm}[1]{\| #1 \|_{\mathrm{TV}}}

\newcommandx{\Vnorm}[2][1=V]{\| #2 \|_{#1}}
\newcommandx{\VnormEq}[2][1=V]{\left\| #2 \right\|_{#1}}
\newcommandx{\norm}[2][1=]{\ifthenelse{\equal{#1}{}}{\left\Vert #2 \right\Vert}{\left\Vert #2 \right\Vert^{#1}}}
\newcommandx{\normLigne}[2][1=]{\ifthenelse{\equal{#1}{}}{\Vert #2 \Vert}{\Vert #2\Vert^{#1}}}
\newcommandx{\normL}[2][1=]{\ifthenelse{\equal{#1}{}}{\left\Vert #2 \right\Vert_{\msl}}{\left\Vert #2 \right\Vert^{#1}_{\msl}}}
\newcommandx{\normLLigne}[2][1=]{\ifthenelse{\equal{#1}{}}{\Vert #2 \Vert_{\msl}}{\Vert #2\Vert^{#1}_{\msl}}}

\newcommand{\parenthese}[1]{\left(#1 \right)}
\newcommand{\parentheseLigne}[1]{(#1 )}
\newcommand{\parentheseDeux}[1]{\left[ #1 \right]}
\newcommand{\defEns}[1]{\left\lbrace #1 \right\rbrace }
\newcommand{\defEnsLigne}[1]{\lbrace #1 \rbrace }

\newcommand{\ps}[2]{\left\langle#1,#2 \right\rangle}

\newcommand{\proba}[1]{\mathbb{P}\left( #1 \right)}
\newcommand{\probaLigne}[1]{\mathbb{P}( #1 )}
\newcommandx\probaMarkovTilde[2][2=]
{\ifthenelse{\equal{#2}{}}{{\widetilde{\mathbb{P}}_{#1}}}{\widetilde{\mathbb{P}}_{#1}\left[ #2\right]}}
\newcommand{\probaMarkov}[2]{\mathbb{P}_{#1}\left( #2\right)}
\newcommand{\expe}[1]{\PE \left[ #1 \right]}

\newcommand{\expeLigne}[1]{\PE [ #1 ]}

\newcommand{\expeMarkov}[2]{\PE_{#1} \left[ #2 \right]}

\newcommand{\plusinfty}{+\infty}

\def\ie{\textit{i.e.}}

\def\eqsp{\;}
\newcommand{\coint}[1]{\left[#1\right)}
\newcommand{\ocint}[1]{\left(#1\right]}
\newcommand{\ooint}[1]{\left(#1\right)}
\newcommand{\ccint}[1]{\left[#1\right]}
\renewcommand{\iint}[1]{\left\llbracket#1\right\rrbracket}
\newcommand{\iintLigne}[1]{\llbracket#1\rrbracket}
\newcommand{\cointLigne}[1]{[#1)}

\newcommand{\oointLigne}[1]{(#1)}
\newcommand{\ccintLigne}[1]{[#1]}

\newcommandx{\weight}[2][2=n]{\omega_{#1,#2}^N}

\newcommand{\ball}[2]{\operatorname{B}\left( #1,#2 \right)}
\newcommand{\cball}[2]{\overline{\operatorname{B}}\left( #1,#2 \right)}

 \newcommand{\tcr}[1]{{#1}}

\def\dist{\mathrm{dist}}

\newcommandx\sequence[3][2=,3=]
{\ifthenelse{\equal{#3}{}}{\ensuremath{\{ #1_{#2}\}}}{\ensuremath{\{ #1_{#2}, \eqsp #2 \in #3 \}}}}
\newcommandx\sequenceD[3][2=,3=]
{\ifthenelse{\equal{#3}{}}{\ensuremath{\{ #1_{#2}\}}}{\ensuremath{( #1)_{ #2 \in #3} }}}

\newcommandx{\sequencen}[2][2=n\in\N]{\ensuremath{\{ #1_n, \eqsp #2 \}}}
\newcommandx\sequenceDouble[4][3=,4=]
{\ifthenelse{\equal{#3}{}}{\ensuremath{\{ (#1_{#3},#2_{#3}) \}}}{\ensuremath{\{  (#1_{#3},#2_{#3}), \eqsp #3 \in #4 \}}}}
\newcommandx{\sequencenDouble}[3][3=n\in\N]{\ensuremath{\{ (#1_{n},#2_{n}), \eqsp #3 \}}}

\def\iid{i.i.d.}

\def\eg{e.g.}

\newcommand{\opnorm}[1]{{\left\vert\kern-0.25ex\left\vert\kern-0.25ex\left\vert #1 
    \right\vert\kern-0.25ex\right\vert\kern-0.25ex\right\vert}}

\def\generator{\mathcal{A}}
\def\sgenerator{\bar{\mathcal{A}}}

\def\bfe{\mathbf{e}}

\def\Id{\operatorname{Id}}

\newcommandx{\CPE}[3][1=]{{\mathbb E}_{#1}\left[#2 \left \vert #3 \right. \right]} %%%% esperance conditionnelle
\newcommandx{\CPELigne}[3][1=]{{\mathbb E}_{#1}[#2  \, \vert \,  #3  ]} %%%% esperance conditionnelle
\newcommandx{\CPVar}[3][1=]{\mathrm{Var}^{#3}_{#1}\left\{ #2 \right\}}
\newcommand{\CPP}[3][]
{\ifthenelse{\equal{#1}{}}{{\mathbb P}\left(\left. #2 \, \right| #3 \right)}{{\mathbb P}_{#1}\left(\left. #2 \, \right | #3 \right)}}
\newcommand{\CPPLigne}[3][]
{\ifthenelse{\equal{#1}{}}{{\mathbb P}( #2 \, | #3 )}{{\mathbb P}_{#1}( #2 \,  | #3 )}}

\newcommandx{\osc}[2][1=]{\mathrm{osc}_{#1}(#2)}

\def\Id{\operatorname{Id}}

\def\domain{\mathrm{D}}

\def\martfg{M^{f,g}}
\newcommand\Ddir[1]{D_{#1}}

\def\Refl{\mathrm{R}}

\def\transpose{\operatorname{T}}

\def\w{w}

\def\bX{\bar{X}}
\def\bY{\bar{Y}}

\def\bE{\bar{E}}

\def\S{S}

\def\tI{\tilde{I}}
\def\tmsk{\tilde{\msk}}

\def\yt{\tilde{y}}
\def\ty{\yt}

\def\tT{\tilde{T}}
\def\tS{\tilde{S}}

\def\bX{\bar{X}}
\def\bY{\bar{Y}}

\def\bE{\bar{E}}

\def\bS{\bar{S}}

\def\bN{\bar{N}}

\def\veps{\varepsilon}
\def\sphere{\mss}

\def\rate{\lambda_{\mathrm{c}}}

\newcommand{\ensemble}[2]{\left\{#1\,:\eqsp #2\right\}}
\newcommand{\ensembleLigne}[2]{\{#1\,:\eqsp #2\}}

\def\rmD{\mathrm{D}}%%rmd déjà pris
\def\mrD{\rmD}
\def\mrd{\mathrm{D}}
\def\mrb{\mathrm{B}}
\def\rmb{\mrb}
\def\mrc{\mathrm{C}}

\newcommand{\supp}{\mathrm{supp}}

\newcommand\coupling[2]{\Gamma(\mu,\nu)}
\def\supp{\mathrm{supp}}
\def\tpi{\tilde{\pi}}
\newcommand\adh[1]{\overline{#1}}

\def\opK{\mathrm{K}}

\newcommand{\comp}{\mathrm{c}}
\newcommand{\compl}{\mathrm{c}}

\renewcommand{\geq}{\geqslant}
\renewcommand{\leq}{\leqslant}

\def\bfG{\mathbf{G}}

\def\tH{\tilde{H}}
\def\tN{\tilde{N}}
\def\tvarphi{\tilde{\varphi}}

\def\diff{\nabla}
\def\diffD{\mathrm{D}}
\def\rv{random variable}
\def\rvs{random variables}
\def\PDMP{\mathrm{PDMP}}
\def\ttr{r}

\def\generatorT{\mathcal{L}}
\def\range{\mathrm{Ran}}
\newcommand\restr[2]{{% we make the whole thing an ordinary symbol
  \left.\kern-\nulldelimiterspace % automatically resize the bar with \right
  #1 % the function
  \vphantom{\big|} % pretend it's a little taller at normal size
  \right|_{#2} % this is the delimiter
  }}

\def\closure{\mathrm{Cl}}

%% file: intro.tex
\section{Introduction}
Piecewise Deterministic Markov Processes (PDMP), similarly to
diffusion processes, form an important class of Markov processes,
which are used to model random dynamical systems in various fields
(see \eg~\cite{MalrieuPDMP,ABGKZ}). Recently, interest has grown for their use to sample from a target distribution \cite{BierkensFearnheadRoberts,PetersdeWith,Doucet2015}. The resulting class of algorithms is referred to as PDMP Monte Carlo (PDMP-MC) methods. 
%Recently, interest has grown for
%their use in MCMC algorithms.
 To this end,
natural questions arise as to the stationarity of the target measure,
the ergodicity of the corresponding process and possible bias
introduced by the method. In mathematical physics \cite{BDMM2017} and
biology \cite{Calvez}, the long time behaviour of these processes has
been the subject of several works. In this context, these studies are
done through the Kolmogorov Fokker Planck operator $\generator^{\star}$ of the PDMP of
interest given for all smooth density $\rho$ on $\rset^{2d}$ by
\begin{equation*}
  \generator^{\star} \rho = -\ps{\Xi}{\nabla \rho} + K ( \lambda \rho) - \lambda \rho \eqsp,
\end{equation*}
where $\Xi$ is a smooth vector field of $\rset^{2d}$, $\lambda : \rset^{2d} \to \rset_+$ and $K$ is a non-local collision operator.

The relevance of the present work emerged while writing the companion
paper \cite{DurmusGuillinMonmarche:bouncy}, concerned with the
geometric ergodicity of the Bouncy Particle Sampler (BPS)
\cite{Doucet2015}, an MCMC algorithm which, given a target distribution
$\pi$ on $\rset^d$, introduces a PDMP for which $\pi$ is invariant.
In order to make rigorous several arguments in
\cite{DurmusGuillinMonmarche:bouncy}, technical lemmas had to be
established, in particular to cope with the fact that Markov
semi-groups associated to PDMP lack the regularity properties of
(hypo-)elliptic diffusions, which yields additional difficulties and technicalities. These results, of interest in a more general
framework, are gathered here with the hope that it will set a framework where for example verification of the invariance of a measure becomes a mere calculus via the generator (as it is the case for diffusion processes under mild assumptions). To illustrate our results,  BPS is used as a recurrent
example.

Let us present these different results, together with the organization of the paper. \Cref{sec:framework} contains the basic definitions of our framework, and in particular presents the construction of a PDMP. Alternative constructions are shown in \Cref{sec:alternative-constructions,sec:superposition} to give the same process (i.e.~to give a random variable with the same law on the Skorokhod space). Conditions which ensure that PDMPs are non explosive are presented in \Cref{sec:non-explose}. The synchronous coupling of two PDMPs is defined in \Cref{sec:synchrone}, which aims  to construct simultaneously two different PDMPs, starting at the same initial state, in such a way that they have some probability to stay equal for some time. It yields estimates on the difference of the corresponding semi-groups in total variation norm. In \Cref{sec:stab_c1c}, conditions are established under which the semi-group associated to a PDMP leaves invariant the space of compactly-supported smooth functions. % These conditions are quite restrictive but, thanks to the estimates given by the synchronous coupling, it is often sufficient to suppose that they are satisfied by a smooth approximation of the PDMP.
Using this result, a practical criterion to ensure that a given probability measure $\mu$ is invariant for a PDMP is obtained in Section~\ref{eq:invariant_main}. Indeed, it is well-known that, denoting by $(\sgenerator,\rmD(\sgenerator))$ the strong generator of the Markov semi-group associated to the PDMP, then $\mu$ is invariant if and only if $\int \sgenerator f \rmd \mu = 0$ for all $f$ in a core of $\sgenerator$. Nevertheless, due to the lack of regularization properties of the semi-group, it is generally impossible to determine such a core. We will prove that, under some simple assumptions, it is enough to consider compactly-supported smooth functions $f$. Finally, in \Cref{sec:invariant_2}, we are interested in bounding the $V$-norm between two invariant probability measures $\mu_1$ and $\mu_2$ corresponding to two PDMPs sharing the same differential flow but with different jump rates and Markov kernels.  This question is here mainly motivated by the thinning method used to sample trajectories of PDMPs \cite{lewis:shedler:1979,Thieullen2016}. Indeed, a PDMP can be exactly sampled (in the sense that no time discretization is needed) provided that the associated differential flow can be computed and a simple upper bound on the jump rate is known. When this is not the case, a PDMP with a truncated jump rate can be sampled, and our result gives a control on the ensuing error. %In an MCMC algorithm, it could be more efficient to add some bias by sampling the PDMP with a truncated jump rate, if this process is cheaper to compute so that, at a given numerical cost, it can be sampled a longer time (or with more replicas) so that the variance is reduced.

%%% Local Variables:
%%% mode: latex
%%% TeX-master: "main"
%%% End:

%% file: cadre.tex
\subsection*{Notation and conventions}
For all $a,b \in \rset$, we denote $a_+= \max(0,a)$, $a\vee b =
\max(a,b)$, $a \wedge b = \min(a,b)$. $\Id$ stands for the identity
matrix on $\rset^d$. 

For all $x,y \in \rset^d$, the scalar product between $x$ and $y$ is
denoted by $\ps{x}{y}$ and the Euclidean norm of $x$ by $\norm{x}$.
%We denote,  by %$\sphere^d = \ensemble{\v \in \rset^d}{\norm{\v} = 1}$,
%the $d$-dimensional sphere with radius $1$ and
For all $x \in \rset^d$, $r >0$, we denote  by
$\ball{x}{r}=\ensemble{\w \in \rset^d}{\norm{w-x}
  \tcr{<} r}$ the ball centered at $x$ with radius $r$. The closed ball  centered in $x$ with radius $r$ is denoted by $\cball{x}{r}$. For any
$d$-dimensional matrix $M$, define by $\norm{M} = \sup_{\w \in
  \ball{0}{1}}  \norm{M \w} $ the operator norm associated with
$M$.

%Let $f : \rset^d \to \rset^m$ be a continuously differentiable function. The Jacobian matrix of $f$ is denoted by $\jac(f)$.

Let $(\msm,\mfg)$ be a smooth closed Riemannian sub-manifold of $\rset^N$
and $\mcbb(\msm)$ the associated Borel $\sigma$-field. Let $\infty \notin M$ be a cemetery point. The distance induced by $\mfg$ is denoted by $\dist$. With a slight abuse of notations,  the ball (respectively closed ball)  centered at $x \in \msm$ with radius $r >0$ is denoted by $\ball{x}{r}$ (respectively $\cball{x}{r}$).

For all function $F : \msm \to \rset^m$ and compact set
$\msk \subset \msm$, denote
$\norm{F}_{\infty} = \sup_{x \in \msm} \norm{F(x)}$,
$\norm{F}_{\infty, \msk} = \sup_{x \in \msk} \norm{F(x)}$.  Denote by $\MeasFspace(\msm)$ the space of measurable functions from $\msm$ to $\rset$. Denote by
$\mrb(\msm)$ the set of all measurable and bounded functions from
$\msm$ to $\rset$. The space $\mrb(\msm)$ is endowed with the topology
associated with the uniform norm $\norm{\cdot}_{\infty}$.  Let
$\mrc(\msm)$ stand for the set of continuous functions from $\msm$ to
$\rset$, $\rmc_0(\msm)$ the subset of $\rmc(\msm)$ consisting of
continuous functions vanishing at infinity and, for all
$k \in \nsets$, let $\mrc^k(\msm)$ be the set of $k$-times
continuously differentiable functions from $\msm$ to $\rset$. For all $k \in \nset$, denote
 by $\mrc^{k}_c(\msm)$ and $\mrc^{k}_b(\msm)$ the
set of functions in $\mrc^k(\msm)$ with compact support and the set of
bounded functions in $\mrc^k(\msm)$ respectively. For
$f \in \mrc^k(\msm)$, we denote by $\diffD^kf$ the $k^{\text{th}}$
differential of $f$. For all function $f : \msm \to \rset$, we denote
by $\nabla f$ and $\nabla^2 f$ the gradient and the Hessian of $f$
respectively, if they exist.

We denote by $\mcp(\msm)$ the set of probability
measures on $\msm$. For $\mu,\nu \in \mcp(\msm)$,
$\xi \in \mcp(\msm^{2})$ is called a transference plan between $\mu$
and $\nu$ if for all $\msa \in \mcb{\msm}$,
$\xi( \msa \times \msm) = \mu(\msa)$ and
$\xi(\msm \times \msa) = \nu(\msa)$. The set of transference plan
between $\mu$ and $\nu$ is denoted by $\coupling{\mu}{\nu}$.  The random
variables $X$ and $Y$ on $\msm$ are a coupling between $\mu$ and $\nu$
if the distribution of $(X,Y)$ belongs to $ \coupling{\mu}{\nu}$.  The
total variation norm between $\mu$ and $\nu$ is defined by
\begin{equation*}
  \tvnorm{\mu-\nu} = 2 \inf_{\xi \in \coupling{\mu}{\nu}} \int_{\msm^2} \1_{\tcr{ \Delta_\msm^c}} (x,y) \, \rmd \xi(x,y) \eqsp, 
                     % & = \sup \ensemble{\abs{\int_{\rset^d} f \rmd \mu - \int_{\msm} f\rmd \nu }} {\text{$f :\msm \to \rset$ is Borel measurable and $\sup_{\msm} f < 1$}} \eqsp.
\end{equation*}
where $\Delta_\msm = \ensemble{(x,y) \in \msm^2}{x=y}$ \tcr{and $\Delta_\msm^c = \msm^2 \setminus \Delta_\msm$ is its complement}. 
For all $\mu \in \mcp(\msm)$, define the support of $\mu$ by
\begin{equation*}
  \supp \, \mu = \adh{\ensemble{x \in \msm}{ \text{ for all open set } \msu \ni x,\,  \mu(\msu) >0}} \eqsp.
\end{equation*}

In the sequel, we take the convention that $\inf \emptyset = \plusinfty$. All the random variables
considered in this paper are defined on a fixed probability space
$(\Omega,\mcf,\mathbb{P})$.

%%%%%%%%%%%%
%%%%%%%%%%%%

\section{A first definition of Piecewise Deterministic Markov Processes}
\label{sec:framework}

\subsection*{Definitions and further notation}$~$\\
Let $(\msm,\mfg)$ be a smooth closed Riemannian sub-manifold of $\rset^N$. A PDMP on $\msm$ is defined using a triple \tcr{$(\varphi, (\lambda_i,Q_i)_{i\in\iint{1,\ell}})$}, $\ell \in \nset^*$, referred to as the local characteristics of a PDMP, where 
\begin{itemize}
\item $\varphi $ is a differential flow on $\msm$:
  $\varphi:(t,h,x)\mapsto \varphi_{t,t+h}(x)$ is a measurable function
  from $\rset_+ \times \rset_+\times \msm$ to $\msm$, such that for
  all $t,h_1,h_2\geqslant 0$
  $\varphi_{t+h_1,t+h_1+h_2} \circ \varphi_{t,t+h_1} =
  \varphi_{t,t+h_1+h_2}$, $\varphi_{t,t} = \Id$. Moreover, for all
  $(t,x)\in\rset_+\times \msm$,
  $\tilde h\mapsto \varphi_{t,t+\tilde h}(x)$ is continuously
  differentiable from $\rset_+$ to $\msm$ and for all
  $t,h \in \rset_+$, $\tilde x \mapsto \varphi_{t,t+h}(\tilde x)$ is a
  $\rmc^1$-diffeomorphism of $\msm$. The flow $\varphi$ is
  (time-)homogeneous if for all $t,h \in \rset_+$,
  $\varphi_{t,t+h} = \varphi_{0,h}$, in which case we set
  $\varphi_h = \varphi_{0,h}$.
\item For all $i \in \iint{1,\ell}$, $\lambda_i: \rset_+ \times \msm \to \rset_+$ is a measurable function referred to as a jump rate on $\msm$ which is locally bounded, in the sense that $\norm{\lambda_i}_{\infty,\msk}<\infty$ for all compact $\msk\subset \rset_+\times\msm$. The jump rate $\lambda_i$ is (time)-homogeneous if it does
  not depend on $t$.
\item For all $i \in \iint{1,\ell}$, $Q_i : \rset_+ \times \msm \times \mcb{\msm} \to \ccint{0,1}$ is an inhomogeneous  Markov kernel on $\msm$: for all $\msa \in \mcb{\msm}$, $(t,x) \mapsto Q_i(t,x,\msa)$ is measurable, and for all $(t,x)\in \rset_+\times \msm$, $Q_i(t,x,\cdot) \in \mcp(\msm)$. The Markov kernel $Q_i$ is (time-)homogeneous if it does
  not depend on $t$.
% \item A representation $\bfG$ of a Markov kernel $Q$ is a measurable function $\bfG:t,x,u \mapsto \bfG(t,x,u)$ from $\rset_+\times \msm\times [0,1]$ to $\msm$ such that for all $(t,x,\msa) \in \rset_+\times \msm\times \mcb{\msm}$, $Q(t,x,\msa) = \mathbb P (\bfG(t,x,U) \in \msa)$, where $U$ is a random variable (r.v.) uniformly distributed on $[0,1]$. By \cite[Corollary 7.16.1]{bertskekas:shreve:1978}, such a representation always exists.
% \item A (homogeneous) jump mechanism on $\msm$ is a pair $(\lambda,Q)$ constituted of a (homogeneous) jump rate and a (homogeneous) Markov kernel on $\msm$.
% \item The (homogeneous) local characteristics of a PDMP, $( \varphi,(\lambda_i,Q_i)_{i\in \iintLigne{1,\ell}})$  on $\msm$ is constituted of a (homogeneous) flow $\varphi$ and of $\ell\in\nsets$ (homogeneous) jump mechanisms on $\msm$. When $\ell=1$ we simply write $(\varphi,\lambda,Q)$.
%\item A PDMP triplet $(\varphi,\lambda,Q)$  is constituted of a flow $\varphi$, a jump rate $\lambda$ and a jump kernel $Q$.
%$( \varphi,(\lambda^{(i)})_{i\in \iintLigne{1,n}},(Q^{(i)})_{i\in \iintLigne{1,n}})$ for some $n\geqslant 1$ is constituted of a flow $\varphi$, a family of $n$ jump rates $(\lambda^{(i)})_{i\in \iintLigne{1,n}}$ and a family of $n$ jump kernel $(Q^{(i)})_{i\in \iintLigne{1,n}}$. When $n=1$, we simply denote $(\varphi,\lambda,Q)$.
\end{itemize}

\tcr{In the case $\ell = 1$, the local characteristics $(\varphi, (\lambda_1,Q_1)_{i\in\iint{1}1})$ are  denoted by $(\varphi, \lambda_1, Q_1)$.}

If $\varphi$ is a homogeneous differential flow and, for all
$i \in \iint{1,\ell}$, $\lambda_i,Q_i$ are homogeneous as well, the
local characteristics $(\varphi, (\lambda_i,Q_i)_{i\in\iint{1,\ell}})$
are said to be homogeneous. A (homogeneous) jump mechanism on $\msm$ is
a pair $(\lambda,Q)$ constituted of a (homogeneous) jump rate and a
(homogeneous) Markov kernel on $\msm$. % We consider the following

\subsection*{A first construction of a PDMP}%$\mbox{PDMP}\left(\varphi,(\lambda_i,Q_i)_{i\in \iintLigne{1,\ell}}, \mu_0\right)$}$~$\\
For all $i \in \iint{1,\ell}$, consider a representation $\bfG_i$ of the Markov kernel $Q_i$, \ie~a measurable function $\bfG_i:t,x,u \mapsto \bfG(t,x,u)$ from $\rset_+\times \msm\times [0,1]$ to $\msm$ such that for all $(t,x,\msa) \in \rset_+\times \msm\times \mcb{\msm}$, $Q_i(t,x,\msa) = \mathbb P (\bfG_i(t,x,U) \in \msa)$, where $U$ is a random variable uniformly distributed on $[0,1]$. By \cite[Corollary 7.16.1]{bertskekas:shreve:1978}, such a representation always exists. 

Then, a PDMP $(X_t)_{t \geq 0}$ based on the local characteristics
$(\varphi,(\lambda_i,Q_i)_{i \in\iint{1,\ell}})$ and the initial
distribution $\mu_0$ can be defined
recursively through a Markov chain $(X_{k}',S_k)_{k\in \mathbb N}$ on
$(\msm\cup\{\infty\})\times (\rset_+\cup\{+\infty\})$. For all
$k \in \nset$, $X'_k$ will be the state of the process
$(X_t)_{t \geq 0}$ at times $S_k$. Between two times $S_k$ and
$S_{k+1}$, $(X_t)_{t \geq 0}$ will be a deterministic function of $X'_k$ and $S_k$.
More precisely, consider the following construction. 

\noindent\hfil\rule{\textwidth}{.4pt}\hfil
\begin{construction}
\label{const:1}
Let $W_0$ be a \rv~with distribution $\mu_0 \in \mcp(\msm)$ and $((E_{j,k})_{j\in\iintLigne{1,\ell}},U_k)_{k\in \nsets}$ be  an i.i.d. sequence, independent of $W_0$, such that for all $k\in \nsets$ and $j\in \iintLigne{1,\ell}$, $U_k$ is uniformly distributed on $[0,1]$ and $E_{j,k}$ is an exponential \rv~with parameter $1$, independent of $U_k$ and from $E_{i,k}$ for $i\neq j$. Recall that $\infty \notin M$ is a cemetery point.\\
  Set $S_0 = 0$, $X_0' = W_0$, $I_0 = 1$ and
suppose that $(X_{k}',S_k,I_k)$ and $(X_t)_{t \leq S_k}$ have been defined for some
$k\in \mathbb N$, with $X_{k}'\in\msm$ and $S_k\in \rset_+$. For all
$j \in \iintLigne{1,\ell}$, set
\begin{equation}
  \label{eq:def_s_j_const_1}
  S_{j,k+1}  =  \inf \ensemble{ t\geq S_k }{\ E_{j,k+1} < \int_{S_k}^t \lambda_j\po s, \varphi_{S_k,s}(X_{k}')\pf \rmd s }\eqsp, \, 
S_{k+1}  =  \underset{j\in\iintLigne{1,\ell}}\min S_{j,k+1}\eqsp.
\end{equation}
\begin{itemize}
\item If $S_{k+1} = +\infty$, set $S_m=+\infty$, $X_{m}' = \infty$, $I_m = 1$ for all $m>k$, and $X_t =  \varphi_{S_k,t}(X'_k)$ for all $t \geq S_k$. 
\item  If $S_{k+1} < +\infty$, set  
\begin{equation*}
I_{k+1}  = \min \{j\in \iintLigne{1,\ell},\ S_{j,k+1} = S_{k+1}\}\eqsp, \, X_{k+1}'  =  \bfG_{I_{k+1}}\po S_{k+1}, \varphi_{S_k,S_{k+1}}(X_{k}'),U_{k+1}\pf\eqsp.
\end{equation*}
\tcr{For $t \in \ooint{S_k,S_{k+1}}$, set $X_t =  \varphi_{S_k,t}(X'_k)$ and $X_{S_{k+1}} =X'_{k+1}$.}
\end{itemize}  
For $t \geq \sup_{k \in \nset} S_k$, set $X_t = \infty$.
\end{construction}
\noindent\hfil\rule{\textwidth}{.4pt}\hfil
\bigskip

Define for any $j \in\iint{1,\ell}$, given $X_k, S_k$,
$\psi_{j,k}(t) = \int_{S_k}^{t+S_k} \lambda_j
(s,\varphi_{S_k,s}(X_k')) \rmd s$ and $\psi_{j,k}^{\leftarrow}$ its generalized inverse which is defined since $\psi_{j,k}$ is non-decreasing.  Note that by \Cref{const:1},  for $k \in \nset$,
$i,j \in \iintLigne{1,\ell}$, $i\neq j$, using properties of the generalized inverse, since $\psi_{j,k}$ is continuous, we obtain 
 \begin{equation}
   \label{eq:event_eqal_S_k_zero}
   \proba{S_{j,k+1} = S_{i,k+1} = S_{k+1} < +\infty \, | \, X_k, S_k}=   \proba{E_{j,k+1} = \psi_{j,k}\circ \psi_{i,k}^{\leftarrow} ( E_{i,k+1}) \, , S_{k+1} < +\infty} =0 \eqsp.
 \end{equation}
However, the definition of
$I_{k+1}$ ensures that the process $(X_{k}',S_k)$ is defined not only
almost everywhere on $\Omega$, but in fact on all $\Omega$.

Let $(\mcf_{k}')_{k \in \nset}$ be the filtration associated with
$(X'_k,S_k,I_k)_{k \in \nset}$. Then, $(X'_k,S_k,I_k)_{k \in \nset}$ is an inhomogeneous
Markov chains since for all $k \in \nset$, $\msa \in \mcb{\msm}$,
$t \geq S_k$, $j\in \iint{1,\ell}$,
\begin{align}
  \nonumber
  \CPP{X'_{k+1}\in \msa, S_{k+1} \leq t, I_{k+1} = j}{\tcr{\mcf_k'}}{} 
  &= \1_{\msm}(X'_k) \int_{S_k}^{t}  Q_j(s,\varphi_{S_k,s}(X_k'),\msa)  \lambda_j(s,\varphi_{S_k,s}(X_k')) \\
    \label{eq:carac_chaine_markov_1}
 & \qquad \times  \exp\defEns{-\sum_{i=1}^\ell \int_{S_k}^s   \lambda_i(u,\varphi_{S_k,u}(X_k')) \rmd u} \rmd s \eqsp.
\end{align}
Note that the sequence $(X'_k,S_k)_{k \in \nset}$ is an
inhomogeneous Markov chain  as well, whose   kernel can be straightforwardly
deduced from \eqref{eq:carac_chaine_markov_1}.

 Then, $(X_t)_{t\geqslant 0}$ is a stochastic process on $\msm\cup\{\infty\}$, \ie~it is a \rv~from $(\Omega,\mcf,\mathbb{P})$ to the space
$\mrd(\rset_+,\msm \cup \{\infty\})$ of càdlàg functions from $\R_+$ to $\msm \cup \{\infty\}$, endowed with the Skorokhod
topology, see \cite[Chapter 6]{jacod:shiryaev:2003}. Moreover, $(X_t)_{t \geq 0}$ is a Markov process \cite[Theorem 7.3.1]{jacobsen:2006}, from the class of piecewise deterministic Markov processes (PDMPs). We say that a stochastic process $(\tilde{X}_t)_{t \geq 0}$ is a PDMP with local characteristics $( \varphi,(\lambda_i,Q_i)_{i\in \iintLigne{1,n}})$ and initial distribution $\mu_0$ if it has the same distribution on $\mrd(\rset_+,\msm \cup \{\infty\})$  as $(X_t)_{t \geq 0}$. We will denote by
$\mbox{PDMP}\left(\varphi,(\lambda_i,Q_i)_{i\in \iintLigne{1,\ell}}, \mu_0\right)$
this distribution.
In the sequel, we will see that a given  PDMP can admit several local characteristics. 
Note that, as $\varphi$ is a $\mbox{C}^1$-diffeomorphism,  $(X_t)_{t \geq 0}$ is completely determined by the Markov chain $(X'_k,S_k)_{k \in \nset}$, referred to as the embedded chain associated to the process. The sequence $(S_k)_{k \in \nset}$ is said to be the jump times of the process $(X_t)_{t \geq 0}$.
 A PDMP is said to be homogeneous if its local characteristics are (time) homogeneous. 

For  $(\tilde{X}_t)_{t \geq 0}\in \mrd(\rset_+,\msm \cup \{\infty\})$, we call $\tau_\infty(\tilde{X}) = \inf\{ t\geqslant 0\ :\ \tilde{X}_t = \infty\}$ the explosion time of the process $(\tilde{X}_t)_{t \geq 0}$. A process $(\tilde{X}_t)_{t \geq 0}$ is said to be non-explosive if $\tau_\infty(\tilde{X}) =+\infty$ almost surely. PDMP characteristics are said to be non-explosive if for all initial distribution the associated PDMP is non-explosive.

\Cref{const:1} associated with the characteristics
$(\varphi,(\lambda_i,Q_i)_{i \in \iint{1,\ell}}$ defines a Markov
semi-group $(P_{s,t})_{t\geqslant s\geqslant 0}$ for all $x\in \msm$,
$\msa\in \mcb{\msm}$ and $t\geqslant s \geq  0$ by
\begin{equation*}
%  \label{eq:def_semi_group}
  P_{s,t}(x,\msa)\ = \ \mathbb P( \bar{X}_{t-s}^{s,x} \in \msa) \eqsp,
\end{equation*}
where $(\bar{X}_u^{s,x})_{u \geq 0}$ is a PDMP with characteristics $((\varphi_{s+u,s+t})_{t\geq u \geq 0},(\lambda_i(s+ \cdot,\cdot),Q_i(s+\cdot,\cdot,\cdot))_{i \in \iint{1,\ell}})$ starting from $x$.
Its left-action on $\mrc(\msm)$ and right-action on $\mcp(\msm)$ are then given by
\[P_{s,t} f(x) = \expe{f(\bar{X}_t^{x,s})}\eqsp,\quad \nu P_{s,t}(\msa) = \int_{\msm} \PP(\bar{X}_t^{x,s}\in\msa) \rmd \nu (x) \eqsp,\]
 for all $f\in \mrb(\msm)$, $x\in \msm$, $\nu\in \mcp(\msm)$, $\msa\in  \mcb{\msm}$ and $t\geqslant s\geqslant 0$. The Markov property of $(X_t)_{t\geqslant 0}$ is equivalent to the semi-group property $P_{u,s}P_{s,t}=P_{u,t}$ for all $t\geqslant s\geqslant u \geqslant 0$. If $(\varphi,(\lambda_i,Q_i)_{i \iint{1,\ell}})$ is non explosive, then $P_{s,t}$ is a Markov kernel for all $t\geqslant s\geqslant 0$ and we say that $(P_{s,t})_{t \geq s \geq 0}$ is non explosive.  Otherwise, it is a sub-Markovian kernel  only. For a homogeneous process, we simply write $P_t = P_{0,t}$ for all $t \geq 0$.

 For a PDMP $(X_t)_{t\geqslant 0}$ with jump times
 $(S_k)_{k\in\mathbb N}$, we say that at time $S_{k+1}$, $k \geq 0$, a true jump
 occurred if $X_{S_{k+1}} \neq
 \varphi_{S_{k},S_{k+1}}(X_{S_k})$. Else, we say that at time $S_{k+1}$ a
 fantom jumped occurred. Note that, in the definition of homogeneous
 PDMPs with characteristics $(\varphi,\lambda,Q)$ given in
 \cite[standard conditions p. 62]{davis:1993}, fantom jumps are
 impossible, since it is assumed that for all $x \in \msm$,
 $Q(x,\{x\})=0$. This is not the case with the definition we gave in
 \Cref{sec:framework}, where the notion of jump times depends on the
 jump mechanisms used to define the process. We will see that in
 \Cref{sec:superposition} that under our settings, based on
 characteristics $(\varphi,(\lambda_i,Q_i)_{i \in \iint{1,\ell}})$
 which define a PDMP $(X_t)_{t \geq 0}$, we can always define some
 characteristics $(\varphi,\lambda,Q)$ which define a PDMP
 $(Z_t)_{t \geq 0}$ with the same distribution as $(X_t)_{t \geq 0}$ but no fantom jump.

 The condition imposed by
 \cite{davis:1993} implying that a PDMP has no fantom jump can be
 very useful since it allows a one-to-one correspondence between the path of the continuous-time process $(X_t)_{t \geq 0}$ and of its embedded chain $(X_k',S_k)_{k \in \nset}$. With our construction, the continuous process is completely determined by its embedded chain but not the opposite.  

 On the other hand, adding fantom jumps sometimes turns out to be convenient. Here is an example: let $(\varphi,\lambda,Q)$ be the characteristics of a PDMP $(X_t)_{t\geqslant 0}$, and suppose that there exists $\lambda_*>0$ such that $\lambda(t,x)\leqslant \lambda_*$ for all $t\geqslant 0$ and $x\in\msm$. From \Cref{eq:prop-superposition} below,  $(X_t)_{t\geqslant 0}$ has the same distribution as the PDMP $(Z_t)_{t\geqslant 0}$ obtained through \Cref{const:1} from the characteristics $(\varphi,\lambda_*,\tilde Q)$ with for all $(t,x,\msa) \in \rset_+ \times \msm \times \mcb{\msm}$,
\[\tilde Q(t,x,\msa) \ = \ \frac{\lambda(t,x)}{\lambda_*} Q(t,x,\msa) + \defEns{1 - \frac{\lambda(t,x)}{\lambda_*}} \updelta_x(\msa)\eqsp.\]
The jump times of $(Z_t)_{t\geqslant 0}$ are given by a Poisson
process with intensity $\lambda_*$. Adding fantom jumps to PDMPs
turns our to be very useful in practice for their simulation. This method is referred to as thinning (see
\cite{Thieullen2016} and references therein for more
details). \tcr{Note that fantom jumps and the link between the process
  and the embedded chain are also investigated for a different type of
  PDMPs in \cite{BLBMZ3}.}

 Another use of fantom jump is presented in \cite{Costa}. The stability or ergodicity of a PDMP  $(X_t)_{t\geqslant0}$ and of its embedded chain $(X_k',T_k)_{k\in\mathbb N}$ may differ, but this is no more the case if fantom jumps are added at constant rate, \ie~if we consider the PDMP with  characteristics $(\varphi,\lambda+1,\tilde Q)$, where $\tilde{Q}$ is given for all $(t,x,\msa) \in \rset_+ \times \msm \times \mcb{\msm}$ by
 \[\tilde Q(t,x,\msa) \ = \ \frac{\lambda(t,x)}{1+\lambda(t,x)} Q(t,x,\msa) + \frac{1}{1+\lambda(t,x)} \updelta_x(\msa)\]
 and its embedded chain. See \cite{Costa} for more details.

 There are other differences between the assumptions we consider on
 the characteristics of a PDMP and those in \cite[standard conditions
 p. 62]{davis:1993}. First, in contrast to \cite{davis:1993}, we
 consider that the differential flow $(\varphi_{t})_{t \geq 0}$ is
 defined on $\rset_+$. In addition, the jump rates
 $(\lambda_i)_{\i \in \iint{1,\ell}}$ are supposed to be locally
 bounded to prevent some artificial pathological behavior such as an
 infinite number of fantom jumps in a finite time. However, the following weaker condition would have been 
 sufficient to define $(X_t)_{t\geqslant0}$: for all
 $(t,x) \in \rset_+\times \msm$, there exists $h >0$ such that
 $\int_{t}^{t+h} \lambda(s,\varphi_{t,s}(x)) \rmd s < \plusinfty$.  On
 the other hand, we don't assume a priori that PDMPs are
 non-explosive.
 
 \subsection*{Examples}$~$\\
 Several examples of PDMP can be found in \cite{MalrieuPDMP} and
 references therein. In the present paper, a particular attention will be
 paid to the family of velocity jump PDMP, described as follows. Let
 $\msv \subset \rset^d$ be a smooth complete Riemannian submanifold,
 and set $\msm = \rset^d\times \msv$. Then, $\msm$ is a smooth
 complete Riemannian submanifold of $\rset^{2d}$ endowed with the
 canonical Euclidean distance and tensor metric. We say that a PDMP $(X_t,Y_t)_{t\geqslant0}$ on $\msm$
 (where $X_t \in \rset^d$ and $Y_t\in \msv$ for all $t\geqslant 0$)
 with characteristics
 $( \varphi,(\lambda_i,Q_i)_{i\in \iintLigne{1,\ell}})$ is a velocity
 jump PDMP if $\varphi$ is homogeneous and given for any
 $t \in \rset_+$ and $(x,v) \in \rset_+ \times \msv$ by
 \begin{equation}
   \label{eq:transport_libre}
   \varphi_t(x,y)= (x+ty,y) 
 \end{equation}
 and if for all $i\in \iintLigne{1,\ell}$, all $\msa\in  \mcb{\rset^d}$ and all $(t,x,y)\in\rset_+\times\rset^d\times \msv$,
 \[Q_i(t,(x,y),\msa\times\msv)\ = \ \updelta_x(\msa)\eqsp.\]

 Consider the PDMP $(X_t,Y_t)_{t \geq 0}$ associated with this choice
 of characteristics and $(X'_k,Y_k',S_k)_{k \geq 0}$ the
 corresponding embedded chain. Note that by construction for all $t\in \coint{S_k,S_{k+1}}$, $k \in \nset$, $X_t = X'_k + (t-S_k) Y'_k$ and $Y_t = Y'_k$. Therefore for all $t < \sup_{k \in \nset} S_k$, $X_t = \int_0^t Y_s \rmd s$ and only $(Y_t)_{t \geq 0}$ can be discontinuous in time. 
 
 The class of velocity jump processes gathers the Zig-Zag process \cite{BierkensFearnheadRoberts}, the Bouncy Particle Sampler (BPS) \cite{PetersdeWith} and many of their variants. \tcr{The choice for $(\lambda_i,Q_i)$, $i \in \{1,\ldots,\ell\}$, of these different (but similar) processes are mainly of one of the following type (here we only consider homogeneous mechanisms):}
 \begin{itemize}
 \item refreshment mechanism: the rate $\lambda(x,y)$ only depends on $x\in \rset^d$, and the kernel $Q$ is constant, \ie~there exists $\nu \in \mcp(\msv)$ such that for all $(x,y)\in \rset^d\times \msv$ and all $(\msa,\msa')\in  \mcb{\msm}\times  \mcb{\msv}$
   \[Q((x,y),\msa\times\msa')\ = \ \updelta_x(\msa) \nu(\msa')\eqsp.\]
 \item deterministic bounce mechanism: \tcr{there exists a measurable function $g:\rset^d \rightarrow \rset^d$, locally bounded, such that for all $(x,y) \in \msm$, $\lambda(x,y) = \ps{g(x)}{y}_+$ and $Q((x,y), \{x\} \times \{\Refl(x,y)\}) = 1 $, for a measurable function $\Refl : \msm \to \msy$. A particular example in the case $\msy = \sphere^d$ or $\msy= \rset^d$ and  $\Refl$
 is given for all $(x,y) \in \msm$ by}
   \begin{equation}
     \label{eq:def_refl}
     \Refl(x,y) \ = \
     \begin{cases}
       y - 2  \norm{g(x)}^{-2 }\ps{g(x)}{y} g(x)  & \text{ if $g(x) \not = 0$} \eqsp,\\
       y & \text{ otherwise} \eqsp.
     \end{cases}
   \end{equation}
   Note that $\Refl(x,y)$ is simply the orthogonal reflection of $y$ with respect to $g(x)$ if $g(x) \neq 0$.
 \item randomized bounce mechanism: there exists a measurable function $g:\rset^d \rightarrow \rset^d$ such that for all $(x,y) \in \msm$ and $\msa \in \mcb{\rset^d}$, $\msa' \in \mcb{\msy}$, $\lambda(x,y) = \ps{g(x)}{y}_+$ and $Q((x,y), \msa \times \msa') = \updelta_{x}(\msa) \tilde{Q}((x,y),\msa')$, where $\tilde Q$ is a Markov kernel on $\msm \times \mcb{\msy}$.
   % $\lambda(x,y) = \ps{g(x)}{y}_+$ and such that a representation $\bfG$ of $Q$ is such that, if $W$ is a uniform \rv~on $[0,1]$,
   % \[\mathbb E\po \ps{\na U(x)}{ \bfG ((x,y),W)}_+\pf \ = \ \ps{\na U(x)}{y}_-\eqsp.\]
 \end{itemize}

 For instance, \cite{BDMM2017} studies the velocity jump process  associated with the linear Boltzmann equation, which gives an exemple of refreshment mechanism. The Zig-Zag (ZZ) process  \cite{BierkensFearnheadRoberts}  and the Bouncy Particle Sampler (BPS) \cite{PetersdeWith,MonmarcheRTP,Doucet2017} are recently proposed PDMP used to sample from a target density $\pi \propto \exp(-U)$, where $U: \rset^d \to \rset$ is a continuously differentiable function.
 The ZZ process is a velocity jump process with $\msy = \{-1,1\}^d$ and $d$ deterministic bounce mechanisms $(\lambda_i,Q_i)_{i \in \iint{1,d}}$ given for all $i \in \iint{1,d}$, $x \in \rset^d$, $y \in \{-1,1\}^d$ and $\msa \in \mcb{\rset^d}$ by
 \begin{equation*}
   \lambda_i(x,y) = (y_i \partial U(x) /\partial x_i)_+ \eqsp, \qquad Q_i((x,y),\{ x \} \times \{(y_1,\ldots,y_{i-1},-y_i,y_{i+1},\ldots,y_d)\}) = 1 \eqsp.
 \end{equation*}
 Note that in this case, for all $x \in \rset^d$,  $g_i(x) = (\partial U(x)/ \partial x_i) \bfe_i$ where $\bfe_i$ is the $i^{\text{th}}$ vector of the standard basis of $\rset^d$. Additional refreshment mechanisms can be added to the process. 
  % The BPS \cite{PetersdeWith,MonmarcheRTP,Doucet2017} is a velocity jump process with $\msv = \rset^d$ or $\mathbb S_{d-1}$ with only one pure bounce mechanism associated to $g(x) = \na U(x)$ for some $U \in \rmc^1(\rset^d)$ and a refreshment mechanism $(\lambda,Q)$ where $\lambda(x,y) = \rate>0$ is constant and $Q((x,y),\cdot) = \updelta_x \otimes \nu$ with $\nu$ being either a Gaussian measure on $\rset^d$ or the uniform measure on $\mathbb S_{d-1}$. 
In the rest of this paper, we will repeatedly use the BPS process as an illustration to our different results.

\begin{example}\label{ex:BPS}
Let $\msv$ be a smooth closed sub-manifold of $\rset^d$ rotation invariant, \ie~for any rotation $O$ of $\rset^d$, $O \msv = \msv$.  Let $\rate>0$ and $\loiy \in \mcp(\msv)$. The BPS process associated with the potential $U$, refreshment rate $\rate$ and refreshment distribution $\loiy$ is the PDMP on $\msm = \rset^d\times \msv$ with characteristics $( \varphi,(\lambda_i,Q_i)_{i\in \iintLigne{1,2}})$ where $\varphi$ is given by \eqref{eq:transport_libre} and for all $(x,y)\in\rset^{d}\times\msv$, $\msa \in \mcb{\rset^d}$, $\lambda_1(x,y) = \ps{y}{\nabla U(x)}_+$, $\lambda_2(x,y) = \rate$, $Q_1((x,y),\msa \times\{\Refl(x,y)\}) = \updelta_x(\msa)$ and $Q_2((x,y),\cdot) = \updelta_x \otimes \loiy$, where $\Refl$ is given by \eqref{eq:def_refl} with $g(x)=\nabla U(x)$. Note that $(\lambda_1,Q_1)$ is the pure bounce mechanism associated with $g$ (see Figure~\ref{fig:BPS}) and  $(\lambda_2,Q_2)$ is a refreshment mechanism.
\end{example}
Variants of the BPS with randomized bounces have been recently introduced in \cite{michel:senecal:2017,wu:robert:2017,vanetti:etal:2017}.
\begin{center}
  \begin{figure}\begin{tikzpicture}[line cap=round,line join=round,>=triangle 45,x=1.2cm,y=1.2cm]
\clip(0.46,0.49) rectangle (9.36,6.42);
\draw [rotate around={22.34:(3.83,2.62)},dash pattern=on 1pt off 1pt] (3.83,2.62) ellipse (1.1cm and 0.56cm);
\draw [rotate around={-176.09:(3.72,2.57)},dash pattern=on 3pt off 3pt] (3.72,2.57) ellipse (3.32cm and 1.48cm);
\draw [rotate around={-169.9:(3.74,2.65)},dash pattern=on 3pt off 3pt] (3.74,2.65) ellipse (1.93cm and 1.03cm);
\draw [rotate around={-166.74:(4.51,2.83)},dash pattern=on 3pt off 3pt] (4.51,2.83) ellipse (4.74cm and 2.44cm);
\draw [color=fftttt] (5.14,4.48)-- (7.08,4.82);
% \draw [color=wwwwff,domain=0.46:9.36] plot(\x,{(-24540.65--1074.63*\x)/-3512.92});
% \draw [color=wwwwff,domain=0.46:9.36] plot(\x,{(--19691.79-3512.92*\x)/-1074.63});
\draw [color=fftttt] (8.5,3.45)-- (7.08,4.82);
\draw (5.36,5.08) node[anchor=north west] {$ (Y_t)_{t \leq S_k} $};
\draw (7.0,4.18) node[anchor=north west] {$ (Y_t)_{t >S_k} $};
\draw (6.60,5.36) node[anchor=north west] {$X_{S_k}$};
\draw (7.39,5.83) node[anchor=north west] {$ \nabla U(X_{S_k}) $};
\draw (3.5,5.55) node[anchor=north west] {$\{ \nabla U(X_{S_k})\}^{\perp} $};
\draw [color=wwwwff] (7.56,6.37)-- (6.54,3.05);
\draw [color=wwwwff] (2.63,6.18)-- (9.24,4.16);
\begin{scriptsize}
\fill [color=qqqqff] (7.08,4.82) circle (1.5pt);
\end{scriptsize}
\end{tikzpicture}
% \begin{tikzpicture}[line cap=round,line join=round,>=triangle 45,x=1.2cm,y=1.2cm]
% \clip(0.46,0.49) rectangle (9.36,6.42);
% \draw [rotate around={22.34:(3.83,2.62)},dash pattern=on 1pt off 1pt] (3.83,2.62) ellipse (0.92cm and 0.47cm);
% \draw [rotate around={-176.09:(3.72,2.57)},dash pattern=on 2pt off 2pt] (3.72,2.57) ellipse (2.76cm and 1.23cm);
% \draw [rotate around={-169.9:(3.74,2.65)},dash pattern=on 2pt off 2pt] (3.74,2.65) ellipse (1.61cm and 0.86cm);
% \draw [rotate around={-166.74:(4.51,2.83)},dash pattern=on 2pt off 2pt] (4.51,2.83) ellipse (3.95cm and 2.03cm);
% \draw [color=fftttt] (5.14,4.48)-- (7.08,4.82);
% %\draw [color=wwwwff,domain=0.46:9.36] plot(\x,{(-24540.65--1074.63*\x)/-3512.92});
% %\draw [color=wwwwff,domain=0.46:9.36] plot(\x,{(--19691.79-3512.92*\x)/-1074.63});
% \draw [color=fftttt] (8.5,3.45)-- (7.08,4.82);
% \draw (5.36,5.03) node[anchor=north west] {$ (Y_t)_{t \leq S_k} $};
% \draw (7.18,4.18) node[anchor=north west] {$ (Y_t)_{t >S_k} $};
% \draw (6.74,5.36) node[anchor=north west] {$X_{S_k}$};
% \draw (7.39,5.83) node[anchor=north west] {$ \nabla U(X_{S_k}) $};
% \draw (3.5,5.65) node[anchor=north west] {$\{ \nabla U(X_{S_k})\}^{\perp} $};
% \draw [color=wwwwff] (7.56,6.37)-- (6.54,3.05);
% \draw [color=wwwwff] (2.63,6.18)-- (9.24,4.16);
% \begin{scriptsize}
% \fill [color=qqqqff] (7.08,4.82) circle (1.5pt);
% \end{scriptsize}
% \end{tikzpicture}
\caption{\tcr{Representation of a deterministic bounce in the case $g=\nabla U$ for some potential $U$. The dashed lines are level sets of $U$. At the jump time $S_k$, $V_{S_k}$ is reflected with respect to $\nabla U(X_{S_k})$.}}\label{fig:BPS}
\end{figure}
\end{center}

%%% Local Variables:
%%% mode: latex
%%% TeX-master: "main"
%%% End:

%% file: constructions.tex
\section{Alternative constructions}
\label{sec:alternative-constructions}

Consider PDMP characteristics
$( \varphi,(\lambda_i,Q_i)_{i\in \iintLigne{1,\ell}})$, an initial
distribution $\mu_0 \in \mcp(\msm)$ and the associated process
$(X_t)_{t \geq 0}$ defined in \Cref{sec:framework}. The goal of this
Section is to construct another process $(Y_t)_{t \geq 0}$ on the same
probability space $(\Omega,\mcf,\mathbb{P})$ with the same
distribution on $\mrd(\rset_+,\msm \cup \{\infty\})$ as
$(X_t)_{t \geq 0}$. \tcr{This alternative construction will turn out
  to be interesting to obtain a characterization of the distribution
  of the first jump time of type $i \in \{1,\ldots,\ell\}$; see
  \Cref{prop:premier_bounce} below. The main difference between the
  two constructions is the following. In \Cref{const:1}, the sequence
  $(E_{j,k+1})_{j\in \iintLigne{1,\ell}}$ is only used after the
  $k^{\mathrm{th}}$ jump to define $S_{k+1}$, and they
  are all discarded afterwards. In contrast, \Cref{const:2} takes
  advantage of the memoryless property of the exponential
  distribution: the random variables
  $(\tilde{H}_{j,k+1})_{j\in \iintLigne{1,\ell}}$ which play the same
  role than $(E_{j,k+1})_{j\in \iintLigne{1,\ell}}$ to define the jump
  times $S_{k+1}$ are simply updated so that they are
  \iid~exponential random variables independent of $\mcf'_{k}$ (see \Cref{lem:prop_bH}), which
  is the main reason why \Cref{const:2} defined a process with
  the same distribution than \Cref{const:1} (see \Cref{propo:eg_deux_proc_1}). % they to and only , only the exponential
  % clock that has lead to the $k^{th}$ jump (i.e. with index
  % $j\in \iintLigne{1,\ell}$ such that $S_k=S_{j,k}$) is refreshed to a
  % new value, the others are kept in the next step, taking advantage of
  % the memoryless property of the exponential law. This requires the
  % definition of a counter variable ($\tilde N_{j,k}$ below) that keeps
  % track of how many jumps of type $j$ have already occurred before the
  % $k^{th}$ jump. The second difference is that, in
  % Construction~\ref{const:1}, the randomness involved in the position
  % after the $k^{th}$ jump is simply given by a variable $U_k$ while,
  % in Construction~\ref{const:2}, each jump type $j$ has its own stock
  % of variables $(U_{j,k})_{k\in\nsets}$ so that $U_{j,k}$ is consumed
  % only at the $k^{th}$ jump of type $j$.
}

\noindent\hfil\rule{\textwidth}{.4pt}\hfil
\begin{construction}
\label{const:2}
 Let $W_0$ be a \rv~ with distribution $\mu_0 \in \mcp(\msm)$ and $(E_{j,k},U_{j,k})_{j\in\iintLigne{1,\ell},k\in \nsets}$ be  an i.i.d. family, independent of $W_0$, such that for all $k\in \nset$ and $j\in \iintLigne{1,\ell}$, $U_{j,k}$ is uniformly distributed on $[0,1]$  and $E_{j,k}$ is an exponential \rv~ with parameter $1$, independent of $U_{j,k}$.\\
  Set $\tS_0 = 0$, $Y_0' = W_0$, $\tI_0=0$, $\bN_{0} = 1$, $\tH_{j,1} = E_{j,1}$ and $\tN_{j,1} = 1$, for $j\in \cco 1,\ell\ccf$. Suppose that $(Y_t)_{t \leq \S_k}$ and  $(Y_{k}',\tS_k,\tI_k,\bN_{k},(\tH_{j,k+1},\tN_{j,k+1})_{j \in \iint{1,\ell}})$ have been defined for some $k\in \mathbb N$. Set 
\begin{equation*}
\tS_{j,k+1}   =  \inf \ensemble{ t>\tS_k}{ \tH_{j,k+1} < \int_{\tS_k}^t \lambda_j\parentheseLigne{s, \varphi_{\tS_k,s}(Y_{k}')} \rmd s }\eqsp, \, \tS_{k+1}  =  \underset{j\in\iint{1,\ell}}\min \tS_{j,k+1} \eqsp.
\end{equation*}
\begin{itemize}
\item If $\tS_{k+1} = +\infty$, set $\tS_m=+\infty$, $Y'_{m} = \infty$, $\tI_m = 1$ for all $m>k$ and $Y_t = \varphi_{\tS_k,t}(Y'_k)$ for $t \geq \tS_k$.
\item  If $\tS_{k+1} < +\infty$, set
  \[
    \begin{array}{rclcrcl}
      \tI_{k+1} & =  & \min \{j\in \cco 1,\ell \ccf,\ \tS_{j,k+1} = \tS_{k+1}\} &\eqsp, \, & \bN_{k+1} =  \tN_{\tI_{k+1},k+2}&  =&  \tN_{\tI_{k+1},k+1}+1 \eqsp, \\
      Y_{k+1}' &= &
\bfG_{\tI_{k+1}} (  \tS_{k+1}, \varphi_{\tS_k,\tS_{k+1}}(Y_{k}'),U_{\tI_{k+1},\bN_{k+1}})
 & \eqsp, \, & \tH_{\tI_{k+1},k+2}  &= &  E_{\tI_{k+1},\bN_{k+1}} \eqsp, 
  \end{array}
\]
and for $j\neq \tI_{k+1}$,
\begin{equation*}
  \tH_{j,k+2}  =  \tH_{j,k+1} - \int_{\tS_k}^{\tS_{k+1}} \lambda_j\parentheseLigne{s, \varphi_{\tS_k,s}(Y_{k}')} \rmd s \eqsp, \quad \tN_{j,k+2} = \tN_{j,k+1} \eqsp. 
\end{equation*}
Set $Y_t =  \varphi_{\tS_k,t}(Y'_k)$ for $t \in \oointLigne{\tS_k,\tS_{k+1}}$ and  $Y_{\tS_{k+1}} = Y'_{k+1}$.
\end{itemize}
For $t \geq \sup_{k \in \nset} \tS_k$, set $Y_t = \infty$.
\end{construction}
\noindent\hfil\rule{\textwidth}{.4pt}\hfil
\bigskip

We show in the following result that the two constructions we consider define the same distribution on $\rmD(\rset_+,\msm\cup\{\infty\})$.
 \begin{proposition}
   \label{propo:eg_deux_proc_1}
The two Markov chains $(X'_k,S_k,I_k)_{k \in \nset}$ and
   $(Y'_k,\tS_k,\tI_k)_{k \in \nset}$ have the same distribution on
   $((\msm\cup\{\infty\}) \times (\rset_+\cup\{+\infty\}))^{\nset}$. Therefore, $(X_t)_{t \geq 0}$ and $(Y_t)_{t \geq 0}$ have the same distribution on $\mrd(\rset_+,\msm\cup\{\infty\})$.
 \end{proposition}
 We preface the proof by a lemma.
 Denote by $(\tmcf_{k})_{k\in \nset}$ and $(\tmcf_{k}')_{k\in \nset}$ the
filtration associated with the sequence of random variables $(Y_{k}',\tS_k,\tI_k)_{k \in \nset}$ and $(Y_{k}',\tS_k)_{k \in \nset}$.  

% Before showing that $(\bX_t,\bY_t)_{t \geq 0}$ has the same
% distribution as $(\tX_t,\tY_t)_{t \geq 0}$, we need to derive some
% properties on the distributions associated with the sequence of random
% variables $(\bH_n)_{n \in \nset^*}$. 
\begin{lemma}
  \label{lem:prop_bH}
Let $k \in \nset$. For all $i \in \iint{1,\ell}$, given
  $\{\tI_{k} = i \}\cap \{\tS_{k} < \plusinfty\}$,
  $(\tH_{j,k+1})_{j \in \iint{1,\ell}\setminus\{i\}}$
  % $(\tH_{j,k}-\int_{0}^{\tT_{k+1}} \lambda_j( \tS_k+s, \varphi_{\tS_k,\tS_k+s}(Y_{k}')) \rmd s)_{j \in \iint{1,\ell}\setminus\{i\}}$
  are
  \iid~exponential random variables with parameter $1$, independent of
  $\tmcf'_{k}$.   In
  addition, given $\{\tS_k < \plusinfty\}$,
  $(\tH_{j,k+1})_{j \in \iint{1,\ell}}$ are \iid~exponential random
  variables with parameter $1$, independent of $\tmcf'_{k}$. Finally, $\probaLigne{\tS_{j,k+1}  =\tS_{j',k+1}  < \plusinfty | \tmcf_k'} = 0$ for any $j,j' \in\iint{1,\ell}$. 
\end{lemma}
\begin{proof}
  First, it is sufficient to consider the case $\ell >1$, since in the
  case $\ell = 1$, $H_{1,k+1} = \bE_{1,k+1}$ for any $k\in\nset$.

  Note that the first statement is equivalent to for any
  $i \in \iint{1,\ell}$,
  $(t_j)_{j \in\iint{1,\ell}\setminus\{i\}} \in \rset_+^{\ell-1}$,
  $\msb \in \tmcf_k'$, setting $\msa_{i,k} = \bigcap_{j=1,j\not= i}^{\ell} \{\tH_{j,k+1} \geq t_j \}$,
  \begin{equation*}
\textstyle{    \expe{  \1_{\rset_+}(\tS_{k}) \1_{i}(\tI_{k}) \1_{\msa_{k,i}} \1_{\msb}} = \exp\parentheseLigne{-\sum_{j\in\iint{1,\ell}\setminus\{i\}} t_j} \expe{  \1_{\rset_+}(\tS_{k}) \1_{i}(\tI_{k})  \1_{\msb}} } \eqsp.
  \end{equation*}
  Indeed taking $\msb = \Omega$, we get
  $\probaLigne{\msa_{k,i} \, |\, \tI_{k} = i, \tS_k < \plusinfty} =
  \exp\parentheseLigne{-\sum_{j\in\iint{1,\ell}} t_j}$. This statement
  is also equivalent to
  \begin{equation}
    \label{proof:lem:prop_bH_eq_2}
\textstyle{    \1_{\rset_+}(\tS_{k}) \CPPLigne{\bigcap_{j=1,j\not= i}^{\ell} \{\tH_{j,k+1} \geq t_j\}\cap \{\tI_{k} = i \} }{\tmcf'_{k}}  = \exp\parentheseLigne{-\sum_{j\in\iint{1,\ell}\setminus\{i\}} t_j} \1_{\rset_+}(\tS_{k}) \CPPLigne{\tI_k = i}{\tmcf'_{k}} } \eqsp,
\end{equation}
which is statement that we show. Similarly, we show that for the second statement, it is sufficient to show that $ \1_{\rset_+}(\tS_{k}) \probaLigne{\bigcap_{j=1}^{\ell} \{\tH_{j,k+1} \geq t_j\} \, |\, \tmcf'_{k}}  = \exp\parentheseLigne{-\sum_{j\in\iint{1,\ell}} t_j}$ for any $(t_j)_{j\in\iint{1,\ell}} \in \rset_+^{\ell}$.

  % Note that the second statement is equivalent to for all
  % $k \in \nset$ and $i \in \iint{1,\ell}$,
  % $(\tH_{j,k+1}-B_{j,k+1})_{j \in \iint{1,\ell}\setminus\{i\}}$ are
  % \iid~exponential random variables with parameter $1$,  independent of $\sigma(\tmcf'_{k-1},\tS_{i,k})$ given
  % $\{\tS_{k} = \tS_{i,k}\}$ since for all $\msa \in \tmcf'_k$,
  % $\{\tS_{k} = \tS_{i,k}\} \cap \msa \in
  % \sigma(\tmcf'_{k-1},\tS_{i,k})$, which is the result that we will
  % show.

The proof is by induction on $k \in \nsets$. For $k=0$, by definition
the first two statements hold. The last one follows from
\eqref{eq:def_s_j_const_1} for $k=0$ and \Cref{const:1},
\Cref{const:2}. % The second part follows
  % from the memoryless property of the exponential distribution and
  % because for all $i \in \iint{1,\ell}$, $(\tH_{1,j} = E_{1,j})_{j \in \iint{1,\ell}\setminus\{i\}}$, is  independent of
  % $\sigma(\tmcf'_{0},\tS_{1,i})$.

  Assume now that the result holds for $k \in \nset$ and let
  $i \in \iint{1,\ell}$.

  First, we show that the last statement for $k+1$: for any 
  $j',j \in \iint{1,\ell}$,
  $\probaLigne{S_{j,k+1} = S_{j',k+1} < \plusinfty|\tmcf_k'} = 0$. Indeed, given
  $\tmcf_k'$, by definition $S_{j,k+1} = S_{j',k+1}$ if and only if
  $\psi_{j,k+1}^{\leftarrow}(H_{j,k+1}) =
  \psi_{j',k+1}^{\leftarrow}(H_{j',k+1})$, where
  $\psi_{j,k+1}^{\leftarrow}$ is the generalized inverse of the nondecreasing function $\psi_{j,k+1}(t) =  \int_{\tS_k}^{\tS_k+t} \lambda_j\parentheseLigne{s, \varphi_{\tS_k,s}(Y_{k}')} \rmd s$.
Since $\psi_{j',k+1}$ is continuous, we get using properties of the generalized inverse that given
  $\tmcf_k'$, $S_{j,k+1} = S_{j',k+1}$ if and only if
  $\psi_{j',k+1} \circ \psi_{j,k+1}^{\leftarrow}(H_{j,k+1}) =
  H_{j',k+1}$, which implies by the induction hypothesis, $\probaLigne{S_{j,k+1} = S_{j',k+1} \,|\, \tmcf_k'} = 0$. It follows that given $\tmcf_{k+1}'$ and $\{\tS_{k+1} < \plusinfty\}$, using again properties of the generalized inverse,
  \begin{multline}
    \label{proof:lem:prop_bH_eq_1}
    \{I_{k+1} = i\} = \{H_{i,k+1} = \psi_{i,k+1}(\tS_{i,k+1})\} \cap \defEns{\cap_{j\in\iint{1,\ell}\setminus\{i\}} \{\tS_{j,k+1} > \tS_{i,k+1}\}}\\ = \{H_{i,k+1} = \psi_{i,k+1}(\tS_{k+1} - \tS_k)\} \cap \defEns{\cap_{j\in\iint{1,\ell}\setminus\{i\}} \{H_{j,k+1} > B_{j,k+1}\}} \eqsp,
  \end{multline}
  where we have set $B_{j,k+1} = \1_{\rset_+} (\tS_{k+1})\int_{\tS_{k}}^{\tS_{k+1}}
  \defEnsLigne{\lambda_j(s, \varphi_{\tS_{k},s}(Y_{{k}}'))} \rmd s$.

  We are now able to show that the first two statements hold for $k+1$. Note that  almost surely, we have by \eqref{proof:lem:prop_bH_eq_1} and the induction hypothesis,
for
all $t_1,\ldots,t_\ell \geq 0$,
\begin{align*}
  &  \textstyle{ \1_{\rset_+}(\tS_{k+1}) \CPPLigne{\bigcap_{j=1,j\not= i}^{\ell} \{\tH_{j,k+2} \geq t_j\}\cap \{\tI_{k+1} = i \} }{\tmcf'_{k+1}} }\\
  &   \textstyle{ =\1_{\rset_+}(\tS_{k+1}) \CPPLigne{\{H_{i,k+1} = \psi_{i,k+1}(\tS_{k+1})\} \cap \bigcap_{j=1,j\not= i}^{\ell} \{\tH_{j,k+1}-B_{j,k+1} \geq t_j\}\cap\{H_{j,k+1} > B_{j,k+1}\} }{\tmcf_{k+1}'} }\\
  &    \textstyle{    = \1_{\rset_+}(\tS_{k+1})\exp\parentheseLigne{-\sum_{j=1,j \not = i}^{\ell} \{t_j+B_{j,k+1}\} } 
    \CPPLigne{\{H_{i,k+1} = \psi_{i,k+1}(\tS_{k+1})\}  }{\tmcf_{k+1}'} } \\
  % &   \textstyle{ = \1_{\rset_+}(\tS_{k+1})\exp\parentheseLigne{-\sum_{j=1,j \not = i}^{\ell} t_j} 
  %   \CPPLigne{ I_{k+1} = i  }{\tmcf_{k+1}'} \eqsp.
  %   }
\end{align*}
establishing \eqref{proof:lem:prop_bH_eq_2} for $k+1$.
For the second statement, using \eqref{proof:lem:prop_bH_eq_2} for $k+1$ and that   $E_{i,\bN_{k+2}}$ is an exponential random variable independent of $\sigma(\tmcf'_{k+1},
  (\tH_{j,k+1})_{j \in \iint{1,\ell}})$ given $\{I_{k+1} = i\}$ for all $i \in \iint{1,\ell}$, we obtain that 
\begin{align*}
  &  \textstyle{\1_{\rset_+} (\tS_{k+1}) \CPPLigne{\bigcap_{j=1}^{\ell}\{\tH_{j,k+2} \geq t_j \}}{\tmcf'_{k+1}}} &\\
&   \qquad \qquad \textstyle{ =   \1_{\rset_+} (\tS_{k+1}) \sum_{i=1}^\ell  \CPPLigne{\bigcap_{j=1}^{\ell}\{\tH_{j,k+2} \geq t_j \} \cap\{ I_{k+1} = i\}}{\tmcf'_{k+1}} } \\
&\qquad \qquad \textstyle{=   \1_{\rset_+} (\tS_{k+1})   \sum_{i=1}^\ell  \CPPLigne{\bigcap_{j=1, j \not = i}^{\ell}\{\tH_{j,k+2} \geq t_j \} \cap\{I_{k+1} = i\} \cap \{\bE_{i,\bN_{k+2}} \geq t_i\}}{\tmcf'_{k+1}} }\\
  &\qquad  \qquad  \textstyle{= \1_{\rset_+} (\tS_{k+1})  \exp\parentheseLigne{-\sum_{j=1}^{\ell} t_j}  \sum_{i=1}^\ell  \CPPLigne{I_{k+1} = i}{\tmcf'_{k+1}}} = \1_{\rset_+} (\tS_{k+1})  \exp\parentheseLigne{-\sum_{j=1}^{\ell} t_j}\eqsp.
\end{align*}
\end{proof}

\begin{proof}[Proof of \Cref{propo:eg_deux_proc_1}]
  We show that the two processes $(X'_k,S_k,I_k)_{k \in \nset}$ and
  $(Y'_k,\tS_k,\tI_k)_{k \in \nset}$ have the same distribution.  Note
  that, since   $(X'_0,S_0,I_0)$ and $(Y'_0,\tS_0,\tI_0)$ have the same distribution,
  this result is equivalent  to show that
  $(Y'_k,\tS_k,\tI_k)_{k \in \nset}$ is also a Markov chain with a
  Markov kernel characterized by \eqref{eq:carac_chaine_markov_1}.
  %Let $(\tmcf_k)_{k \in\nsets}$ be the filtration associated with $(Y'_k,\tS_k,\tI_k)_{k \in \nset}$.
% \begin{multline}
% \label{eq:carac_chaine_markov_1}
%   \CPP{X'_{k+1}\in \msa, S_{k+1} \leq t, i_{k+1} = j}{\mcf_k} \\
% = \int_0^{t-S_k} \rmd s \parentheseDeux{\lambda_j(S_k+s,\varphi_{S_k,S_k+s}(X_k')) \, \exp\defEns{-\sum_{i=1}^\ell \int_0^s \rmd u \lambda_i(S_k+u,\varphi_{S_k,S_k+u}(X_k')) }}\eqsp.
% \end{multline}
Let $k \in \nset$, $\msa \in \mcb{\msm}$, $t \geq \tS_k$,
$i\in \iint{1,\ell}$. Note first that given $\{I_{k+1} = i \}$ and $\tS_k$, for $i \in\iint{1,\ell}$, $\psi_{i,k+1}^{\leftarrow} (\tH_{i,k+1}) = \tS_{i,k+1} - \tS_k$, where $\psi_{i,k+1}^{\leftarrow}$ is generalized inverse of  $\psi_{i,k+1}(t) =  \int_{\tS_k}^{\tS_k+t} \lambda_i\parentheseLigne{s, \varphi_{\tS_k,s}(Y_{k}')} \rmd s$. Then,
By \Cref{lem:prop_bH} and definition of
$(Y'_{\tilde{k}},\tS_{\tilde{k}},\tI_{\tilde{k}})_{\tilde{k} \in \nset}$, since, given $\{\tI_{k+1}=i\} \cap \{\tS_{k+1} < \plusinfty\}$,
$Y'_{k+1} = \bfG(\tS_{i,k+1}, \varphi_{\tS_k,\tS_{i,k+1}}(Y'_k),U_{i,\bN_{k+1}})$ and $U_{i,\bN_{k+1}}$ is independent of $\sigma(\tmcf_k,(\tH_{j,k+1})_{j \in\iint{1,\ell}})$, then setting
$\mcg_{i,k+1} =\sigma(\tmcf_k,\tH_{i,k+1},U_{i,\bN_{k+1}})$ and $B_{j,k+1} = \1_{\rset_+}(\tS_{i,k+1})\int_{\tS_k}^{\tS_{i,k+1}} \lambda_j( s,
\varphi_{\tS_k,s}(Y_{k}')) \rmd s$, for $j \in \iint{1,\ell}\setminus\{i\}$, we have
\begin{align*}  
  &\CPPLigne{Y'_{k+1}\in \msa, \tS_{k+1} \leq t, \tI_{k+1} = i}{\tmcf_k} \\
  &     =   \CPPLigne{Y'_{k+1} \in \msa, \tS_{i,k+1} \leq t , \tS_{i,k+1} <  \tS_{j,k+1}, \text{ for all $j \in \iint{1,\ell}\setminus\{i\}$}, \tI_{k+1} = i}{\tmcf_k}    \\
  &     =   \textstyle{ \CPELigne{\1_\msa(Y'_{k+1})  \1_{\ccint{\tS_k,t}}(\tS_{i,k+1}) \CPPLigne{ \bigcap_{j \in \iint{1,\ell}, j \not = i} \{\tS_{i,k+1} < \tS_{j,k+1}\}}{\mcg_{i,k+1}}}{\tmcf_k}   }  \\
&     =   \textstyle{ \CPELigne{\1_\msa(Y'_{k+1}) \1_{\ccint{\tS_k,t}}(\tS_{i,k+1}) \CPPLigne{  \bigcap_{j \in \iint{1,\ell}, j \not = i} \{ B_{j,k+1} < \tH_{j,k+1}\}}{\mcg_{i,k+1}}}{\tmcf_k} } \\
  &=  \textstyle{ \CPELigne{\1_\msa(Y'_{k+1}) \1_{\ccint{\tS_k,t}}(\tS_{i,k+1}) \exp\defEnsLigne{-\sum_{j=1,j \not = i}^\ell \int_{\tS_k}^{\tS_{i,k+1}}  \lambda_j(u,\varphi_{\tS_k,u}(Y_k'))\rmd u }}{\tmcf_k} } \eqsp. 
  % &\CPPLigne{Y'_{k+1}\in \msa, \tS_{k+1} \leq t, \tI_{k+1} = j}{\tmcf_k} \\
  % &     =   \CPPLigne{\bfG(\tT_{j,k+1}+ \tS_k, \varphi_{\tS_k,\tT_{j,k+1}+ \tS_k}(Y'_k),U_{k+1}) \in \msa, \tT_{j,k+1} \leq t - \tS_k, \tT_{j,k+1} \leq \tT_{i,k+1}, \text{ for all $i \in \iint{1,\ell}\setminus\{j\}$}}{\tmcf_k}    \\
  % &     =   \CPELigne{\1_\msa(\bfG(\tT_{j,k+1}+ \tS_k, \varphi_{\tS_k,\tT_{j,k+1}+ \tS_k}(Y'_k),U_{k+1})) \1_{\ccint{0,t-S_k}}(\tT_{j,k+1}) \CPPLigne{ \tT_{j,k+1} \leq \tT_{i,k+1}, \text{ for all $i \in \iint{1,\ell}\setminus\{j\}$}}{\sigma(\tmcf_k,\tH_{j,k+1},U_{k+1})}}{\tmcf_k}   \\
  % &     =   \CPELigne{\1_\msa(\bfG(\tT_{j,k+1}+ \tS_k, \varphi_{\tS_k,\tT_{j,k+1}+ \tS_k}(Y'_k),U_{k+1})) \1_{\ccint{0,t-S_k}}(\tT_{j,k+1}) \CPPLigne{ B_{i,k+1} < \tH_{i,k+1}, \text{ for all $i \in \iint{1,\ell}\setminus\{j\}$}}{\sigma(\tmcf_k,\tH_{i,k+1},U_{k+1})}}{\tmcf_k}\\
  % &=  \CPE{\1_\msa(\bfG(\tT_{j,k+1}+ \tS_k, \varphi_{\tS_k,\tT_{j,k+1}+ \tS_k}(Y'_k),U_{k+1})) \1_{\ccint{0,t-S_k}}(\tT_{j,k+1}) \exp\defEns{-\sum_{i=1,i \not = j}^\ell \int_0^{\tT_{j,k+1}} \rmd u \lambda_i(S_k+u,\varphi_{S_k,S_k+u}(X_k')) }}{\tmcf_k} \eqsp. 
\end{align*}
The proof then follows from $\psi_{i,k+1}^{\leftarrow} (\tH_{i,k+1}) = \tS_{i,k+1} - \tS_k$ \Cref{lem:prop_bH} and the definition of $Y'_{k+1}$.
\end{proof}

%The only  difference between $(X_t)_{t \geq 0}$ and $Y$ is the following: in the first case, for $k\geqslant 1$, at time $T_k$, if $j\neq i_k$, then the variable $E_{j,k}$ is thrown away and a new variable $E_{j,k+1}$ is used to defined $T_{j,k+1}$. In the second case, at time $S_k$, if $i\neq j_k$, then $E_{j,N_{j,k}}$ is kept.
For $k\geqslant 1$ with $S_k <\infty$, we say that, at time $S_k$, the process $(X_t)_{t \geq 0}$ given by \Cref{const:1} has made a jump of type $I_k$, or equivalently that $S_k$ is a jump time of type $I_k$. Define
\begin{equation}
  \label{eq:def_t_j}
  T^{(j)}\ = \ \inf\ensemble{S_k}{\ k\geqslant 1,\ I_k = j} \eqsp,
\end{equation}
 the first jump time of type $j$. Then, one example of application of \Cref{propo:eg_deux_proc_1} is the following result.
 \begin{proposition}
   \label{prop:premier_bounce}
Let $\ell=\ell_1+\ell_2$ with $\ell_1,\ell_2\in\nsets$. Let $(X_t)_{t \geq 0}$ be a PDMP on $\msm$ with characteristics $(\varphi,(\lambda_i,Q_i)_{i \in \iint{1,\ell}})$ given by \Cref{const:1} and initial distribution $\mu_0$. Define $T = \min\ensembleLigne{T^{(i)}}{i\in \iint{\ell_1+1,\ell} }$, where $T^{(j)}$ is given by \eqref{eq:def_t_j} for all $j \in \iint{1,\ell}$ . Then the cumulative distribution function of $T$ is given for all $u\geqslant 0$ by
\begin{equation}
   \label{eq:premier_bounce}
\proba{ T \leq u}  =  \proba{ E <  \sum_{i\in \iint{\ell_1+1}}^{\ell}\int_0^{u \wedge \tau_\infty(Z)}\lambda_{i}\po s,Z_s\pf \rmd s }\eqsp,
 \end{equation}
 where $ (Z_t)_{t \geq 0}$ is a PDMP with characteristics $(\varphi,(\lambda_i,Q_i)_{i\in \iint{1,\ell_1}})$  and initial distribution $\mu_0$, and $E$ is a standard exponential \rv~independent of $ (Z_t)_{t \geq 0}$.
 \end{proposition}

 \begin{proof}
   Let $(Y_t)_{t \geq 0}$ be a PDMP defined by  \Cref{const:2} with characteristics  $(\varphi,(\lambda_i,$ $Q_i)_{i\in \iint{1,\ell}})$  and initial distribution $\mu_0$. Define similarly to $(X_t)_{t \geq 0}$ for all $j \in \iint{1,\ell}$,
%    \begin{equation}
%   \label{eq:def_t_j_y}
% \eqsp,
% \end{equation}
$  \tT^{(j)}\ = \ \inf\ensembleLigne{\tS_k}{\ k\geqslant 1,\ \tI_k = j} $ and $\tT = \min\ensembleLigne{\tT^{(i)}}{i \in \iint{\ell_1+1,\ell_2} }$. Note that since by \Cref{propo:eg_deux_proc_1}, $(X'_k,S_k,I_k)_{k \in \nset}$ and $(Y'_k,\tS_k,\tI_k)_{k \in \nset}$ has the same distribution, $T$ and $\tT$ have the same distribution and it suffices to show that the cumulative distribution function of $\tT$ is given by \eqref{eq:premier_bounce}.

Let $(Z_t)_{t \geq 0}$ be a PDMP with characteristics $(\varphi,(\lambda_i,Q_i)_{i \in \iint{1,\ell_1}})$ and initial distribution $\mu_0$ defined by \Cref{const:2} and based on the \rvs~$W_0$, $(E_{j,k},U_{j,k})_{j \in \iint{1,\ell_1}, k \in \nset}$, and let $(Z_k',R_k,J_k)_{k \in \nset}$ be the corresponding embedded chain. By construction, for all $t \leq \tT \wedge \tau_{\infty}(Z)$, $Y_t =Z_t$. In addition, define
\begin{equation*}
  N = \inf\ensemble{k \in \nsets}{\text{ there exists } i \in \iint{\ell_1+1,\ell}, \tI_k = i} \eqsp.
\end{equation*}
By definition, $\tT= \tS_N$ on $\{ \tT <  \plusinfty\}$ and for all $k \in \nsets$, on $\{N=k\}$ for all $t \in \ccintLigne{0,\tS_k}$, $Y_t = Z_t$, and for all $n \leq k$, $Y'_n = Z_n'$, $\tS_n = R_n$. Therefore, for all $k\in \nsets$, on $\{N\geq k\}\cap\{\tau_{\infty}(Z) \geq t\}$, for all $i \in \iint{\ell_1+1,\ell}$, we have by induction
\begin{equation*}
  \tS_{i,k} = \inf \ensemble{ t>R_{k-1}}{ \tH_{i,k} < \int_{R_{k-1}}^t \lambda_i\parenthese{s, \varphi_{R_{k-1},s}(Z_{k}')} \rmd s }
   = \inf \ensemble{ t> 0 }{ E_{i,0} < \int_{0}^t \lambda_i\parenthese{s, Z_s} \rmd s } \eqsp.
 \end{equation*}
Since $\{\tT < \tau_{\infty}(Z)\} \subset \{\tT < \plusinfty\}$, we thus obtain
\begin{align}
  \nonumber
  \{\tT < \tau_{\infty}(Z)\} \cap   \{\tT \geq t\} &=  \{\tT < \tau_{\infty}(Z)\} \cap \defEns{ \bigcap_{k \in \nsets, \, i \in \iint{\ell_1+1,\ell}} \{ \tS_{i,k} \geq t\} \cap \{N = k\}}\\
  \label{eq:premier_bounce_1}
  &=  \{\tT < \tau_{\infty}(Z)\} \cap \defEns{  \bigcap_{  i \in \iint{\ell_1+1,\ell}} \defEns{ E_{i,0} \geq \int_{0}^t \lambda_i\parenthese{s, Z_s} \rmd s } }\eqsp.
\end{align}

In addition, since $\{ \tT = \plusinfty\} = \{\tT \geq \tau_{\infty}(Z)\}$, we have
\begin{align}
  \nonumber
  \{\tT \geq \tau_{\infty}(Z) \vee t \}  = \{\tT = \plusinfty\} &= \bigcap_{i \in \iint{\ell_1+1,\ell}, \, k \in \nsets} \{ \tS_{i,k} \geq R_k \} \\
  \label{eq:premier_bounce_2}
                                                         & =    \bigcap_{  i \in \iint{\ell_1+1,\ell}} \defEns{ E_{i,0} \geq \int_{0}^{\tau_{\infty}(Z)} \lambda_i\parenthese{s, Z_s} \rmd s }  \eqsp.
\end{align}
Combining \eqref{eq:premier_bounce_1} and \eqref{eq:premier_bounce_2} and since $(E_{i,0})_{i \in \iint{\ell_1+1,\ell_2}}$ is independent of $(Z_s)_{s \geq 0}$, we get
\begin{equation*}
  \proba{\tT \geq t } = \proba{\bigcap_{  i \in \iint{\ell_1+1,\ell}} \defEns{ E_{i,0} \geq \int_{0}^{\tau_{\infty}(Z) \wedge t } \lambda_i\parenthese{s, Z_s} \rmd s }} = \expe{\exp\parenthese{- \sum_{i=\ell_1+1}^{\ell} \int_0^{\tau_{\infty}(Z)\wedge t}  \lambda_i\parenthese{s, Z_s} \rmd s   }}\eqsp,
\end{equation*}
which concludes the proof.
\end{proof}

%% file: superposition.tex
\section{Superposition and splitting of jump mechanisms}
\label{sec:superposition}
We now introduce a  tool to deal with PDMP:  superposition and splitting of jump mechanisms. \tcr{This is a natural generalization of the fact that the sum of two independent Poisson processes with jump rates $\bar{\lambda}_1,\bar{\lambda}_2 : \rset_+ \to \rset_+$, is a Poisson process associated with the jump rate $\bar{\lambda}_1+\bar{\lambda}_2$. % It gives some flexibility in the representation of a given PDMP, different representations having different advantages (one may want a constant jump rate; or no fantom jumps; only one jump mechanism;  or to distinguish different type of jumps, for instance to design particular couplings, as in \cite{DurmusGuillinMonmarche:bouncy}).
}
\begin{theorem}\label{prop:superposition}
Let $(\lambda_{i,1},Q_{i,1})_{i\in\cco 1,\ell_1\ccf}$ and $(\lambda_{i,2},Q_{i,2})_{i\in\cco 1,\ell_2\ccf}$ be two families of jump mechanisms on $\msm$. Suppose that for all $(t,x,\msa) \in \rset_+\times \msm\times \mcb{\msm}$,
\begin{eqnarray}\label{eq:prop-superposition}
\sum_{i=1}^{\ell_1} \lambda_{i,1}(t,x)(  Q_{i,1}(t,x,\msa) - \updelta_x(\msa)) & = & \sum_{i=1}^{\ell_2} \lambda_{i,2}(t,x)(  Q_{i,2}(t,x,\msa) - \updelta_x(\msa)) \eqsp.
\end{eqnarray}
Then, for all differential flow $\varphi$ and initial distribution $\mu_0\in\mcp(\msm)$,
\[\PDMP( \varphi,(\lambda_{i,1},Q_{i,1})_{i\in \cco 1,{\ell_1}\ccf},\mu_0)=\PDMP( \varphi,(\lambda_{i,2},Q_{i,2})_{i\in \cco 1,{\ell_2}\ccf},\mu_0)\eqsp.\] 
\end{theorem}

%For $k\geqslant 1$, we say that, at time $S_k$ (resp. $W_k$), the process $Y$ (resp. $Z$) has made a jump of type $i_k$ (resp. $j_k$), or equivalently that $S_k$ (resp. $W_k$) is a jump time of type $i_k$ (resp. $j_k$). If there exist $j\in \cco 1,n\ccf$ such that $Q_t(x,\cdot) = \updelta_x$ for all $(t,x)\in\rset_+\times \msm$, we say that the jumps of type $j$ are phantom jumps. It appears clearly in the construction of $Z$ that if the jumps of type $n$ are phantom jumps, then the process $Z'$ defined similarly to $Z$ but only with the jump rates and kernels $(\lambda^{(i)},Q^{(i)})_{i\in\cco 1,n-1\ccf}$ has the same distribution on $\mrd(\rset_+,\msm \cup \{\infty\})$ as $Z$.

If $(\lambda_{i},Q_{i})_{i\in\iint{1,\ell}}$ is a family of jump mechanisms, we define  the associated minimal jump mechanism $(\lambda_\mrm,Q_\mrm)$ and the associated total jump mechanism $(\lambda_\rmt,Q_\rmt)$  for all $(t,x,\msa)\in\rset_+\times\msm\times \mcb{\msm}$, by 
\begin{equation}
  \label{eq:def_min_total_jump_mecha}
\begin{aligned}
\lambda_\mrm(t,x) & =  \sum_{i=1}^\ell \lambda_i(t,x) Q_i( t,x,\msm\setminus\{x\})\\
Q_\mrm(t,x,\msa) & =   \begin{cases}
 \lambda_\mrm^{-1}(t,x)  \sum_{i=1}^\ell \lambda_i(t,x) Q_i( t,x,\msa \setminus\{x\} ) & \text{if } \lambda_\mrm(t,x) \neq 0,\\
\updelta_x(\msa)  & \text{otherwise.}
\end{cases}   \\
\lambda_\rmt(t,x) & =  \sum_{i=1}^\ell \lambda_i(t,x) \\
  Q_\rmt(t,x,\msa) & =
                     \begin{cases}
 \lambda_\rmt^{-1}(t,x)  \sum_{i=1}^\ell \lambda_i(t,x) Q_i( t,x,\msa  ) & \text{if } \lambda_\rmt(t,x) \neq 0,\\
\updelta_x(\msa)  & \text{otherwise} \eqsp.
\end{cases}   
\end{aligned}
\end{equation}
The jump mechanism $(\lambda_\mrm,Q_\mrm)$ is minimal in the sense
that if $\lambda_\mrm(t,x) \neq 0$, for $t \in \rset_+$, $x \in \msm$,
then $Q_\mrm(t,x,\{x\}) = 0$. As a consequence, if $(X_t)_{t \geq 0}$ is a PDMP with
characteristics $(\varphi, \lambda_\mrm,Q_\mrm)$ and jump times $S_k$, $k\in\mathbb N$,
 then almost surely $X_{S_{k+1}} \neq \varphi_{S_{k},S_{k+1}}(X_{S_k})$,
 and therefore $(X_t)_{t \geq 0}$ has no fantom jumps.

For all
$(t,x,\msa)\in\rset_+\times\msm\times \mcb{\msm}$ and
$i\in\cco 1,\tcr{\ell}\ccf$,
%\begin{eqnarray*}
%\lambda_m(t,x) (  Q_m(t,x,\msa) - \updelta_x(\msa)) & = & \sum_{i=1}^n \lambda_i(t,x)(  Q_i( t,x,\msa \setminus \{x\} )  - \updelta_x(\msa)  Q_i( t,x,\msm \setminus \{x\} ) )
%\end{eqnarray*}
\[Q_i( t,x,\msa \setminus \{x\} )  - \updelta_x(\msa)  Q_i( t,x,\msm \setminus \{x\} ) \ = \ Q_i( t,x,\msa)  - \updelta_x(\msa) \eqsp .\]
\tcr{This means that, if $(\lambda_{i,1},Q_{i,1})_{i \in \iint{1,\ell_1}}$ and $(\lambda_{i,2},Q_{i,2})_{i \in \iint{1,\ell_2}}$ are two families of jump mechanisms satisfying \eqref{eq:prop-superposition}, they have the same associated minimal jump mechanism. Therefore, the statement of \Cref{prop:superposition} is equivalent to prove that, for all family $(\lambda_{i},Q_{i})_{i\in\iint{1,\ell}}$, the associated minimal and total jump mechanisms are such that
\[\PDMP(\varphi,(\lambda_{i},Q_{i})_{i \in \iint{1,\ell}},\mu_0) = \PDMP(\varphi,\lambda_{\mrt},Q_\mrt,\mu_0)=\PDMP(\varphi,\lambda_{\mrm},Q_\mrm,\mu_0).\]
 We first show the first equality in the following lemma.}

\begin{lemma}\label{lem:superposition1}
%Let $(X_t)_{t\geqslant0}$ be a PDMP with characteristics $(\varphi,(\lambda_{i},Q_{i})_{i\in\iint{1,\ell}})$  and let $(\lambda_\mrt,Q_\mrt)$ be the associated total jump mechanism. Then a PDMP with characteristics $(\varphi,\lambda_\mrt,Q_\mrt)$ has the same distribution as   $(X_t)_{t\geqslant0}$.
  Let $(\lambda_\mrt,Q_\mrt)$ be the total jump mechanism associated to   $(\lambda_{i},Q_{i})_{i\in\iint{1,\ell}}$. Then, for all flow $\varphi$ and  $\mu_0\in\mcp(\msm)$, $\PDMP(\varphi,(\lambda_{i},Q_{i})_{i\in\iint{1,\ell}},\mu_0) = \PDMP(\varphi,\lambda_\mrt,Q_\mrt,\mu_0)$.
\end{lemma}
\begin{proof}
  Let $(X_t)_{t\geqslant0}$ be a PDMP with characteristics $(\varphi,(\lambda_{i},Q_{i})_{i\in\iint{1,\ell}})$ and initial distribution $\mu_0$ defined by
\Cref{const:1}, and  $(X_k',S_k)_{k\in\mathbb N}$ be its embedded chain. Since the process is completely determined by its embedded chain, and by the Markov property, it is sufficient to prove that the Markov kernel of $(X_k',S_k)_{k\in\mathbb N}$  is equal to the Markov kernel of the embedded chain associated to a PDMP with  characteristics $(\varphi,\lambda_\mrt,Q_\mrt)$. %In other words, for $k\in\mathbb N$, we consider the conditional distribution of $(X_{k+1}',S_{k+1})$ conditionally to $(X_k',S_k)$. 
%First, for all $t\geqslant 0$,
%\begin{eqnarray*}
%\mathbb P\po S_{k+1} > t\ |\ X_k',S_k) & = & \prod_{j=1}^\ell \mathbb P \po E_{j,k+1} > \int_{0}^t \lambda_j\po S_k+s, \varphi_{S_k,S_k+s}(X_{k}')) \rmd s)
                                              %  \end{eqnarray*}

 Summing out \eqref{eq:carac_chaine_markov_1} over $j\in\iint{1,\ell}$, we get for all $t\geqslant 0$ and $\msa\in\mcb{\msm}$
 \begin{align}
  \nonumber &\CPP{X'_{k+1}\in \msa, S_{k+1} \leq t}{\mcf_k}{} 
  \\
  \nonumber
  &= \1_{\msm}(X'_k) \int_{S_k}^{t}   \sum_{j=1}^\ell  Q_j(s,\varphi_{S_k,s}(X_k'),\msa)  \lambda_j(s,\varphi_{S_k,s}(X_k')) \, \exp\parentheseDeux{-\sum_{i=1}^\ell \int_{S_k}^s \,  \lambda_i(u,\varphi_{S_k,u}(X_k'))  \rmd u} \rmd s \\
    \nonumber
  &= \1_{\msm}(X'_k) \int_{S_k}^{t} Q_\rmt(s,\varphi_{S_k,s}(X_k'),\msa)\lambda_\mrt(s,\varphi_{S_k,s}(X_k')) \, \exp\parentheseDeux{-  \int_{S_k}^s  \,  \lambda_\mrt(u,\varphi_{S_k,u}(X_k'))\rmd u }\rmd s \eqsp,
\end{align}
which concludes, since from \eqref{eq:carac_chaine_markov_1} this is exactly the Markov kernel of the embedded chain associated to a PDMP with  characteristics $(\varphi,\lambda_\mrt,Q_\mrt)$ defined by \eqref{eq:def_min_total_jump_mecha}.
\end{proof}

Before showing \Cref{prop:superposition}, we need the following technical lemma. In the sequel, $\Id$ denotes the identity Markov kernel, defined by $\Id(t,x,\msa) = \updelta_{x}(\msa)$ for all $t\geqslant 0$, $x\in\msm$, $\msa\in\mcb{\msm}$. The following lemma gives a rigorous proof of the intuitive idea that adding fantom  jumps does not change the distribution of the process.

\begin{lemma}\label{lem:superposition2}
For any characteristics $(\varphi,\lambda,Q)$, jump rate $\lambda':\msm \to \rset_+$ and $\mu_0\in\mcp(\msm)$,
\[\PDMP(\varphi,\lambda,Q,\mu_0) = \PDMP(\varphi,\{(\lambda,Q),(\lambda',\Id)\},\mu_0)\eqsp.\]
%Let $(X_t)_{t\geqslant0}$ be a PDMP with characteristics $(\varphi,\lambda,Q)$ and $\lambda'$ be any jump rate. Then the PDMP  with characteristics $(\varphi,(\lambda,Q),(\lambda',\Id))$ has the same law as   $(X_t)_{t\geqslant0}$.
\end{lemma}

\begin{proof}
%We consider $(X_t)_{t\geqslant0}\sim\PDMP (\varphi,(\lambda,Q),(\lambda',Id),\mu_0)$ defined from \rv~$W_0$ and $ (E_{1,k},U_{1,k})_{k\in\nset}$ by Construction 1, with embedded chain $(X_k',T_k)_{k\in\mathbb N}$. Let $ (E_{2,k},U_{2,k})_{k\in\nset}$ be an independent copy of $ (E_{1,k},U_{1,k})_{k\in\nset}$. Let $S_0 = 0$ and, for $j\in\{1,2\}$, $N_{j,0}=0$ and $H_{j,0} = E_{j,0}$ and suppose that $(S_k,(N_{j,k},H_{j,k})_{j\in\{1,2\}})$ has been defined for $k\in\nset$, with $S_k<\infty$. Then, set
%\[S_{2,k+1} \ = \ \inf\ensemble{t\geqslant S_k}{\ H_{2,k} > \int_{S_k}^{t\wedge \tau_\infty(X)} \lambda'\po s,X_s)\rmd s }\eqsp,\qquad
%S_{k+1} \ = \ S_{2,k+1} \wedge \inf\ensemble{T_i>S_k}{i\in\nset}\eqsp.\]
%If $S_{k+1} = +\infty$, set $S_m=+\infty$ for all $m>k+1$. If $S_{k+1} = S_{2,k+1} < +\infty$, set $N_{2,k+1} = N_{2,k} +1$ and $H_{2,k+1}   =  E_{2,N_{2,k+1}}$. If $S_{k+1} \neq S_{2,k+1}  $ and $S_{k+1}< +\infty$, set $N_{2,k+1} = N_{2,k} $ and
%\[H_{2,k+1}  =  H_{2,k} - \int_{S_k}^{S_{k+1}} \lambda'\po  s, X_s) \rmd s \eqsp.\]
%Finally, for all $k\in\nset$ such that $S_k<\infty$ and all $t\in[S_k,S_{k+1})$, set $Y_k' = X_{S_k}$ and $Y_t = \varphi_{S_k,t}(X_{S_k})$.

We consider $(Y_t)_{t\geqslant0}$ a PDMP with characteristics $(\varphi,\{(\lambda,Q),(\lambda',\Id)\})$ and initial distribution $\mu_0$ defined from \rvs~$W_0$ and $ (E_{j,k},U_{j,k})_{j\in\{1,2\},k\in\nset}$ by \Cref{const:2}, and its embedded chain $(Y_k',\tS_k,\tI_k)_{k\in\mathbb N}$. Let $R_0=0$ and, for $k\geqslant1$, let $R_{k} = \inf\ensembleLigne{\tS_i>R_{k-1}}{ i\in\mathbb N,\ \tI_i = 1}$ be the $k^{\text{th}}$ jump of type 1 (i.e. associated with the jump  mechanism $(\lambda,Q)$). For $i\geqslant 1$ such that $\tI_i=2$, the $i^{\text{th}}$ jump is a fantom one, i.e. $Y_{\tS_i} =\varphi_{\tS_{i-1},\tS_i}(Y_{\tS_{i-1}})$. By the flow property $\varphi_{s,u}\circ\varphi_{t,s} = \varphi_{t,u}$, this implies that $Y_t=\varphi_{R_k,t}(Y_{R_k})$ for all $k\in\mathbb N$ and $t\in[R_k,R_{k+1}\wedge \tau_\infty(Y))$.

If $k\in \nset$ is such that $R_k<\infty$, then
$\ensembleLigne{\varphi_{R_k,t}(Y_{R_k})}{t\in[R_k,(a+R_k)\wedge
  R_{k+1}]}$ is a compact set of $\msm$, on which $\lambda'$ is
bounded (as a locally bounded function), for any $a >0$. Hence, there
cannot be an infinite number of jump of second type between times
$R_k$ and $(a+R_k)\wedge R_{k+1}$, for any $a >0$. In particular,
necessarily,
$\sup\ensemble{R_k}{k\in\nset} = \sup\ensembleLigne{\tS_k}{k\in\nset}
= \tau_\infty(Y)$.  Indeed, for any $k \in \nset$, $\tS_k \leq R_k$ by
definition and if $R_{k}< \plusinfty$, for any $a >0$, almost surely,
there exists $n_k \in \nset$ such that $\tS_{n_k} \geq (R_k+a) \wedge R_{k+1}$. As a consequence,
\begin{equation}
  \label{eq:proof:lem:superposition2_1}
\text{ $Y_t=\varphi_{R_k,t}(Y_{R_k})$ holds for all $k\in\nset$ and
$t\in[R_k,R_{k+1})$ on $\{ R_k <
\plusinfty\}$ } \eqsp.
\end{equation}
    %     Indeed, if that wasn't the case, then this would mean that $\tau_\infty(Y)) < +\infty$ and $R_k = +\infty$ for all $k$ large enough. Then, considering $k\in\nset$ such that $R_k<+\infty$ and $R_{k+1}=+\infty$, it would mean that there would be an infinite number of jumps of second type on the bounded interval $[R_k,\tau_\infty(Y))$. For all $t$ in this interval, $Y_t = \varphi_{R_k,t}(Y_{R_k})$, so that $\ensemble{Y_t}{t\in[R_k,\tau_\infty(Y))}$ is a compact set of $\msm$, on which $\lambda'$ is bounded, as a locally bounded function. This would be in contradiction with

Define for all $k \in \nsets$, $N_k = \inf\{ i > N_{k-1} \, : \, \tI_i = 1\}$, setting $N_0 = 0$.
\Cref{const:2} implies by a straightforward induction  that for any $k\in\nset$, on $\{ R_k < \plusinfty\}\cap \{N_{k+1} < \plusinfty\}$,
\[R_{k+1} = \inf\ensemble{t\geqslant R_k}{\ E_{1,k+1} < \int_{R_k}^t \lambda( s,\varphi_{R_k,s}(Y_{R_k}) )\rmd s }\eqsp.\]
In addition,  for all $k\in\nset$,
\begin{align*}
&  \{ R_k < \plusinfty\}\cap \{N_{k+1} = \plusinfty\}  \\
\qquad &=   \{ R_k < \plusinfty\}\cap \{N_{k+1} = \plusinfty\} \cap \defEns{\bigcap_{i \geq N_k} \defEns{\tH_{1,i+1} > \int_{S_{i}}^t \lambda( s,\varphi_{S_{i},s}(Y_{S_{i}}) )\rmd s  }} \\
\qquad   & =  \{ R_k < \plusinfty\}\cap \{N_{k+1} = \plusinfty\} \cap \defEns{E_{1,k+1}  \geq  \int_{R_k}^{\plusinfty} \lambda( s,\varphi_{R_k,s}(Y_{R_k}) )\rmd s  } \eqsp.
\end{align*}
Therefore, on $\{ R_k < \plusinfty\}\cap \{N_{k+1} = \plusinfty\}$,
\[R_{k+1} = \plusinfty=  \inf\ensemble{t\geqslant R_k}{\ E_{1,k+1} < \int_{R_k}^t \lambda( s,\varphi_{R_k,s}(Y_{R_k}) )\rmd s }\eqsp,\]
and, if $R_{k+1}<\infty$, $Y_{R_{k+1}} = \bfG_{1} (  R_{k+1}, \varphi_{R_{k},R_{k+1}}(Y_{R_{k}}),U_{1,k}) $.

Therefore, denoting $Z_k' = Y_{R_k}$ for all $k\in\mathbb N$, then $(Z_k',R_k)$ is the embedded chain associated to a PDMP $(Z_t)_{t\geqslant0}$ with characteristics $(\varphi,\lambda,Q)$ and constructed with the \rvs~$W_0$ and $( E_{1,k},U_{1,k})_{k\geqslant 0}$ (through either Construction 1 or 2, since there is only one jump mechanism so that both coincides). Finally, for all $k\in\mathbb N$ and $t\in[R_k,R_{k+1})$, $Z_t = \varphi_{R_k,t}(Z_k') = Y_t$ by \eqref{eq:proof:lem:superposition2_1}, which concludes the proof.
\end{proof}

\begin{proof}[Proof of \Cref{prop:superposition}]
  As previously mentioned, to show \Cref{prop:superposition}, it is sufficient to prove that, for all differential flow $\varphi$, initial distribution $\mu_0$  and family of jump mechanisms  $(\lambda_i,Q_i)_{i \in \iint{1,\ell}}$,
   $$\PDMP(\varphi,(\lambda_{i},Q_{i})_{i \in \iint{1,\ell}},\mu_0) = \PDMP(\varphi,\lambda_{\mrt},Q_\mrt,\mu_0)=\PDMP(\varphi,\lambda_{\mrm},Q_\mrm,\mu_0)$$ holds, where $(\lambda_{\mrt},Q_{\mrt})$ and $(\lambda_{\mrm},Q_{\mrm})$ are the associated total and minimal jump mechanisms  defined by \eqref{eq:def_min_total_jump_mecha}. The first identity is given by \Cref{lem:superposition1}, therefore it remains to show the second one.

 By \Cref{lem:superposition2}, since $\lambda_{\mrt} - \lambda_{\mrm}$ is by definition a jump rate (\ie~a positive and locally bounded measurable function), we get for all differential flow $\varphi$ and initial distribution $\mu_0$ that
  \begin{equation*}
%    \label{eq:2}
    \PDMP(\varphi,\lambda_{\mrm},Q_{\mrm},\mu_0) =     \PDMP(\varphi,\{(\lambda_{\mrm},Q_{\mrm}),(\lambda_{\mrt} - \lambda_{\mrm},\Id)\},\mu_0) \eqsp.
  \end{equation*}
  The proof is concluded upon noting that the total jump mechanism associated with $\{(\lambda_{\mrm},Q_{\mrm}),(\lambda_{\mrt} - \lambda_{\mrm},\Id)\}$ is equal to $(\lambda_{\mrt}, Q_{\mrt})$ and using \Cref{lem:superposition1} again.
%   Let $(Z_t)_{t\geqslant0}$ be a PDMP  with characteristics $(\varphi,(\lambda_\mrm,Q_\mrm),(\lambda_\mrt-\lambda_\mrm,\Id))$. According to \Cref{lem:superposition2}, it has the same law as a PDMP with  characteristics $(\varphi,\lambda_\mrm,Q_\mrm)$. Finally, denoting $(\lambda_\mrt',Q_\mrt')$ the total jump mechanism associated to $(Z_t)_{t\geqslant0}$, we remark that for all $t\geqslant0$, $x\in\msm$ and $\msa\in\mcb{\msm}$,
% \[\lambda_\rmt'(t,x) \ =\ \lambda_\mrm(t,x) + \lambda_\mrt(t,x)-\lambda_\mrm(t,x) \ =\  \lambda_\mrt(t,x)\eqsp,\]
% and
% \begin{eqnarray*}
% \lambda_\mrt'(t,x) Q_\mrt'(t,x,\msa) & = & \sum_{i=1}^n \lambda_i(t,x) Q_i( t,x,\msa \setminus \{x\} ) + \updelta_x(\msa) \sum_{i=1}^n \lambda_i(t,x) ( 1 - Q_i( t,x,\msm \setminus \{x\} ))\\
% & = & \sum_{i=1}^n \lambda_i(t,x)(  Q_i( t,x,\msa \setminus \{x\} ) + \updelta_x(\msa)    Q_i( t,x,\{x\}  ))\\
% & = & \sum_{i=1}^n \lambda_i(t,x)  Q_i( t,x,\msa \cap  ) \ = \ \lambda_\mrt(t,x) Q_\mrt(t,x,\msa)\eqsp.
% \end{eqnarray*}
% In other words, $(X_t)_{t\geqslant0}$ and $(Z_t)_{t\geqslant0}$ have the same total jump mechanism and, according to \Cref{lem:superposition1}, have the same law. This finally means that $(X_t)_{t\geqslant0}\sim\PDMP(\varphi,\lambda_\mrm,Q_\mrm,\mu_0)$, which concludes.
\end{proof}

\begin{example*}
  \label{ex:superposition:bps}
  By \Cref{prop:superposition}, the BPS process defined in \Cref{ex:BPS} is a PDMP with characteristics $(\varphi,\lambda,Q)$, where for all $t \geq 0$, $(x,y)\in \rset^d \times \msy$, and $\msa \in \mcbb(\rset^d \times \msy)$, $\varphi_t(x,y) = (x+ty,y)$,
  \begin{equation}
    \label{eq:def_lambda_total_bps}
  \lambda(x,y) = \ps{\nabla U(x)}{y}_+ + \rate \eqsp,  
  \end{equation}
   and
  \begin{equation*}
  Q((x,y),\msa)  =  \lambda^{-1}(x,y)\defEns{\ps{\nabla U(x)}{y}_+ \updelta_{(x,\Refl(x,y))}(\msa) + \rate (\updelta_x \otimes \loiy)(\msa)} \eqsp,
  \end{equation*}
where $\Refl$ is defined in \eqref{eq:def_refl} with $g = \nabla U$. 
\end{example*}

% \begin{example}
%   \label{ex:superposition:bps}
%   By \Cref{lem:superposition}, the BPS process admits as local characteristics $(\varphi,\lambda,Q)$, where for all $t \geq 0$, $(x,y)\in \rset^d \times \msy$, and $\msa \in \mcbb(\rset^d \times \msy)$, $\varphi_t(x,y) = (x+ty,y)$, 
%   \begin{align}
%     \label{eq:def_lambda_total_bps} 
%     \lambda(x,y)& = \ps{\nabla U(x)}{y} + \rate \eqsp, \\
%     \nonumber
%   Q((x,y),\msa) & =  \lambda^{-1}(x,y)\defEns{\ps{\nabla U(x)}{y} \updelta_{(x,\Refl(x,y))}(\msa) + \rate (\updelta_x \otimes \loiy)(\msa)} \eqsp,
% \end{align}
% where $\Refl$ is defined in \eqref{eq:def_refl} with $g = \nabla U$. 
% \end{example}

%%% Local Variables:
%%% mode: latex
%%% TeX-master: "main"
%%% End:

%% file: explosion.tex
\section{Non-explosion}
\label{sec:non-explose}

It is generally easier to prove that a given particular PDMP is non-explosive than to provide good general conditions that ensure non-explosion for PDMPs. Nevertheless, we give here two results on that topic that will prove useful in the rest of this work, and may be of interest in other situations.

\begin{proposition}\label{prop:non-explosion1}
Let $(X_t)_{t\geqslant0}$ be a PDMP with characteristics $ ( \varphi,(\lambda_i,Q_i)_{i\in \iint{1,\ell}})$ and initial distribution $\mu_0$ given by \Cref{const:1} for some \rvs~$W_0$ and $((E_{j,k})_{j\in\iintLigne{1,\ell}},U_k)_{k\in \nset}$. For all $M>0$, let $(X_t^M)_{t\geqslant0}$ be a PDMP with characteristics $ ( \varphi,(M\wedge \lambda_i,Q_i)_{i\in \iint{1,\ell}})$ and initial distribution $\mu_0$ be given by \Cref{const:1} for the same \rvs~$W_0$ and $((E_{j,k})_{j\in\iintLigne{1,\ell}},U_k)_{k\in \nset}$ as $(X_t)_{t\geqslant 0}$.
  The process $(X_t)_{t\geqslant0}$ is non-explosive if and only if, for all $t\geqslant 0$,
  \begin{equation*}
    \lim_{M \to \plusinfty} \proba{X_s = X_s^M, \, \text{for all $s \in \ccint{0,t}$}} = 1 \eqsp.
  \end{equation*}
\end{proposition}
\begin{proof}
  Suppose that $(X_t)_{t\geqslant0}$ is non-explosive. Then for almost
  all $\omega\in\Omega$, the process only jumps a finite number of
  time between times $0$ and $t$, so that
  $\ensembleLigne{X_s}{s\in[0,t]}$ is a compact set of $\msm$. Since
  the rate jumps are locally bounded,
  $\lambda_{\star}(\omega) = \sup\ensembleLigne{\sum_{j=1}^\ell
    \lambda_j(s,X_s)}{s\in[0,t]}$ is finite for almost all
  $\omega\in\Omega$. We get for all $\omega \in \Omega$ and
  $M > \lambda_{\star}(\omega)$, by definition that
  $\sup_{s \in \ccint{0,t}} \abs{X_s - X_s^M} = 0$ and therefore, for
  almost all $\omega \in \Omega$,
  $\lim_{M \to \plusinfty}\1_{\{0\}}(\sup_{s \in \ccint{0,t}}
  \absLigne{X_s - X_s^M}) = 1$. Thus, since
  $\{ X_s = X_s^M, \text{ for all }s \in \ccint{0,t} \} = \{ \sup_{s
  \in \ccint{0,t}} \absLigne{X_s - X_s^M} = 0\}$, the proof is
concluded using the Lebesgue dominated convergence theorem.

Now, to prove the converse, remark that $(X_t^M)_{t\geqslant 0}$ is non-explosive for all $M>0$, since its jump rates are bounded. In particular, $X_t^M \in \msm$ for all $t\geqslant 0$, so that
\[\mathbb P( \tau_\infty(X)>t)  = \mathbb P( X_t \in \msm) \geq  \mathbb P(X_t = X_t^M)\]
for all $M>0$. The conclusion then follows taking $M\to \plusinfty$.
\end{proof}

In particular, \Cref{prop:non-explosion1} implies that, if $(\varphi,(\lambda_i,Q_i)_{i \in \iint{1,\ell}})$ is non explosive, denoting by $(P_{s,t})_{t\geqslant s\geqslant 0}$ and $(P_{s,t}^M)_{t\geqslant s\geqslant 0}$ the Markov semi-group associated to characteristics  $(\varphi,(\lambda_i,Q_i)_{i \in \iint{1,\ell}})$ and  $(\varphi,(M \wedge \lambda_i,Q_i)_{i \in \iint{1,\ell}})$, $M >0$ then, since for all $\mu_0\in\mcp(\msm)$ and all $t\geqslant  0$,
\begin{equation*}%\label{eq:taux-tronque}
\tvnorm{\mu_0 P_{0,t} -\mu_0 P_{0,t}^M } =  \sup_{\msa\in\mcb{\msm}} \abs{ \mathbb P(X_t \in \msa) - \mathbb P(X_t^M \in \msa) }\  \leqslant 2 \mathbb P(X_t \neq X_t^M)\eqsp,
\end{equation*}
we get $\lim_{M \to \plusinfty} \tvnorm{\mu_0 P_{0,t} -\mu_0 P_{0,t}^M } = 0$.

The second result concerning non-explosion of PDMPs is the following:

\begin{proposition}\label{prop:non-explosion}
  Let  $( \varphi,(\lambda_i,Q_i)_{i\in \iint{1,\ell}})$ be homogeneous characteristics with $\ell>1$. Assume that the characteristics $( \varphi,(\lambda_i,Q_i)_{i\in \iint{1,\ell-1}})$ are non-explosive and  $\norm{\lambda_\ell}_\infty < \plusinfty$. Then  $( \varphi,(\lambda_i,Q_i)_{i\in \iint{1,\ell}})$ are non-explosive as well.
\end{proposition} 

\begin{proof}
  Let $\mu_0\in\mcp(\msm)$ and $(Y_t)_{t \geq 0}$ be a PDMP with
  characteristics $( \varphi,(\lambda_i,Q_i)_{i\in \iint{1,\ell}})$ and
  initial distribution $\mu_0$ given by \Cref{const:2} based on \rvs~$W_0$ and
  $(E_{j,k},U_{j,k})_{j\in\iintLigne{1,\ell},k\in \nset}$. Let
  $(Y_k',\tS_k,\tI_k)_{k\in\nset}$ be the corresponding embedded chain.
  First, consider the decomposition 
  \begin{equation}
\label{eq:definition_true_decomposition}
   \proba{\sup_{n \in \nset} \tS_n < \plusinfty} = \proba{\sup_{n \in \nset} \tS_n < \plusinfty, \sup_{n \in \nsets}\tN_{\ell,n} = \plusinfty}
+\proba{\sup_{n \in \nset} \tS_n < \plusinfty, \sup_{n \in \nsets}\tN_{\ell,n} < \plusinfty} \eqsp.
  \end{equation}
Let us show that both terms of the right-hand-side are equal to $0$.

%%%
% N^{\ell} pseudo inverse of $N_{\ell}$
Define recursively $(N^{(\ell)}_{n})_{n \in \nset}$ by $N^{(\ell)}_{0}=0$ and for all $n \in \nset$, $  N^{(\ell)}_{n+1} = \inf \ensembleLigne{k > N^{(\ell)}_{n}}{\tI_k = \ell} $
% \begin{equation*}
% \eqsp.
% \end{equation*}
Note that for any $n,k \in\nset$ and 
\begin{equation}
\label{eq:pseudo_inv_nn}
  \text{$ N^{(\ell)}_{n} = k$ if and only if $\tN_{\ell,k} = n $} \eqsp, \quad \tS_{N^{(\ell)}_n} = \tS_{\ell,N^{(\ell)}_n} \eqsp.
\end{equation}

Then, on $\{\sup_{n \in \nsets}\tN_{\ell,n} = \plusinfty\}$, for all $n
\in \nset$, $N^{(\ell)}_{n} < \plusinfty$ almost surely. Hence, on $\{\sup_{n \in \nsets}\tN_{n,\ell} = \plusinfty\}$, for all $n \in
\nsets$, by definition  we have almost surely 
\begin{align*}
\sup_{k \in \nsets} \tS_k   \geq
 \sum_{k =1}^{n}\defEns{ \tS_{N_k^{(\ell)}} - \tS_{N_{k-1}^{(\ell)}}} \geq \sum_{i=1}^n (E_{\ell,i}/\norm{\lambda_{\ell}}_{\infty}) \eqsp,
\end{align*}
where the last inequality follows from the bound on $\{N^{(\ell)}_k < \plusinfty \}$,
\begin{align*}
  \tS_{N^{(\ell)}_k} =   \tS_{\ell,N^{(\ell)}_k} &= \inf\ensemble{t > \tS_{N^{(\ell)}_k-1}}{H_{\ell,N^{(\ell)}_k} < \int_{\tS_{N^{(\ell)}_k-1}}^{t} \lambda_{\ell}(X_s) \rmd s}\\
  & =  \inf\ensemble{t > \tS_{N^{(\ell)}_{k-1}}}{E_{\ell,k} < \int_{\tS_{N^{(\ell)}_{k-1}}}^{t} \lambda_{\ell}(X_s) \rmd s} \geq  \tS_{N^{(\ell)}_{k-1}} + E_{\ell,k}/\norm{\lambda_{\ell}}_{\infty} \eqsp.
\end{align*}
Therefore by the law of large number,
\begin{equation*}
  \proba{\sup_{n \in \nset} \tS_n < \plusinfty, \sup_{n \in \nsets}\tN_{\ell,n} = \plusinfty} \leq \inf_{n \in \nsets} \proba{ \sum_{i=1}^n (E_{\ell,i}/\norm{\lambda_{\ell}}_{\infty}) < \plusinfty} = 0 \eqsp.
\end{equation*}

We bound now the second term in
\eqref{eq:definition_true_decomposition}. Let $k \in \nsets$. Note
that by \Cref{const:2}, on $\{\sup_{n \in \nset} \tN_{\ell,n} = k\}$,
for all $ i \in \iint{1,k-1}$,
$\defEnsLigne{X_t \, : \, t \in
  \cointLigne{\tS_{N^{(\ell)}_i},\tS_{N^{(\ell)}_{i+1}}}}$ is a PDMP
    with characteristics
    $(\varphi,(\lambda_i,Q_i)_{i \in \iint{1,\ell-1}})$ with initial
    data $X_{\tS_{N^{(\ell)}_i}}$. Then, by \eqref{eq:pseudo_inv_nn}, the Markov property
      and an immediate induction using that
      $(\varphi,(\lambda_i,Q_i)_{i \in \iint{1,\ell-1}})$ is non
      explosive, we have that on
      $\{\sup_{n \in \nsets} \tN_{\ell,n} = k\}$, almost surely
      $X_{\tS_{N^{(\ell)}_i}} \in \msm$, for $i \in \iint{1,k}$. The proof is then concluded since, using that $(\varphi,(\lambda_i,Q_i)_{i \in \iint{1,\ell-1}})$ is non explosive, \Cref{const:2} and the Markov property again  on $\{\sup_{n \in \nset} \tN_{\ell,n} = k\}$,  $\defEnsLigne{X_t \, : \, t \in \cointLigne{\tS_{N^{(\ell)}_k},\plusinfty}}$ is a PDMP with characteristics $(\varphi,(\lambda_i,Q_i)_{i \in \iint{1,\ell-1}})$ started at $X_{\tS_{N^{(\ell)}_k}}$. 
\end{proof}

% Remark that, since a non-homogeneous PDMP $(X_t)_{t\geqslant0}$ with characteristics $(\varphi,\lambda,Q)$ can always be seen as the homogeneous PDMP $(X_t,t)_{t\geqslant 0}$ on $\msm\times\rset_+$ with characteristics $(\bar{\varphi},\lambda,\bar{Q})$
% defined for all $t,s \in \rset_+$, $x \in \msm$ and $\msa \in \mcbb(\msm \times \rset_+)$ by
% \begin{equation*}
%   \bar{\varphi}_{s}((x,t)) = (\varphi_{t,t+s}(x),t+s)  \eqsp, \qquad \qquad \bar{Q}((x,t),\msa) = \int_{\msm\times \rset_+} \1_{\msa}((y,u))  Q(x,\rmd y) \otimes \updelta_t(\rmd u) \eqsp,
% \end{equation*}
%  this result extends to non-homogeneous characteristics with the following assumptions: denote by $\sigma_t$ the shift of time $t$, i.e. $\sigma_t f(s) = f(t+s)$ for all $s\in\rset_+$ and function $f$. Suppose that the characteristics $( \sigma_t \varphi,(\sigma_t\lambda_i,\sigma_tQ_i)_{i\in \iint{1,\ell-1}})$ are non-explosive for all $t\in\rset_+$, and that there exists a jump rate $\lambda_*$ on $\rset_+$ such that for all $(x,t)\in\msm\times\rset_+$, $\lambda_\ell(t,x) \leqslant \lambda_*(t)$. Then $( \varphi,(\lambda_i,Q_i)_{i\in \iint{1,\ell}})$ are non-explosive. An application of this remark is that the BPS with non-constant temperature studied in \cite{DurmusGuillinMonmarche:bouncy} is non-explosive. Nevertheless, here, as an illustration of \Cref{prop:non-explosion}, we will only tackle the constant-temperature BPS:

Let us come back to our main example.
\begin{example*}
\begin{proposition}\label{prop:non-explosion_BPS}
The BPS process defined in  \Cref{ex:BPS} is non-explosive for any  initial distribution.
\end{proposition}
\begin{proof}
Using the notation of \Cref{ex:BPS}, consider $(\bX_t,\bY_t)_{t\geqslant 0}$ the PDMP on $\rset^d\times \msy$ with characteristics $(\varphi,\lambda_1,Q_1)$ and initial condition $(\bX_0,\bY_0) = (x,y)\in\msm$ defined by \Cref{const:1} from an i.i.d. sequence of \rvs~$(E_k,U_k)_{k\in\nset}$, associated with the sequence of jump times $(S_k)_{ k \in \nset}$. Note that almost surely $\norm{\bY_t}=\norm{y}$ for all $t\in[ 0,\sup_{n \in \nset} S_n)$, so that  $\norm{\bX_t - x} \leqslant t\norm{y}$ and
\[\lambda_1(\bX_t,\bY_t) \leqslant C(t)= \norm{y}\sup\ensemble{\norm{\na U(x')}}{ x'\in\rset^d,\ \norm{x' - x} \leqslant t\norm{y} }\eqsp.\]
Therefore on $\{\sup_{n\in \nset} S_n < \plusinfty\}$, by \Cref{const:1} almost surely there exists $C \geq 0$ ($C=C(a)$ for some $a >0$) such that 
\begin{equation*}
%  \label{eq:definition_true_decomposition_bound_TDeux}
S_{n+1} -S_n \geq
E_n/(C+1) \text{ for all $n \in \nsets$ } \eqsp.  
\end{equation*}
Then, we have on $ \{ \sup_{n \in \nset} S_n <
\plusinfty\}$, that almost surely there exists $C \geq 0$  such that
\begin{equation*}
 \plusinfty > \sup_{n \in \nset} S_n = \sum_{n \in \nset} \defEns{S_{n+1} - S_n}  \geq \sum_{n \in \nsets} \{E_n/(C+1)\} \eqsp.
\end{equation*}
As a result,
\begin{align*}
  \proba{  \sup_{n \in \nset} S_n <
\plusinfty} &\leq \proba{\bigcup_{\ell \in \nsets}  \defEns{\sum_{n \in \nsets} E_n/(\ell+1) < \plusinfty}}  \leq 0 \eqsp.
\end{align*}
It follows that the PDMP with characteristics $(\varphi,\lambda_1,Q_1)$ is non-explosive. By \Cref{prop:non-explosion}, the BPS is non-explosive.
\end{proof}
\end{example*}

%%% Local Variables:
%%% mode: latex
%%% TeX-master: "main"
%%% End:

%% file: synchrone.tex
\section{Comparison of PDMP via Synchronous coupling}
\label{sec:synchrone}

\tcr{A synchronous coupling is a coupling between two PDMPs that ensures that, as long as possible, the two corresponding processes jump at the same times and to the same location. This construction yields total variation estimates between the marginal distributions of different PDMPs sharing the same differential  flow.}

The result of this section is crucial in many aspects: first, it gives stability estimates with respect to the jump rates and the underlying Markov kernel for a modification of a PDMP for example for an approximate thinning procedure. Second, \tcr{it gives a way to design smooth (in some sense) approximations of a given (non-smooth) PDMP,} which will be an essential tools to verify assumptions for the BPS in the following sections. The main goal of this section is to prove the following:
\begin{theorem}
\label{prop:synchrone}
Let $(P_{s,t}^{(1)})_{t\geqslant s\geqslant0}$ and $(P_{s,t}^{(2)})_{t\geqslant s\geqslant0}$ be two non-explosive PDMP semigroups with characteristics $( \varphi,\lambda^{(1)},Q^{(1)})$ and $( \varphi,\lambda^{(2)},Q^{(2)})$ respectively. Suppose that there exists a measurable $g :\rset_+ \to \rset_+$ satisfying for all $t\geqslant 0$, 
\begin{equation}
\label{eq:prop:synchrone_1}
  g(t) \geq  \sup_{\substack{x \in \msm,\\ \msa \in\mcb{\msm}}} \{
  \lambda^{(1)}(t,x)\wedge \lambda^{(2)}(t,x)
  \absLigne{Q^{(1)}(t,x,\msa)-Q^{(2)}(t,x,\msa)}\} + \sup_{x \in \msm} \absLigne{\lambda^{(1)}(t,x)-\lambda^{(2)}(t,x)} \eqsp
\end{equation}
or alternatively
\begin{equation}
\label{eq:prop:synchrone_(2)}
  g(t) \geq  \sup_{\substack{x \in \msm,\\ \msa \in\mcb{\msm}}}|\lambda^{(1)}(t,x)( Q^{(1)}(t,x,\msa)-\updelta_x(\msa)) - \lambda^{(2)}(t,x)( Q^{(2)}(t,x,\msa)-\updelta_x(\msa))|  \eqsp.
\end{equation}
Then for all $t\geqslant 0$ and $x \in \msm$, $ \tvnorm{\updelta_x P_{0,t}^{(1)}   - \updelta_x P_{0,t}^{(2)} } \leq  2 \defEnsLigne{ 1 - \exp\parentheseLigne{-\int_0^t g(s)\rmd s} }$.
% \begin{equation*}
% \tstyle \tvnorm{\updelta_x P_{0,t}^{(1)}   - \updelta_x P_{0,t}^{(2)} } \leq  2 \defEnsLigne{ 1 - \exp\parentheseLigne{-\int_0^t g(s)\rmd s} }\eqsp.
% \end{equation*}
% Then there exists  a transference plane $\xi$ between the laws of $X$ and $Y$ such that, if   $(Z_t,Z_t')_{t\geqslant0}$ has law $\xi$,
%\begin{eqnarray*}
%\mathbb P ( Z_s = Z_s'\text{ for all }s\in[0,t]) & \geqslant & \exp ( - \int_0^t g(s) \rmd s)\tvnorm{\mu_1-\mu_(2)} \eqsp.
%\end{eqnarray*}
\end{theorem}

\begin{remark}
  \label{rem:sync}
  Note that if $(\lambda^{(i)})_{i=1,2}$ and $(Q^i)_{i=1,2}$ are two locally bounded jump rates and Markov kernels respectively, then
  \begin{multline}
    \label{eq:rem_sync_1}
    \sup_{\substack{x \in \msm,\\ \msa \in\mcb{\msm}}} \{
  \lambda^{(1)}(t,x)\wedge \lambda^{(2)}(t,x)
  \abs{Q^{(1)}(t,x,\msa)-Q^{(2)}(t,x,\msa)}\} \\
  \leq \sup_{\substack{x \in \msm,\\ \msa \in\mcb{\msm}}} \{
  \abs{   \lambda^{(1)}(t,x) Q^{(1)}(t,x,\msa)- \lambda^{(2)}(t,x) Q^{(2)}(t,x,\msa)}\} \eqsp. 
\end{multline}
Indeed, let $x \in \msm$ and $\msa \in \mcbb(\msm)$. Without loss of generality, we can assume that $\lambda^{(1)}(t,x) \geq \lambda^{(2)}(t,x)$. If $Q^{(1)}(t,x,\msa) \geq Q^{(2)}(t,x,\msa)$, then we have
\begin{multline*}
  \lambda^{(1)}(t,x)\wedge \lambda^{(2)}(t,x)
  \abs{Q^{(1)}(t,x,\msa)-Q^{(2)}(t,x,\msa)}\\
  \leq  \lambda^{(1)}(t,x) Q^{(1)}(t,x,\msa)- \lambda^{(2)}(t,x) Q^{(2)}(t,x,\msa) = \abs{   \lambda^{(1)}(t,x) Q^{(1)}(t,x,\msa)- \lambda^{(2)}(t,x) Q^{(2)}(t,x,\msa)} \eqsp. 
\end{multline*}
Otherwise $Q^{(1)}(t,x,\msa^{\compl}) \geq Q^{(2)}(t,x,\msa^{\compl})$ and we get
\begin{equation*}
  \lambda^{(1)}(t,x)\wedge \lambda^{(2)}(t,x)
  \abs{Q^{(1)}(t,x,\msa)-Q^{(2)}(t,x,\msa)}  \leq   \abs{   \lambda^{(1)}(t,x) Q^{(1)}(t,x,\msa^{\compl})- \lambda^{(2)}(t,x) Q^{(2)}(t,x,\msa^{\compl})} \eqsp. 
\end{equation*}
Therefore, for all $x\in \msm$ and $\msa \in \mcbb(\msm)$,
\begin{equation*}
    \lambda^{(1)}(t,x)\wedge \lambda^{(2)}(t,x)
  \abs{Q^{(1)}(t,x,\msa)-Q^{(2)}(t,x,\msa)}  \leq \sup_{\tilde \msa \in \mcbb(\msm)} \abs{   \lambda^{(1)}(t,x) Q^{(1)}(t,x,\tilde \msa)- \lambda^{(2)}(t,x) Q^{(2)}(t,x, \tilde \msa)} \eqsp,
\end{equation*}
which implies \eqref{eq:rem_sync_1}. Therefore, to establish that \eqref{eq:prop:synchrone_1}, it is sufficient to show that there exists a measurable function $g :\rset_+ \to \rset_+$ such that for all $t \in \rset_+$,
\begin{equation}
  \label{eq:rem_syn_2}
  g(t) \geq 2 \sup_{\substack{x \in \msm,\\ \msa \in\mcb{\msm}}} \{
  \absLigne{   \lambda^{(1)}(t,x) Q^{(1)}(t,x,\msa)- \lambda^{(2)}(t,x) Q^{(2)}(t,x,\msa)}\} \eqsp,
\end{equation}

Conversely, we easily get for all $t \in \rset_+$,
\begin{multline}
\label{eq:rem_syn_3}
  \sup_{\substack{x \in \msm,\\ \msa \in\mcb{\msm}}} \{
  \absLigne{   \lambda^{(1)}(t,x) Q^{(1)}(t,x,\msa)- \lambda^{(2)}(t,x) Q^{(2)}(t,x,\msa)}\}
  \\ \leq 
    \sup_{\substack{x \in \msm,\\ \msa \in\mcb{\msm}}} \{
  \lambda^{(1)}(t,x)\wedge \lambda^{(2)}(t,x)
  \absLigne{Q^{(1)}(t,x,\msa)-Q^{(2)}(t,x,\msa)}\} + \sup_{x \in \msm}\absLigne{\lambda^{(1)}(t,x) -\lambda^{(2)}(t,x)}\eqsp.
\end{multline}

Therefore \eqref{eq:prop:synchrone_1} and \eqref{eq:rem_syn_2} are essentially equivalent up to a factor $2$.
\end{remark}
The proof of \Cref{prop:synchrone} relies on the construction of the Markovian
synchronous coupling between $(P_{s,t}^{(1)})_{t\geq s \geq 0}$ and
$(P_{s,t}^{(2)})_{t\geq s \geq 0}$. More precisely, we want to construct a PDMP $(X_t,Y_t)_{t \geq 0}$ on $\msm^2$ starting from $(x,y) \in \msm^2$ such that the distributions of $(X_t)_{t \geq 0}$ and $(Y_t)_{t \geq 0}$ are $\PDMP(\varphi,\lambda^{(1)},Q^{(1)},\updelta_x)$ and  $\PDMP(\varphi,\lambda^{(2)},Q^{(2)},\updelta_y)$ respectively. In the case where $x=y$, the synchronous coupling attempts to keep $X_s = Y_s$ for all $s\in[0,t]$ by ensuring that, as much as possible, both processes jump at the same time and, when they do, jump as much as possible to the same point. Let us give the formal definition of $(X_t,Y_t)_{t \geq 0}$.

First, by \cite[Corollary 5.22]{villani:2009} or \cite[Theorem 19.1.12]{douc:moulines:priouret:2018}, there exists a jump
kernel $K_0$ on $\msm^2$, such that
for all $t \in \rset_+$ and $(x,y)\in \msm^2$, $K_0(t,(x,y),\cdot)$ is
an optimal transference plane of $Q^{(1)}(t,x,\cdot)$ and
$Q^{(2)}(t,y,\cdot)$ for the total variation, \ie~ for any $\msa \in\mcbb(\msm)$, $K_0(t,(x,y),\msa \times \msm) = Q^{(1)}(x,\msa)$, $K_0(t,(x,y), \msm \times \msa ) = Q^{(2)}(y,\msa)$ and 
\begin{equation}\label{eq:synchrone-k0}
2 K_0(t,(x,y),\Delta_\msm^c ) \ =  \ \tvnorm{Q^{(1)}(t,x,\cdot)-Q^{(2)}(t,y,\cdot)}\eqsp. 
\end{equation}
Define, for $i=0,1,2$  and $j=1,2$ the jump rate $r_i$ and the jump kernel $K_j$ on $\msm^2$ as follows:
for $t \in \rset_+$,$(x,y)\in \msm^2$ and $\msa,\msb\in\mcb{\msm}$, 
\begin{equation}
\label{eq:synchrone-rest}
\begin{aligned}
r_0(t,(x,y)) &= \lambda^{(1)}(t,x)\wedge \lambda^{(2)}(t,y) \eqsp, &  \\
r_1(t,(x,y))& = (\lambda^{(1)}(t,x)- \lambda^{(2)}(t,y))_+\eqsp,   \qquad    K_1(t,(x,y),\msa\times\msb)& =  Q^{(1)}(t,x,\msa) \updelta_y(\msb)\eqsp,\\
r_2(t,(x,y))& = (\lambda^{(2)}(t,x)- \lambda^{(1)}(t,y))_+ \eqsp,  \qquad   K_2(t,(x,y),\msa\times\msb)& =  \updelta_x(\msa)  Q^{(2)}(t,y,\msb)\eqsp.  
\end{aligned}
\end{equation}
\[ \begin{array}{rclcrcl}
\end{array}\]
Let $\varphi^{\otimes }$  be the flow on $\msm^2$  defined for all $t \in \rset_+$ and $(x,y)\in \msm^2$ by
\[\varphi^{\otimes }(t,(x,y)) = ( \varphi(t,x),\varphi(t,y))\eqsp.\]

\begin{lemma}\label{lem:verif-synchrone}
Let $(x,y)\in\msm^2$ and $(X_t,Y_t)_{t\geqslant 0}$ be a PDMP on $\msm^2$ with initial distribution $\updelta_{(x,y)}$ and characteristics $(\varphi^{\otimes },(r_i,K_i)_{i\in\iintLigne{0,2}})$. Suppose that it is non explosive. Then $(X_t)_{t\geqslant 0}$  and $(Y_t)_{t\geqslant 0}$ have distributions $\PDMP(\varphi,\lambda^{(1)},Q^{(1)},\updelta_x)$ and  $\PDMP(\varphi,\lambda^{(2)},Q^{(2)},\updelta_y)$ respectively.
\end{lemma}
As a consequence,  $(X_t,Y_t)_{t\geqslant 0}$ is referred to as  a synchronous coupling of $(\updelta_xP_{0,t}^{(1)})_{t\geqslant 0}$ and $(\updelta_xP_{0,t}^{(2)})_{t\geqslant 0}$ .
 
 \begin{proof}
   We only show the result for $(X_t)_{t\geqslant0}$, the case for $(Y_t)_{t\geqslant0}$ being similar. Consider the Markov kernel on $\msm^2 \times \mcb{\msm^2}$ defined for all $(x,y)\in \msm^2$ and $\msa \in \mcb{\msm^2}$ by
   \begin{equation*}
     \tilde K(t,(x,y),\msa)  =
     \begin{cases}
&              \updelta_{(x,y)}(\msa) \qquad \qquad \qquad  \text{if } r_0(x,y)+r_1(x,y)  = 0\\
 &      \fraca{r_0(x,y)K_0(t,(x,y),\msa)+r_1(x,y)K_1(t,(x,y),\msa)}{r_0(x,y)+r_1(x,y)} \qquad 
  \text{otherwise} \eqsp.
     \end{cases}
   \end{equation*}
 \tcr{   By \Cref{prop:superposition} and since $\lambda^{(1)}= r_0+r_1$,  we have that $\PDMP(\varphi^{\otimes },\lambda^{(1)},\tilde K,r_2,K_2, \updelta_x \otimes \updelta_y) = \PDMP(\varphi^{\otimes },$ $(r_i,K_i)_{i \in \iint{0,2}},  \updelta_x \otimes \updelta_y)$. As a consequence,  we can assume that $(X_t,Y_t)_{t\geqslant 0}$ is a PDMP with characteristics $(\varphi^{\otimes },\lambda^{(1)},\tilde K,r_2,K_2)$ with initial distribution $\updelta_{(x,y)}$ defined by \Cref{const:2} based on some \rvs~$(E_{j,k},$ $U_{j,k})_{j\in\iintLigne{1,2},k\in \nset}$. We are now able to show that $(X_t)_{t \geq 0}$ is distributed according to  $\PDMP(\varphi,\lambda^{(1)},Q^{(1)},\updelta_x)$.} % Then $(X_t,Y_t)_{t\geqslant 0}\sim\PDMP(\varphi^{\otimes },(r_i,K_i)_{i\in\iintLigne{0,2}},\updelta_{(x,x)})$ from \Cref{prop:superposition}.
% Let us check that, in that case,  $(X_t)_{t\geqslant 0}$ is the PDMP with characteristics $(\varphi,r_0+r_1,\tilde K)$ and initial law $\updelta_{x}$ obtained through Construction 1 from the r.v.  $(E_{1,k},U_{1,k})_{k\in \nset}$.
 Let $((X_i',Y_i'),\tS_i,\tI_i)_{i\in\nset}$ be the embedded chain associated with $(X_t,Y_t)_{t\geqslant 0}$. Set $R_0 = 0$, $J_0=0$ and, for $k\in\nset$,
 \[ R_{k+1}  =  \inf\ensemble{\tS_i > R_k}{i\in\nset,\ \tI_i=1}\eqsp, \qquad J_{k+1} =  \inf\ensemble{i > J_k}{\tI_i=1}  \eqsp. \]
 Note that $R_{k} = S_{J_k}$ if $J_k < \plusinfty$ and the process being non-explosive, $\sup_{n \in \nset} \tS_n = \plusinfty$, so that $\sup_{n \in \nset} R_n = \plusinfty$.
 For all $k\in\mathbb N$, set $\bX_k = X_{R_k}$  if $R_k<\infty$ and $\bX_k = \infty$ otherwise. 
 For $k \in \nset$, on $\{R_k < \plusinfty\}$, by  definition of $K_2$ and $\varphi^\otimes$, an easy induction
 using   implies that for any $j\in\iint{J_k,J_{k+1}-1}$, $t \in \coint{S_{j},S_{j+1}}$, 
$X_t =\varphi_{S_{J_j},t}(\bX_{S_{J_j}}) =  \varphi_{R_k,t}(\bX_k)$, and 
 \begin{equation*}
%   \label{eq:proof_lem_synch_eq_H}
   H_{1,j+1} = H_{j,1} - \int_{S_j}^{S_{j+1}} \lambda^{(1)}( s, \varphi_{S_j,s}(X_{S_j})) \rmd s  = E_{k,1} -  \int_{R_k}^{S_{j+1}} \lambda^{(1)}( s, \varphi_{R_k,s}(\bX_k)) \rmd s \eqsp.
 \end{equation*}
Therefore for all $k\in\nset$ such that $R_k<\infty$ and all $t\in [R_k,R_{k+1})$, $X_t = \varphi_{R_k,t}(\bX_k)$ and
 \begin{equation}\label{eq:synchrone1}
R_{k+1}   =  \inf \ensemble{ t\geqslant R_k}{ E_{1,k+1} < \int_{R_k}^t \lambda^{(1)}( s, \varphi_{R_k,s}(\bX_k)) \rmd s }\eqsp,
\end{equation}
Finally, if $R_{k+1}<\infty$, denoting by $\bfG_1$, $\bfG_2$ the representation of $\tilde K$ and $K_2$ respectively, used in the construction of the process, for any $j \in \iint{J_k,J_{k+1}-2}$
\begin{equation}\label{eq:synchrone2}
  \begin{aligned}
(X_{S_{j+1}},Y_{S_{j+1}}) & =   \bfG_2( S_{j+1},( \varphi_{R_k,S_{j+1}}(\bX_k),Y_{S_{j}}),U_{2,j+1}) \\
(X_{R_{k+1}},Y_{R_{k+1}})  &= \bfG_1( R_{k+1},( \varphi_{R_k,R_{k+1}}(\bX_k),Y_{S_{J_{k+1}-1}}),U_{1,k+1})\eqsp.    
  \end{aligned}
\end{equation}
It follows that $U_{1,k+1}$ and $E_{1,k+1}$ are independent of $Y_{S_{J_{k+1}-1}}$ and $\mcf_k$ where $(\mcf_{\tilde{k}})_{\tilde{k} \in \nset}$ is the filtration associated with
$(\bX_k,R_k)_{k \in \nset}$. Using that $E_{1,k+1}$ and $U_{1,k+1}$ are independent, \eqref{eq:synchrone1} and \eqref{eq:synchrone2} yield, for all $k\in\nset$, $t\geqslant R_k$ and $\msa\in\mcb{\msm}$,
\begin{align*}
  &\CPP{\bX_{k+1}\in \msa, R_{k+1} \leq t}{\mcf_k}{} \\
  & =  \1_{\msm}(\bX_k) \int_{R_k}^{t}  \CPP{  \bfG_1(s,(\bX_{k+1},Y_{S_{J_{k+1}-1}} )),U_{1,k+1}) \in \msa}{\mcf_k\vee \sigma(Y_{S_{J_{k+1}-1}})}{}  \\
  & \qquad \qquad \qquad \qquad \qquad \times \lambda^{(1)}(s,\varphi_{R_k,s}(\bX_k)) \, \exp\defEns{-  \int_{R_k}^s   \lambda^{(1)}(u,\varphi_{R_k,u}(\bX_k)) \rmd u} \rmd s \eqsp.
\end{align*}
 Now, for all $k\in\nset$,  $t\geqslant 0$, $(x,y)\in\msm^2$ and $\msa\in\mcb{\msm}$, 
\begin{align}\label{eq:synchrone3}
&  \lambda^{(1)}(t,x) \mathbb P ( \bfG_1(t,(x,y),U_{1,k+1}) \in \msa\times \msm ) \\
  \notag & \qquad \qquad = r_0( t,(x,y)) K_0( t,(x,y),\msa\times\msm) + r_1( t,(x,y)) K_1( t,(x,y),\msa\times\msm)\\
\notag  & \qquad \qquad  = ( r_0( t,(x,y)) + r_1( t,(x,y))) Q^{(1)}(t,x,\msa) = \lambda^{(1)}(t,x) Q^{(1)}(t,x,\msa)\eqsp.
\end{align}
for all $k\in\nset$, $t\geqslant R_k$ and $\msa\in\mcb{\msm}$, implying
\begin{multline*}
  \CPP{\bX_{k+1}\in \msa, R_{k+1} \leq t}{\mcf_k}{} \\
   = \1_{\msm}(\bX_k) \int_{R_k}^{t}  Q^{(1)}(s,x,\msa)  \lambda^{(1)}(s,\varphi_{R_k,s}(\bX_k)) \, \exp\defEns{-  \int_{R_k}^s   \lambda^{(1)}(u,\varphi_{R_k,u}(\bX_k)) \rmd u} \rmd s\eqsp.
\end{multline*}

 As a consequence, $(\bX_k,R_k)_{k\in\nset}$ is the embedded chain corresponding to a PDMP with characteristics $(\varphi,\lambda^{(1)},Q^{(1)})$. The fact that $X_t = \varphi_{R_k,t}(\bX_k)$ for all $k\in\nset$ such that $R_k<\infty$ and all $t\in \coint{R_k,R_{k+1}}$ concludes the proof.
 \end{proof}

 \begin{proof}[Proof of \Cref{prop:synchrone}]
\begin{enumerate}[leftmargin=*,wide=0pt,label=(\Alph*)]   
\item We first consider the case where \eqref{eq:prop:synchrone_1} holds. 
The proof is divided in two main steps. In the first one, we assume that $\lambda^{(1)}$ and $\lambda^{(2)}$ are uniformly bounded and the second one is the extension of the first result in the case where the jump rates are not bounded. 

\begin{enumerate}[leftmargin=*,wide=0pt,label=(\arabic*)]
\item Assume  that $\normLigne{\lambda^{(i)}}_{\infty} \leq M$, for
  $i=1,2$ and some $M >0$.  Let $(P^{(1)}_{s,t})_{ t \geq s \geq 0}$ and
  $(P^{(2)}_{s,t})_{ t \geq s \geq 0}$ be two PDMPs semigroups with
  characteristics $(\varphi,Q^{(1)},\lambda^{(1)})$ and
  $(\varphi,Q^{(2)},\lambda^{(2)})$ respectively. Since $\lambda^{(1)}$ and $\lambda^{(2)}$
  are uniformly bounded,  the synchronous characteristics  
  $(\varphi,(r_i,K_i)_{i \in \iint{0,2}})$, where
  $(r_i,K_i)_{i \in \iint{0,2}}$ is defined in \eqref{eq:synchrone-k0} and \eqref{eq:synchrone-rest},
  are non explosive. From \Cref{lem:verif-synchrone}, the
  synchronous coupling $(X_t,Y_t)_{t \geq 0}$ defined above is a Markov coupling between
  these two semigroups. Then, by
  characterisation of the total variation distance for any $x \in \msm$ and $t \geq 0$,
  \begin{equation*}
    \tvnorm{\updelta_x P_{0,t}^{(1)}- \updelta_x P_{0,t}^{(2)}}  \leq 2 \proba{X_t=Y_t} \leq 2 \proba{X_s=Y_s \, \text{ for any $s \in \ccint{0,t}$}} \eqsp,
  \end{equation*}
  where $(X_t,Y_t)_{t \geq 0}$ has distribution $\PDMP(\varphi^{\otimes},(r_i,K_i)_{i \in\iint{0,2}},\updelta_x^{\otimes 2})$. 
 However, to do so, we consider a process $(X_t,Y_t)_{t \geq 0}$ based on different characteristics from
  $(\varphi^{\otimes} ,(r_i,K_i)_{i \in \iint{0,2}})$ but still having the expected distribution using  \Cref{prop:superposition}.

Define the Markov kernels, for all $(x,y) \in \msm^2$ and $(\msa,\msb) \in \mathcal{B}(\msm)^{2}$, dropping the subscript $\msm$ for the diagonal $\Delta$
\begin{align*}
  K_{0,\Delta} (t,(x,y),\msa \times \msb) &=
                        \begin{cases}
                         \dfrac{ K_0(t,(x,y),\msa \times \msb \cap \Delta)}{ K_0(t,(x,y), \Delta) } & \text{ if }  K_0(t,(x,y), \Delta) \not = 0 \\
                          \updelta_{(x,y)}(\msa \times \msb) & \text{ otherwise} \eqsp,
                        \end{cases}
\\
  K_{0,\neq} (t,(x,y),\msa \times \msb) &=
                      \begin{cases}
                       \dfrac{ K_0(t,(x,y),\msa \times \msb \cap \Delta^{\comp})}{ K_0(t,(x,y), \Delta^{\comp})}  & \text{ if }  K_0(t,(x,y), \Delta^{\comp}) \not = 0 \\
                        \updelta_{(x,y)}(\msa \times \msb) & \text{ otherwise} \eqsp,
                      \end{cases}
\\
  K_{1,\Delta}(t,(x,y),\msa \times \msb) & = 
                         \begin{cases}
                         \dfrac{  K_1(t,(x,y),(\msa\cap\{y\} )\times \msb) }{K_1(t,(x,y),\{y\}\times \msm)} & \text{ if }  K_1(t,(x,y),\{y\}\times \msm) \not = 0 \\
                          \updelta_{(x,y)}(\msa\times \msb) & \text{ otherwise} \eqsp,
                        \end{cases}
  \\
    K_{1,\neq}(t,(x,y),\msa \times \msb) & = 
                         \begin{cases}
                        \dfrac{   K_1(t,(x,y),(\msa\setminus\{y\} )\times \msm) }{ K_1(t,(x,y),(\msm \setminus\{y\}) \times \msm)}  & \text{ if }  K_1(t,(x,y), (\msm \setminus \{y\}) \times \msm) \not = 0 \\
                          \updelta_{(x,y)}(\msa \times \msb) & \text{ otherwise} \eqsp,
                        \end{cases}
  \\
  K_{2,\Delta}(t,(x,y),\msa\times \msb) & = 
                         \begin{cases}
                         \dfrac{  K_2(t,(x,y),\msa \times( \msb\cap\{x\}) )}{ K_2(t,(x,y),\msm \times \{x\})} & \text{ if }  K_2(t,(x,y),\msm \times  \{x\}) \not = 0 \\
                          \updelta_{(x,y)}(\msa \times \msb) & \text{ otherwise} \eqsp,
                         \end{cases}
\\
    K_{2,\neq}(t,(x,y),\msa\times \msb) & = 
                         \begin{cases}
                         \dfrac{  K_2(t,(x,y),\msa \times  (\msb\setminus\{x\})) }{ K_2(t,(x,y),\msm \times  (\msm \setminus\{x\}))} & \text{ if }  K_2(t,(x,y), \msm \times (\msm \setminus \{x\})) \not = 0 \\
                          \updelta_{(x,y)}(\msa \times \msb) & \text{ otherwise} \eqsp.
                        \end{cases}
\end{align*}
Define also the rate jumps for all $(x,y) \in \msm^2$ by 
\begin{align*}
&  \ttr_{0,\Delta}(t,(x,y)) = K_0(t,(x,y),\Delta) \ttr_0(x,y) \eqsp, \, \ttr_{0,\neq}(t,x,y) =  K_0(t,(x,y),\Delta^{\comp}) \ttr_0(x,y) \eqsp, \\
&  \ttr_{1,\Delta}(t,(x,y)) = K_1(t,(x,y),\{y\} \times \msm) \ttr_1(t,(x,y)) \eqsp, \, \ttr_{1,\neq}(t,,(x,y)) =  K_1(t,(x,y),\{y\}^{\compl} \times \msm)  \ttr_1(t,(x,y))\eqsp,\\
 & \ttr_{2,\Delta}(t,,(x,y)) = K_2(t,(x,y), \msm \times \{x\}) \ttr_2(t,(x,y)) \eqsp, \, \ttr_{2,\neq}(t,(x,y)) =  K_2(t,(x,y), \msm \times \{x\}^{\compl})  \ttr_2(t,(x,y)) \eqsp.
\end{align*}
By \Cref{prop:superposition}, for any initial distribution $\mu_0$ on
$\msm^2$,
$$\PDMP(\varphi^{\otimes},(r_i,K_i)_{i \in \iint{0,2}}, \mu_0) =
\PDMP(\varphi^{\otimes},(r_{i,\circ},K_{i,\circ})_{i \in
  \iint{0,2},\circ \in \{\Delta,\neq\}}, \mu_0)\eqsp.$$
Let
$(X_t,Y_t)_{t \geq 0}$ be a PDMP associated with the characteristics $(\varphi, (r_{i,\circ},K_{i,\circ})_{i \in
  \iint{0,2},\circ \in \{\Delta,\neq\}})$ with initial distribution $\updelta_x^{\otimes 2}$, $x \in \msm$, and let
$(S^{\neq,i}_n)_{n \in \nset ,\, i \in \iint{0,2}}$ be the jump times
associated with the jump rates
$r_{0,\neq},r_{1,\neq},r_{2,\neq}$ respectively.
By \Cref{lem:verif-synchrone} since $(X_t,Y_t)_{ t \geq 0}$ is non-explosive and \Cref{prop:premier_bounce}, we get
for all $t \geq 0$
\begin{align*}
&  \proba{X_t \not = Y_t}  \leq 1-\proba{\min_{i\in \iint{0,2}} S_1^{\neq,i} \geq t, X_s = Y_s \, \text{ for all $s\in \ccint{0,t}$}} \\
  & \leq 1 - \expe{ \exp\parentheseDeux{-\int_{0}^t \defEns{r_{0,\neq}(s,(\bar X_s, \bar  X_s)) + r_{1,\neq}(s,(\bar  X_s, \bar X_s)) + r_{2,\neq}(s,(\bar  X_s, \bar X_s))} \rmd s   }} \eqsp,
\end{align*}
where $(\bar X_s,\bar X_s)_{s \geq 0}$ is a PDMP with characteristics $(\tvarphi, (r_{i,\Delta},K_{i,\Delta})_{i \in \iint{0,2}})$ starting at $(x,x)$.
This result concludes the proof since by definition, \eqref{eq:synchrone-k0}, \eqref{eq:synchrone-rest} and \eqref{eq:prop:synchrone_1}, for all $y \in \msm$ and $s \geq 0$,
we have 
\begin{equation*}
r_{0,\neq}(s,(y,y)) + r_{1,\neq}(s,(y,y)) + r_{2,\neq}(s,(y,y)) \leq g(s) \eqsp.
\end{equation*}
\item In the case where $\lambda^{(1)}$ and $\lambda^{(2)}$ are not uniformly bounded, consider for all $M >0$ the two semi-groups $(P_{s,t}^{(1),M})_{t\geq s \geq 0}$ and  $(P_{s,t}^{(2),M})_{t\geq s \geq 0}$ associated with the characteristics
  $(\varphi, \lambda^{(1)} \wedge M, Q^{(1)})$ and   $(\varphi, \lambda^{(2)} \wedge M, Q^{(2)})$ respectively.
  Then for all $M >0$  the triangle inequality yields
  \begin{equation*}
    \tvnorm{\updelta_x P_{0,t}^{(1)}-\updelta_x P_{0,t}^{(2)}} \leq  \tvnorm{\updelta_x P_{0,t}^{(1)}- \updelta_x P_{0,t}^{(1),M}} +  \tvnorm{\updelta_x P_{0,t}^{(1),M}-\updelta_x P_{0,t}^{(2),M}} +  \tvnorm{\updelta_x P_{0,t}^{(2),M}- \updelta_x P_{0,t}^{(2)}} \eqsp.
  \end{equation*}
  Using \Cref{prop:non-explosion1} and the assumption that the semi-groups we consider are non-explosive,
  \begin{equation}
  \label{eq:proof_fin_synch_1}
    \tvnorm{\updelta_x P_{0,t}^{(1)}-\updelta_x P_{0,t}^{(2)}} \leq \limsup_{M \to \plusinfty}     \tvnorm{\updelta_x P_{0,t}^{(1),M}-\updelta_x P_{0,t}^{(2),M}} \eqsp.
  \end{equation}
On the other hand, by the first part of the proof for all $M >0$,
\begin{equation}
  \label{eq:proof_fin_synch_2}
    \tvnorm{\updelta_x P_{0,t}^{(1),M}-\updelta_x P_{0,t}^{(2),M}} \leq  2 \defEns{ 1 - \exp\parenthese{-\int_0^t g(s)\rmd s} } \eqsp,
  \end{equation}
where $g$ satisfies \eqref{eq:prop:synchrone_1}  since for all $t \in \rset_+$ and $M>0$,
\begin{multline*}
   g(t) \geq  \sup_{\substack{x \in \msm,\\ \msa \in\mcb{\msm}}} \{
M\wedge  \lambda^{(1)}(t,x)\wedge \lambda^{(2)}(t,x)
  \absLigne{Q^{(1)}(t,x,\msa)-Q^{(2)}(t,x,\msa)}\} \\+ \sup_{x \in \msm} \absLigne{ M \wedge \lambda^{(1)}(t,x)- M \wedge \lambda^{(2)}(t,x)} \eqsp.
\end{multline*}
Combining \eqref{eq:proof_fin_synch_1} and \eqref{eq:proof_fin_synch_2} concludes the proof.
\end{enumerate}
\item Let us finish the proof by assuming \eqref{eq:prop:synchrone_(2)} and showing that the conclusion of \Cref{prop:synchrone} still holds. Indeed by \Cref{prop:superposition}, for all initial distribution $\mu_0 \in \mcp(\msm)$, $\PDMP(\varphi,Q^{(1)},\lambda^{(1)},\mu_0) = \PDMP(\varphi,\tilde{Q}^{(1)},\tilde{\lambda}^{(1)},\mu_0)$ and $\PDMP(\varphi,Q^{(2)},\lambda^{(2)},\mu_0) = \PDMP(\varphi,\tilde{Q}^{(2)},\tilde{\lambda}^{(2)},\mu_0)$ where $\tilde{\lambda}^{(1)}  =\tilde{\lambda}^{(2)} = \lambda^{(1)} \vee \lambda^{(2)}$ and for $t \in \rset_+$, $x \in \msm$, $\msa \in \mcb{\msm}$ and $i=1,2$
\begin{equation*}
  \tilde{Q}^i(t,x,\msa) = \frac{\lambda^{(i)}(t,x)}{\lambda^{(1)}(t,x) \vee \lambda^{(2)}(t,x) } Q^i(t,x,\msa) + \parenthese{1-\frac{\lambda^{(i)}(t,x)}{\lambda^{(1)}(t,x) \vee \lambda^{(2)}(t,x) }} \updelta_x(\msa) \eqsp.
\end{equation*}
Therefore,
$(P_{s,t}^{(1)})_{t\geq s \geq 0}$ and $(P_{s,t}^{(2)})_{t\geq s \geq 0}$ are also associated with $(\varphi,\tilde{Q}^{(1)},\tilde{\lambda}^{(1)})$ and $(\varphi,\tilde{Q}^{(2)},\tilde{\lambda}^{(2)})$. Applying the case where $g$ is given by \Cref{eq:prop:synchrone_1} to these characteristics concludes.
\end{enumerate}
\end{proof}

%%% Local Variables:
%%% mode: latex
%%% TeX-master: "main"
%%% End:

%% file: generator.tex
\section{Generator}
\label{sec:generator}
From this section, only homogeneous processes are
considered. Nevertheless, some results below can be applied to inhomogeneous
PDMP since if $(X_t)_{t\geq 0}$ is such a process on $\msm$ with
characteristics $(\varphi,Q,\lambda)$, then the process
$(X_t,t)_{t \geq 0}$ is a homogeneous PDMP on
$(\msm \times \rset_+, \mcbb(\msm \times \rset_+))$ with
characteristics $(\bar{\varphi},\bar{Q},\lambda)$
defined for all $t,s \in \rset_+$, $x \in \msm$ and $\msa \in \mcbb(\msm \times \rset_+)$ by
\begin{equation*}
  \bar{\varphi}_{s}((x,t)) = (\varphi_{t,t+s}(x),t+s)  \eqsp, \qquad \qquad \bar{Q}((x,t),\msa) = \int_{\msm\times \rset_+} \1_{\msa}((y,u))  Q(x,\rmd y) \otimes \updelta_t(\rmd u) \eqsp.
\end{equation*}

This section is devoted to the introduction of the strong and extended generators of a non-explosive PDMP, \tcr{which will be a central tool for the study of invariant measures (see \Cref{theo:core_smooth_approx} and \Cref{lem:critere-invariant} below).}

Consider a homogeneous PDMP semigroup $(P_t)_{t \geq 0}$ with non-explosive characteristics $(\varphi,\lambda,Q)$.
Note that $(P_t)_{t \geq 0}$ is a contraction semigroup on $\mrb(\msm)$, \ie~for all $s,t \in \rset_+$, $P_{s+t} = P_t P_s$ and for all function $f \in \mrb(\msm)$, $\norm{P_tf}_{\infty} \leq  \norm{f}_{\infty}$. In addition, define the subset $\mrb_0(\msm) \subset \mrb(\msm)$ by
\begin{equation}
  \label{eq:def_mrb_0}
 \mrb_0(\msm) = \ensemble{f \in \mrb(\msm)}{\lim_{t \to 0}\norm{P_t f-f}_{\infty} = 0} \eqsp.
\end{equation}
By \cite[p.28-29]{davis:1993}, $\mrb_0(\msm)$ is a closed subspace of $\mrb(\msm)$ and a Banach space for the uniform norm. Then by definition, $(P_t)_{t \geq 0}$ is a strongly continuous semigroup on $\rmb_0(\msm)$, \ie~for all $f \in \mrb_0(\msm)$, $\lim_{t \to 0}\norm{P_t f-f}_{\infty} = 0$.

Define $(\sgenerator, \rmD(\sgenerator))$ the strong generator of $(P_t)_{t \geq 0}$ by
\begin{align*}
 \rmD(\sgenerator) &= \ensemble{f \in \mrb_0(\msm)}{\text{there exists $g: \msm \to \rset$ } \lim_{t \to 0} \norm{t^{-1}(P_tf-f)-g}_{\infty} = 0 } \eqsp,\\
 \sgenerator f &= g \text{ satisfying $\lim_{t \to 0} \norm{t^{-1}(P_tf-f)-g}_{\infty} = 0$ for all $f \in  \rmD(\sgenerator) $} \eqsp. 
\end{align*}
A subset $\msd \subset \mrD(\sgenerator)$ is a core of
$(\sgenerator,\mrD(\sgenerator))$ if the closure of the restriction of
$\sgenerator$ to $\msd$ is equal to $(\sgenerator,\mrD(\sgenerator))$.
The strong generator of $(P_t)_{t \geq 0}$ is a common tool to show that
a probability measure on $(\msm,\msb({\msm}))$ is invariant for
$(P_t)_{t \geq 0}$. Indeed by \cite[Proposition
9.2]{ethier:kurtz:1986}, $\mu$ is invariant measure for
$(P_t)_{t \geq 0}$ if only if for all $f \in \msd$, where $\msd$ is a
core for $(\sgenerator,\mrD(\sgenerator))$,
$\int_{\msm} \sgenerator f(x) \mu(\rmd x) = 0$.  Therefore, the strong generator
$(\sgenerator,\mrD(\sgenerator))$ is an essential tool to study
$(P_t)_{t \geq 0}$. Unfortunately, characterizing the domain
$\mrD(\sgenerator)$ is not possible in many cases of interest. In addition, while it
would be possible to only use a core of
$(\sgenerator,\mrD(\sgenerator))$, there \tcr{are} very few results giving
such a subset for PDMPs contrary to diffusion processes (see
\eg~\cite[Chapter 8]{ethier:kurtz:1986}). However, we will see in \tcr{
Section~\ref{eq:invariant_main}} that for a class of PDMPs, to show that a probability measure $\mu$ is
invariant, it is sufficient to show that for all
$f \in \mrc^1_c(\msm)$, $\int_{\msm} \generator f(x) \rmd \mu(x)=0$,
where $(\generator,\rmD(\generator))$ is the extended generator of $(P_t)_{t \geq 0}$, defined as follows. 

For $x\in\msm$, denote $\mathbb P_x$ the distribution $\PDMP(\varphi,\lambda,Q,\delta_x)$ on $\mrd(\rset_+,\msm )$ and $\mathbb{E}_x$ the corresponding expectation. Let $(\bX_t)_{t\geqslant 0}$ be the canonical process on $\mrd(\rset_+,\msm )$, defined by $\bX_t(\omega) = \omega_t$ for all $\omega \in \mrd(\rset_+,\msm )$, and let $(\mcf_t)_{t \geq 0}$ be its associated filtration. Let $\bS_0=0$ and, for $k\in\nset$, $\bS_{k+1} = \inf\{t>\bS_k\ : \ \bX_t \neq \varphi_{t-\bS_k}(\bX_{\bS_k})\}$ be its true jump times.
Define for all $t \in \rset_+$, $\bN_t = \sum_{k\in\nsets} \1_{\ccint{0,t}}(\bS_k)$ and consider the following assumption
\begin{assumption}
  \label{ass:poisson_proc_integ}
  For all $x \in \msm$ and $t \in \rset_+$, $\expeMarkov{x}{\bN_t}<\plusinfty$.
\end{assumption}
For all $t \geqslant 0$ and for all measurable functions $f,g : \msm \to \rset$, such that, for all $x\in\msm$,
$s \mapsto g(\bX_s)$ is $\mathbb P_x$-almost surely locally integrable, denote
\begin{equation}
\label{eq:def_martin_local_pdmp_generator}
  \martfg_t = f(\bX_t)-f(\bX_0)-\int_0^t g(\bX_s) \rmd s \eqsp.
\end{equation}
The (extended) generator and its domain $(\generator, \domain(\generator))$
associated with the semi-group $(P_t)_{t \geq 0}$ are defined as follows: $f \in
\domain(\generator)$ if there exists a measurable function $g : \msm \to \rset$
such that $(\martfg_t)_{t\geqslant0} $ is a local martingale under $\mathbb P_x$ for
all $x\in\msm$ and, for such a function, $\generator f = g$.
Despite its very formal definition, $(\generator,
\domain(\generator))$ associated with $(P_t)_{t \geq 0}$ can be easily
described. Indeed, under \Cref{ass:poisson_proc_integ},  \cite[Theorem 26.14]{davis:1993}  shows that $\domain(\generator) = \mse_1 \cap \mse_2$ where
\begin{equation*}
%  \label{eq:definition_esp_fun_abs_cont}
\mse_1 =\defEns{f \in \MeasFspace(\msm) \, : \, t \mapsto f(\varphi_t(x)) \text{ is absolutely continuous on $\rset_+$ for all $x\in\msm$} } \eqsp,
\end{equation*}
and $\mse_2$ is the set of measurable functions $f : \msm \to \rset$ such that there exists an increasing sequence of $(\mcf_t)_{t \geq 0}$-stopping time $(\sigma_n)_{n\geqslant 0}$, such that for all $x\in\msm$, $\lim_{n \to \plusinfty } \sigma_n = \plusinfty$ $\mathbb P_x$-almost surely and for all $n \in \nset$,
\begin{equation*}
%  \label{eq:condition_generator}
  \expeMarkov{x}{\sum_{k = 0}^{\plusinfty} \1_{\{\bS_{k+1} \leq \sigma_n\}} \abs{f(\bX_{\bS_{k+1}} ) - f\po\varphi_{\tcr{\bS_{k+1}}-\bS_k}(\bX_{\bS_k})\pf} }< \plusinfty \eqsp.
\end{equation*}
Then, for all $f \in  \domain(\generator)$ and $x\in\msm$,
\begin{equation}\label{eq:def_generateur}
\generator f(x) \ =\ \Ddir{\varphi} f(x) + \lambda(x) \po Qf(x) - f(x)\pf \eqsp,
\end{equation} 
where
\begin{equation*}
  \Ddir{\varphi} f(x) =
  \begin{cases}
    \lim_{t \, \downarrow \,  0} t^{-1}\defEns{f(\varphi_t(x)) - f(x)} \eqsp, &\text{ if this limit exists} \\
    0 & \text{ otherwise} \eqsp.
  \end{cases}  
\end{equation*}
In fact, in  \cite[Theorem 26.14]{davis:1993}, $Q$ is required to satisfy $Q(x,\{x\})=0$ for all $x\in\msm$, but if it is not the case, from \Cref{prop:superposition}, we can apply this result with the minimal jump rate associated to $(\lambda,Q)$ such as introduced in \Cref{sec:superposition}.

Note that $\rmD(\sgenerator) \subset \rmD(\generator)$ and for all $f \in \mrD(\sgenerator)$, $\generator f = \sgenerator f$,  since by \cite[Proposition 14.13]{davis:1993}, for all $f \in \mrD(\sgenerator)$, $(M^{f,\sgenerator f}_t)_{t \geq 0}$ is a $(\mcf_t)_{t \geq 0}$-martingale.

In addition, $\rmc^1(\msm) \subset \domain(\generator)$ and, if $f\in\rmc_c^1(\msm)$, then $\generator f$ is bounded, therefore $(M^{f,\generator f}_t)_{t \geq 0}$ is a $(\mcf_t)_{t \geq 0}$-martingale. Moreover, since we supposed that $b(x)=(\partial_t)_{t=0} \varphi_t(x)$ exists for all $x\in\msm$, then $\Ddir{\varphi} f(x) = \ps{b(x)}{\diff f(x)}$ for all $f\in\rmc^1(\msm)$ and $x \in \msm$. However, we need some conditions on $\lambda$ and $Q$ to show that 
that $\rmc^1_c(\msm) \subset \domain(\sgenerator)$.
\begin{assumption}
  \label{ass:compact_sup}
    Let $(P_t)_{t \geq 0}$ be a non explosive PDMP semi-group with 
    characteristics $(\varphi,\lambda,Q)$. Assume that for all $T \geq 0$, there exists $M \geq 0$ such that for all $x \in \msm$ and $t \in \ccint{0,T}$, $\supp\{P_t(x,\cdot)\} \subset \cball{x}{M}$.
\end{assumption}
\begin{lemma}
  \label{lem:compact_sup}
  Assume \Cref{ass:compact_sup}.
  \begin{enumerate}[label=(\alph*)]
  \item \label{lem:compact_sup_a}  For all $f \in \rmc_c(\msm)$, $T \in \rset_+$, there exists a bounded set $\msa$ such that  $P_t f (x) = 0$, for all $x \not \in \msa$ and $t \in \ccint{0,T}$.
  \item \label{lem:compact_sup_b} Condition \Cref{ass:poisson_proc_integ} is satisfied. 
  \end{enumerate}
\end{lemma}
\begin{proof}
\begin{enumerate}[label=(\alph*),wide, labelwidth=!, labelindent=0pt]
\item 
Let $f \in \rmc_c(\msm)$, $T \in \rset_+$. By assumption, there exist $M_f,M_T \in \rset_+$ such that $\supp(f) \subset \ball{x_0}{M_f}$, $x_0 \in \msm$, and $\supp\{P_t(x,\cdot)\} \subset \cball{x}{M_T}$ for all $t \in \ccint{0,T}$, $x \in \msm$. Therefore, we get that, for all  $x \not \in \cball{x_0}{M_T+M_f+1}$, $$P_tf(x) = \int_{\msm}\1_{\cball{x}{M_T}\cap \ball{x_0}{M_f}}(y) f(y) P_t(x,\rmd y) = 0 \eqsp.$$ Indeed, by construction $\cball{x}{M_T}\cap \ball{x_0}{M_f} = \emptyset$ since by the triangle inequality,  $\dist(x,y) \leq M_T$ implies that $\dist(x_0,y) \geq M_f+1$. 
\item  Let $(X_t)_{t\geq 0}$
  be a PDMP process with characteristics $(\varphi,\lambda,Q)$ started from $x \in \msm$, given  by \Cref{const:1}, with jump times $(S_k)_{k \in \nset}$. Note that by definition, for all $k \in \nset$, $S_k \leq \bS_k$, where $(\bS_k)_{k \in \nset}$ is the true jump times of the process. Therefore, defining $N_t = \sum_{k=1}^{\plusinfty} \1_{\ccint{0,t}}(S_k)$, we have for all $t \geq 0$, $\bN_t \leq N_t$, and to show that \Cref{ass:poisson_proc_integ} holds, it suffices to show that $\expeLigne{N_T} < \plusinfty$ for all $T \in \rset_+$ and $x\in\msm$.

  Let $T \geq 0$ and $M_T \geq 0$ be such that for all
  $t \in \ccint{0,T}$ and $y \in \msm$,
  $\supp\{P_t(y,\cdot)\} \subset \ball{y}{M_T}$.
  Then, since $(X_t)_{t \geq 0 }$ is a càdlàg process,  almost surely, for all $t \in \ccint{0,T}$, $X_t \in \cball{x}{M_T}$. Therefore, by \eqref{eq:def_s_j_const_1}, for all $k \in \nset$, and $t \in \ccint{0,T}$,
  \begin{equation*}
  \1_{\ccint{0,t}}(S_{k+1})  (S_{k+1} - S_k)\geq E_{k+1}/(1+\norm{\lambda}_{\infty,\cball{x}{M_T}}) \eqsp.
\end{equation*}
Then, for all $t \in \ccint{0,T}$, $N_t$ is bounded by
$\sum_{k=1}^{\plusinfty}
\1_{\ccint{0,t}}(E_{k+1}/(1+\norm{\lambda}_{\infty,\ball{x}{M_T}}))$,
which is a Poisson process with rate $1+\norm{\lambda}_{\infty,\ball{x}{M_T}}$. Therefore, for all $t \in \ccint{0,T}$, $\expeLigne{N_t}  < \plusinfty$.
\end{enumerate}
\end{proof}

\begin{proposition}
  \label{propo:c_1_c_sgenerator}
Assume \Cref{ass:compact_sup}.
\begin{enumerate}[label=(\alph*)]
 \item \label{propo:c_1_c_sgenerator_a} If $(t,x) \mapsto \varphi_t(x) \in \rmc^1(\rset_+ \times \msm, \msm)$, then   $\rmc_0(\msm) \subset \mrb_0(\msm)$, where $\mrb_0(\msm)$ is defined by \eqref{eq:def_mrb_0}.
\item \tcr{If} $(t,x) \mapsto \varphi_t(x) \in \rmc^1(\rset_+ \times \msm, \msm)$, $\lambda \in \rmc(\msm)$ and for all $f \in \rmc_c(\msm)$, $\lambda Qf \in \rmc_c(\msm)$, then $\rmc^1_c(\msm) \subset \domain(\sgenerator)$.
\end{enumerate}
\end{proposition}
\begin{proof}
\begin{enumerate}[label=(\alph*),wide, labelwidth=!, labelindent=0pt]
\item

  We show that $\rmc_c(\msm) \subset \mrb_0(\msm)$ which is sufficient since the closure of $\mrc_{c}(\msm)$ for the uniform norm is $\rmc_0(\msm)$ and $\rmb_0(\msm)$ is a Banach space equipped with the topology endowed with this norm. Let   $f \in \rmc_c(\msm)$. By \Cref{lem:compact_sup}-\ref{lem:compact_sup_a}, there exists a compact set $\msk$ such that for any $x \not \in \msk$ and $t \in \ccint{0,1}$, $P_tf(x) =0$. Therefore we only need to show that $\lim_{t \to \plusinfty} \norm{P_t f -f}_{\infty,\msk} = 0$.
  Let $x \in \msk$ and consider  $(X_t)_{t\geq 0}$
  the PDMP process with characteristics $(\varphi,\lambda,Q)$ started from $x$ given  by \Cref{const:1}, with jump times $(S_k)_{k \in \nset}$.   Then by definition, we have for any $x \in \msk $,
  \begin{equation}
    \label{eq:compact_sup_c_1}
    \abs{P_tf(x) -f(x)} \leq \norm{f}_{\infty} \mathbb{P}\parenthese{S_1 \leq t} + \abs{f(\varphi_t(x))-f(x)} \eqsp. 
  \end{equation}
  On the other hand, let $M \geq 0$ be such that for all
  $t \in \ccint{0,1}$ and $y \in \msm$,
  $\supp\{P_t(y,\cdot)\} \subset \ball{y}{M}$, and define $\tmsk = \{ y \in \msm \, : \, \dist(y , \msk) \leq M\}$, which is bounded since $\msk$ is. Therefore, we get  by
  \eqref{eq:def_s_j_const_1} for $k=1$, 
  \begin{equation*}
  \1_{\ccint{0,t}}(S_{1})  S_{1} \geq E_{1}/(1+\norm{\lambda}_{\infty,\cball{x}{M}}) \geq E_{1}/(1+\norm{\lambda}_{\infty,\tmsk}) \eqsp,
\end{equation*}
where $\norm{\lambda}_{\infty,\tmsk} < \plusinfty$ because $\lambda$ is assumed to be locally bounded. From this result and \eqref{eq:compact_sup_c_1}, it follows that for any $x \in \msk $,
  \begin{equation}
    \label{eq:compact_sup_c_2}
    \abs{P_tf(x) -f(x)} \leq \norm{f}_{\infty} \mathbb{P}\parenthese{E_{1} \leq (1+\norm{\lambda}_{\infty,\tmsk}) t} + \abs{f(\varphi_t(x))-f(x)} \eqsp. 
  \end{equation}
   In addition, using that
 $(t,x) \mapsto \varphi_t(x) \in \mrc^1(\msm)$, there exists
 $C_1 \geq 0$ such that for all $t \in \ccint{0,1}$ and $x \in \msk$, $\dist(\varphi_t(x),x) \leq t C_1$. Then, since $f$ is  continuous, it is uniformly continuous on $ \{ y \in \msm \, : \, \dist(y,\msk) \leq C_1 \}$ and therefore for all $\varepsilon >0$, there exists $\eta >0$ such that for all $x \in \msk$ and $s \in \ccint{0,\eta}$, $\absLigne{f(\varphi_s(x)) - f(x)} \leq \varepsilon$. Combining this result and \eqref{eq:compact_sup_c_2}, we obtain that for any $t \in \ccint{0,\eta}$,
  \begin{equation}
%    \label{eq:compact_sup_c_3}
    \norm{P_tf -f}_{\infty,\msk} \leq \norm{f}_{\infty} \mathbb{P}\parenthese{E_{1} \leq (1+\norm{\lambda}_{\infty,\tmsk}) t} + \varepsilon  \eqsp. 
  \end{equation}
  Taking $t, \varepsilon \to 0$ concludes the proof of \ref{propo:c_1_c_sgenerator_a}.

\item 
  Let $f \in \rmc^1_c(\msm)$. By \Cref{lem:compact_sup}-\ref{lem:compact_sup_a} and since $\lambda Qf \in \rmc_c(\msm)$, there exists a compact set $\msk \subset \msm$, such that for all $x \not \in \msk$, $P_tf(x) = 0$, for all $t \in \ocint{0,1}$, $\lambda(x) Qf(x)  = 0$ and $f(x) = 0$. Therefore, for all $t \in \ocint{0,1}$,
  \begin{equation}
    \label{eq:c_1_c_generator_1}
   \norm{t^{-1}(P_tf(x)-f(x)) - \generator f(x)}_{\infty} = \norm{t^{-1}(P_tf(x)-f(x)) - \generator f(x)}_{\infty,\msk}  \eqsp.
  \end{equation}  
  As seen above, since $f \in \rmc^1_c(\msm)$, $(M^{f,\generator f}_t)_{t \geq 0}$, defined by \eqref{eq:def_martin_local_pdmp_generator}, is a $(\mcf_t)_{t \geq 0}$-martingale. Therefore, for all $x \in \msm$,
  \begin{equation*}
   t^{-1}\defEns{ P_tf(x) - f(x)} - \generator f(x) = t^{-1} \expeMarkov{x}{\int_0^t \defEns{\generator f(\bX_s) - \generator f(x)    } \rmd s} \eqsp.
 \end{equation*}
 Then since $f \in \mrc^1_c(\msm)$, $(t,x) \mapsto \varphi_t(x)$ is continuously differentiable, $\lambda$ is locally bounded and
 $\lambda Qf$ is bounded, there exists $C_1 \geq 0$ such that for all
 $t > 0$ and $x \in \msm$, we have
 \begin{equation}
   \label{eq:c_1_c_generator_2}
\abs{   t^{-1}\defEns{ P_tf(x) - f(x)} - \generator f(x) } \leq C_1 \probaMarkov{x}{\bS_1\leq t} + t^{-1} \expeMarkov{x}{\1_{\ooint{t,\plusinfty}}(\bS_1) \int_0^t \defEns{A^1_s +A^2_s + A^3_s    } \rmd s}\eqsp,
 \end{equation}
 where
 \begin{align*}
&   A^1_s =    \ps{(\partial_u)_{u=0} \{\varphi_u\}(\varphi_s(x))}{\nabla f(\varphi_s(x))}-\ps{(\partial_u)_{u=0} \{\varphi_u\}(x)}{\nabla f(x)} \\
&   A^2_s =   \lambda(\varphi_s(x)) Qf(\varphi_s(x)) - \lambda(x) Qf(x) \eqsp, \, \,  A^3_s = -\lambda(\varphi_s(x))f(\varphi_s(x)) - \lambda(x)f(x)\eqsp.
 \end{align*}

Using that $(x,s)\mapsto \varphi_{s}(x)$ is locally
 bounded on $\rset_+ \times \msm$ and $\lambda$ on $\msm$, there
 exists $C_2$ such that for all $t \in \ocint{0,1}$ and $x \in \msk$, 
 \begin{equation}
   \label{eq:c_1_c_generator_3}
   \probaMarkov{x}{\bS_1\leq t}  \leq \int_0^t \rmd s \lambda(\varphi_s(x)) \exp\parenthese{-\int_0^s \rmd u \lambda(\varphi_u(x))} \leq C_2 t \eqsp.
 \end{equation}
 In addition, using that
 $(t,x) \mapsto \varphi_t(x) \in \mrc^1(\msm)$, there exists
 $C_3 \geq 0$ such that for all $t \in \ocint{0,1}$ and $x \in \msk$, $\dist(\varphi_t(x),x) \leq t C_3$. Then, since $\ps{(\partial_u)_{u=0} \varphi_u}{\nabla f}$, $\lambda Qf$ and $\lambda f$ are continuous, they are uniformly continuous on $\{ y \in \msm \, : \, \dist(y,\msk) \leq C_3\}$ and therefore for all $\varepsilon >0$, there exists $\eta >0$ such that for all $x \in \msk$ and $s \in \ccint{0,\eta}$, $\abs{A^i_s} \leq \varepsilon$, $i=1,2,3$. Combining this result and 
 \eqref{eq:c_1_c_generator_3} in \eqref{eq:c_1_c_generator_2}, we get for all $x \in \msk$, $\varepsilon >0$ and $t \in \ccint{0,\eta\wedge 1}$,
 \begin{equation*}
\abs{   t^{-1}\defEns{ P_tf(x) - f(x)} - \generator f(x) } \leq C_1C_2t+ 3\varepsilon \eqsp.
\end{equation*}
Therefore, by \eqref{eq:c_1_c_generator_1} we get for all $\varepsilon >0$,
\begin{equation*}
  \limsup_{t \downarrow 0} \norm{t^{-1}\defEns{ P_tf(x) - f(x)} - \generator f(x)}_{\infty} \leq 3 \varepsilon \eqsp. 
\end{equation*}
Taking $\varepsilon \to 0$ concludes the proof.
\end{enumerate}
\end{proof}

Let us finish by our running example.

\begin{example*}
Consider the BPS process defined in  \Cref{ex:BPS} and suppose that $\msy$ is bounded. It is then easy to verify that \Cref{ass:compact_sup} is verified, and  the strong generator is given for all $f\in \rmc^1_c(\msm)$ by
\begin{multline*}
  {\bar {\mathcal A}}f(x,y)=\langle y,\nabla f(x,y)\rangle +\left(\langle y,\nabla U(x)\rangle\right)_+\left(f(x,R(x,y))-f(x,y)\right)\\
  +\rate\left(\int_\msy f(x,w)\mbox{d}\loiy(w)-f(x,y)\right).
\end{multline*}
\end{example*}

%We call domain of the strong generator the set $\domain_s(\generator)$ of functions $f\in\MeasFspace(\msm)$ such that $(P_t )$

%\begin{lemma}
%\label{lem:synchrone_generator}
%Let $(\bar \generator_t)_{t\geqslant0}$ (resp. $(\generator_t)_{t\geqslant0}$) be the  generator of a PDMP $(\bar X_t,\bar Y_t)_{t\geqslant0}$ on $\msm^2$ (resp. $(X_t)_{t\geqslant0}$ on $\msm$), with $X_0 = \bar X_0=x\in\msm$. Denote $\pi: \rmc^1_c(\msm) \rightarrow \rmc^1(\msm^2)$ defined by $(\pi g)(x,y) = g(x)$ for all $g\in\rmc^1_c(\msm)$ and $(x,y)\in\msm^2$. Suppose that $\bar \generator_t \pi g = \generator_t g $ for all $g\in\rmc^1_c(\msm)$ and $t\geqslant0$. Then $(\bar X_t)_{t\geqslant0}$ and $(X_t)_{t\geqslant0}$ have the same distribution. In particular, $(\bar X_t)_{t\geqslant0}$ is a Markov process.
%\end{lemma}
%
%\begin{proof}
%For $g\in\rmc^1_c(\msm)$ and $t\geqslant0$,
%\begin{eqnarray*}
%\mathbb E \po g(\bar X_t)\pf & =&  \mathbb E \po \pi g(\bar X_t,\bar Y_t)  \pf\\
%& = & g(x) + \mathbb E \po  \int_0^t \bar\generator_s \pi g(\bar X_s,\bar Y_s) \rmd s\pf\\
%& = & g(x) + \mathbb E \po  \int_0^t \generator_s \pi g(\bar X_s) \rmd s\pf,
%\end{eqnarray*}
%and \cite{LaRefQuiVaBien} concludes.
%\end{proof}

%%% Local Variables:
%%% mode: latex
%%% TeX-master: "main"
%%% End:

%% file: regularity.tex
\section{Regularity estimates for PDMP semigroups}
\label{sec:stab_c1c}

\tcr{After introducing the strong and extended generators of a  PDMP in the previous section, we continue to set up the intermediary tools that will  eventually lead to a simple criterion for verifying that a probability measure is invariant for a given  PDMP (\Cref{theo:core_smooth_approx} and \Cref{lem:critere-invariant} in the next section). In case of Markov semigroups defined through diffusion processes, regularization properties can be established which usually imply that some class of smooth functions is a core for the strong generator of the semigroup under consideration. In this spirit, the main goal of this section is to provide similar results for PDMPs, \ie~we provide conditions upon which  some class of smooth functions is a core for the strong generator of a PDMP (see also \cite{BLBMZ3} on that topic). However, compared to diffusion semigroups, we cannot expect using regularization properties.} 

To do so, we need the following definition.
\begin{definition}
  \label{def:compact_comp}
We say that a homogeneous
differential flow $\varphi$ on $\msm$ and a homogeneous Markov kernel $Q$ on $\msm$ are
compactly compatible if for all compact set $\msk \subset \msm$ and $T \geq 0$,
there exists a compact set $\tilde \msk \subset \msm$ satisfying: for
all $n \in \nsets$, $(t_i)_{i \in \iint{1,n}} \in \rset_+^n$,
$\sum_{i=1}^{n} t_i \leq T$,
there exists a sequence $(\msk_i)_{i \in \iint{1,n}}$ of
compact sets of $\msm$ such that, setting $\msk_0 = \msk$,
\begin{enumerate}[label=(\roman*)]
  \item \label{def:compact_comp_item_1} for all $i \in \iint{1,n}$, $\msk_i$ only depends on $(t_j)_{j \in \iint{1,i}}$ and $\cup_{i=0}^{n} \msk_i \subset \tilde \msk$;
  \item \label{def:compact_comp_item_2} for all $i \in \iint{0,n-1}$, $s_{i+1} \in \ccint{0,t_{i+1}}$ and $s_{n+1} \in \ccintLigne{0,T-\sum_{j=1}^n t_j}$,
    \begin{equation*}
%  \label{def:msk_i}
 \bigcup_{x \in \msk_i} \supp \{ Q(\varphi_{t_{i+1}}(x),\cdot) \} \subset
\msk_{i+1} \eqsp, \qquad \varphi_{s_{i+1}}( \msk_{i}) \subset \tilde \msk  \eqsp, \qquad  \varphi_{s_{n+1}}( \msk_{n}) \subset \tilde \msk
\end{equation*}

\end{enumerate}
\end{definition}

Note that by definition, if $\varphi$ and $Q$ are compactly compatible
and the PDMP semigroup with  characteristics $(\varphi, \lambda,Q)$ is
non explosive, for all $T \geq 0$ and all
compact set $\msk \subset \msm$, there exists a compact set
$\tilde \msk \subset \msm$, such that
$\PP (X_t \in \tilde{\msk}, \text{ for all $t \in \ccint{0,T}$} ) =1$,
where $(X_t)_{t \geq 0}$ is a PDMP process with  characteristics $(\varphi, \lambda,Q)$ and starting from $X_0 \in \msk$.

\begin{assumption}
  \label{ass:stability}
The homogeneous  characteristics $(\varphi,\lambda,Q)$ satisfy
\begin{enumerate}[label = (\roman*)]
\item \label{item_1_lem:stab_c1c} the flow  $\varphi$ and the Markov kernel $Q$ are compactly compatible;
\item \label{item_2_lem:stab_c1c} $\lambda\in\mrc^{1}(\msm)$ and for
  all $f \in \mrc^1(\msm)$, $\lambda Qf \in \mrc^1(\msm)$ and there
  exists a locally bounded function $\Psi :\msm \to \rset_+$ such that for all %compact set $\msk \subset \msm$, $\sup_{y \in \msk} \Psi(y) < \plusinfty$ and for all 
  $x \in \msk$,
  \begin{equation*}
    \label{eq:1}
 \norm{\diff (\lambda Qf)(x)} \leq   \norm{\Psi}_{\infty,\msk} \sup \ensemble{\abs{f(y)} + \norm{\diff f(y)}}{y \in \supp \{Q(x,\cdot)\}};
  \end{equation*}
\item \label{item_3_lem:stab_c1c} $(t,x) \mapsto \varphi_t(x)\in \rmc^1(\rset_+\times \msm)$ and  for all compact $\msk \subset \msm$ and $t \geq 0$,
  \begin{equation*}
    \sup \ensemble{\norm{\diff \varphi_s(x)}}{s \in \ccint{0,t}, x \in \msk} < \plusinfty \eqsp.
  \end{equation*}
\end{enumerate}  
\end{assumption}

\begin{theorem}
  \label{lem:stab_c1c}
    Let $(P_t)_{t\geqslant 0}$ be a non explosive PDMP semigroup on $\msm$ with corresponding characteristics $(\varphi,\lambda,Q)$ satisfying \Cref{ass:stability}. 
    Then for all $f \in \mrc^{i}(\msm)$, $i\in\{0,1\}$, $T \in \rset_+$, $P_Tf \in \mrc^{i}(\msm)$. In addition,  for any $f \in \rmc^1(\msm)$ and  compact set $\msk \subset \msm$, there exists $C \geq  0$  such that for all $t \in \ccint{0,T}$,
  \begin{equation}
    \label{eq:lem_stab_clc}
  \sup_{x \in \msk}\{ \abs{P_t f}(x) + \norm{\diff (P_tf)(x)}\} \leq C \eqsp.  
\end{equation}

%   In particular $\mrc^1(\msm)$ is a  core for $(P_t)_{t \geq 0}$. 
% that  and $(x,t)\in\msm\times\rset_+ \mapsto \norm{\na \varphi_t (x)}$ is locally bounded, and that there exist locally bounded functions $C_1:\msm \rightarrow \rset_+$ and $C_2:\rset_+\rightarrow \rset_+$  such that for all $f\in \rmc^1_c(\msm)$, $\lambda Qf\in \rmc^1_c(\msm)$ and for all $x\in\msm$, 
% \[\norm{\na (\lambda Q f)(x)}\ \leqslant  \ C_1(x)\sup\{|f(y)| + \norm{\na f(y)},\ y\in \supp Q(x,\cdot)\}\eqsp,\]
%  and for all $t\geqslant 0$,  $X_s \in \msb(X_0,C_2(t))$ for all $s\in[0,t]$ almost surely. Then, if $f\in \mrc^1_c(\msm)$, $ P_t f\in \mrc^1_c(\msm)$ for all $t \geq 0$. %Therefore, $\mrc^1_c(\msm)$ is a core for $(P_t)_{t \geq 0}$.

% More precisely, if $f\in \mrc^1_c(\msm)$ and $t\geqslant0$, then there exist $C>0$ and a compact $\msk$ of $\msm$  such that,  for all $s\in [0,t]$, the support of $P_s f$ is included in $\msk$ and  $\norm{\na P_s f}_\infty \leqslant C$.
\end{theorem}

% \begin{remark}
%   Note that in the case where the condition $\inf_{x \in \msm} \lambda(x) >0$ is not satisfied, we can simply consider a modified characteristics $(\varphi,\bar \lambda, \bar Q)$ which defines the same PDMP semigroup $(P_t)_{t \geq 0}$ where $\bar Q$ and $\bar \lambda$ are given for all $x \in \msm$ and $\msa \in \mcb{\msm}$ by
%   \begin{equation*}
%     \bar \lambda(x) = \lambda(x) + \rate \eqsp, \, \bar Q(x,\msa) = \bar \lambda^{-1}(x)\lambda(x) Q(x,\msa) + \parenthese{1-\bar \lambda^{-1}(x)\lambda(x)} \updelta_x(\msa) \eqsp,
%   \end{equation*}
%   for $\rate >0$. Then, if the other conditions of \Cref{lem:stab_c1c} are satisfied, the modified characteristics $(\varphi,\bar \lambda, \bar Q)$ satisfy these conditions as well and $\inf_{x \in \msm} \bar \lambda(x) \geq \rate$. 
% \end{remark}
\begin{proof}
We only show the result for $f \in \rmc^{1}(\msm)$, the proof for $f \in \rmc(\msm)$ is similar and left to the reader. 
For all $x \in \msm$
  denote by $(X_t^x)_{t \geq 0}$ a PDMP starting from $x$ associated
  with the characteristics $(\varphi,\lambda,Q)$ and defined by \Cref{const:1}.  Let
  $(S_n^x)_{n \in \nset}$ be the jump times of $(X_t^x)_{t \geq 0}$
  for all $x \in \msm$ and $(\mcf_t)_{t \geq 0}$ the associated filtration. Let $f\in \mrc^1(\msm) $, $T\geqslant 0$ and a
  compact set $\msk \subset \msm$. For $T =0$, the result is
  straightforward so we consider $T >0$. Let $\tilde \msk$ satisfying for all $n \in \nset$,
$(t_i)_{i \in \iint{1,n+1}} \in \rset_+^n$,
$\sum_{i=1}^{n+1} t_i \leq T$,  \ref{def:compact_comp_item_1}-\ref{def:compact_comp_item_2} in \Cref{def:compact_comp}.
% $\cup_{i=0}^{n} \msk_i \subset \tilde \msk$,
% $\varphi_{s_i}( \msk_{i-1}) \subset \tilde \msk$, for all
% $i \in \iint{1,n+1}$ and $s_i \in \ccint{0,t_i}$ setting
% $t_0=0$, where $(\msk_i)_{i \in \{0,\ldots,n\}}$ are given by \eqref{def:msk_i}. 

Since for all $x  \in \msk$, 
  $\PP (X_t^x \in \tilde \msk, \text{ for all } t \in \ccint{0,T}) = 1$,
  for all $t \in \ccint{0,T}$ and $x \in \msm$,
  \begin{equation}
    \label{eq:core_proof_1}
    \abs{P_tf(x)} = \abs{\expe{f(X_t)}} \leq \sup_{y \in \tilde \msk} \abs{f}(y) \eqsp.
  \end{equation}
  
Furthermore,  $(P_t)_{t \geq 0}$
  is assumed to be non explosive. Therefore
  $\sup_{n \in \nset} \S^{x}_n = \plusinfty$ and we can consider the
  following decomposition for all $t \in \ccint{0,T}$ and $x \in \msk$
  \begin{equation}
    \label{eq:decomposition_regular_pDMP}
  P_tf(x)  = \sum_{n=0}^{\plusinfty} \expeMarkov{}{\1_{\coint{S^{x}_n,S^{x}_{n+1}}}(t) f(X^{x}_t)} \eqsp.
\end{equation}
We  show that for all $n \in \nset$,
  \begin{equation*}
%    \label{eq:def_F_n}
   F_{n,t} : x \mapsto
  \expeMarkov{}{\1_{\cointLigne{S^{x}_n,S^{x}_{n+1}}}(t)
    f(X^{x}_t)}  \eqsp,
  \end{equation*}
 is continuously
  differentiable and in addition there exists $C \geq 0$ such that for all
  $n \in \nset$, $t \in \ccint{0,T}$
  \begin{equation}
    \label{eq:estimate_gradient_n}
\sup_{x \in \msk}   \norm{ \diff  F_{n,t} (x)} \leq C^n  /n!\eqsp.%       \expeMarkov{(x,y)}{\1_{\coint{\bS_n^{\epsilon},\bS_n^{\epsilon}}}(t) } \eqsp,
   \end{equation}
   %where $(\bS_n^{\epsilon})_{n \in \nset}$ are the event times of a Poisson process with constant jump rate $\rate$.
   Assume for the moment that this result holds for any compact set $\msk \subset \msm$. Then, we have for all $t \in \ccint{0,T}$ and compact $\msk \subset \msm$,
    \begin{equation*}
      \lim_{N \to \plusinfty}   \sum_{n=N}^{\infty} \sup_{x \in \msk}  \norm{ \diff F_{n,t}(x) } = 0 \eqsp. 
    \end{equation*}
    By \eqref{eq:decomposition_regular_pDMP}, it implies that
    $x \mapsto P_T f(x)$ is continuously differentiable on $\msm$. In
    addition, for all compact set $\msk \subset\msm$, there exists $C \geq 0$ such that for all $t \in \ccint{0,T}$, $ \norm{\diff P_tf(x)}_{\msk} \leq C $.
 %    \begin{equation*}
 %        \label{eq:core_proof_0}
 % \norm{\diff P_tf(x)}_{\msk} \leq C \eqsp.
 %    \end{equation*}
This result and \eqref{eq:core_proof_1} imply \eqref{eq:lem_stab_clc}.
  
    We now turn in showing that for all $n \in \nset$, $F_n$ is
    continuously differentiable and \eqref{eq:estimate_gradient_n}
    holds. We first show this result for $n = 0$. In a second time, we
    make an induction on $n \in \nsets$. 

For all $x \in \msk$ and $t \in \ccint{0,T}$, we have
    \begin{equation*}
      F_{0,t}(x) =  f(\phi_t(x)) \exp \parenthese{- \int_0 ^t \lambda(\phi_s(x)) \rmd s} \eqsp. 
    \end{equation*}
Therefore, for all $x\in \msk$ and $t \in \ccint{0,T}$, we obtain by \Cref{ass:stability}-\ref{item_2_lem:stab_c1c}-\ref{item_3_lem:stab_c1c}
    \begin{multline*}
      \diff F_{0,t}(x) = \defEns{\diff f (\varphi_t(x))} \diff (\varphi_t)(x) \exp \parenthese{- \int_0 ^t \lambda(\varphi_s(x)) \rmd s} \\
+        f(\varphi_t(x))   \parentheseDeux{\int_0 ^t \defEns{\diff  \lambda (\varphi_s(x)) \cdot \diff(\varphi_s)(x) }\rmd s}\exp \parenthese{- \int_0 ^t \lambda(\varphi_s(x)) \rmd s} \eqsp.
\end{multline*}
Since for all $x \in \msk$ and $t \in \ccint{0,T}$, $\varphi_t(x) \in \tilde \msk$, $f \in \mrc^1(\msm)$ and using \Cref{ass:stability}-\ref{item_2_lem:stab_c1c}-\ref{item_3_lem:stab_c1c}, we get there exists $C_0 \geq 0$ such that for all $t \in \ccint{0,T}$,
\begin{equation}
  \label{eq:bound_F_0}
\sup_{x \in \msk} [     \abs{F_0(x)}+    \norm{\diff F_0(x)}  ]\leq C_0 \eqsp.
    \end{equation}
%    where we have used that for all $x \in \msm$, $\lambda(x) \geq \rate$.

We now show the result for $n \in \nsets$. We first give  an explicit expression of $F_n$ for all $n \in \nsets$. Indeed, we have conditioning successively on $\mcf_{S_{n+1}^x}, \cdots, \mcf_{S_1^x}$, for all $x \in \msk$ and $t \in \ccint{0,T}$
\begin{multline*}
%  \label{eq:forme_loimarkov_chain}
  % \expeMarkov{(x,y)}{\1_{\coint{S^{(\epsilon)}_n,S^{(\epsilon)}_{n+1}}}(t) f(X^{(\epsilon)}_t,Y^{(\epsilon)}_t)}
  F_{n,t}(x)
  = \int_0^t \rmd t_1 \exp\parenthese{-\int_0^{t_1} \lambda\defEns{\varphi_{s_1}(x)} \rmd s_1} \int_{\msm } \opK(\varphi_{t_1}(x),\rmd x_1)  \\
\qquad \qquad   \int_0^{t-t_1} \rmd t_2 \exp\parenthese{-\int_0^{t_2} \lambda\defEns{\varphi_{s_2}(x_1)} \rmd s_2}  \int_{\msm } \opK(\varphi_{t_2}(x_1),\rmd x_2) \\
  \cdots   \int_0^{t-\sum_{i=1}^{n-1}t_i} \rmd t_n \exp\parenthese{-\int_0^{t_n} \lambda\defEns{\varphi_{s_n}(x_{n-1})} \rmd s_n} 
  \int_{\msm } \opK(\varphi_{t_n}(x_{n-1}),\rmd x_n) f(\varphi_{t-\sum_{i=1}^n t_i} (x_n))   \\  \exp\parenthese{-\int_0^{t-\sum_{i=1}^n t_i} \lambda\defEns{\varphi_{s_{n+1}}(x_{n})} \rmd s_{n+1}} \eqsp,
\end{multline*}
where $\opK$ is the kernel defined on $(\msm, \mcb{\msm})$  for all $x \in \msm$ and $\msa \in \mcb{\msm}$ by
  \begin{equation*}
    \opK(x,\msa) = \lambda(x) Q(x,\msa)  \eqsp.
  \end{equation*}

 We introduce a sequence of operator
$( Q^{(n)})_{n \in \nset^*}$, defined for all
$g : \rset_+ \times \msm \to \rset$, bounded on all compact of
$\ccint{0,T} \times \msm$ and measurable, $t \in \ccint{0,T}$ and
$x \in \msm$ by
  \begin{multline*}
%\label{eq:def_q_n}
    Q^{(n)} g (t,x) = \int_0^t \rmd t_1 \exp\parenthese{-\int_0^{t_1} \lambda\defEns{\varphi_{s_1}(x)} \rmd s_1} \int_{\msm } \opK(\varphi_{t_1}(x),\rmd x_1) \\
    \cdots \int_0^{t-\sum_{i=1}^{n-1} t_i} \rmd t_n  \exp\parenthese{-\int_0^{t_n} \lambda\defEns{\varphi_{s_n}(x_{n-1})} \rmd s_n} \int_{\msm } \opK(\varphi_{t_n}(x_{n-1}),\rmd x_n) g(t-\sum_{i=1}^n t_i ,x_n) \eqsp.
  \end{multline*}
  Taking for $g$ the function
  $g_F : (s,y) \mapsto f(\varphi_s(y)) \exp(-\int_0^s \rmd u
  \lambda(\varphi_u(y)))$, we have 
  \begin{equation}
    \label{eq:F_n_g_F_lem_prof_tech}
    F_{n,t} = Q^{(n)} g_F(t,\cdot) \eqsp. 
  \end{equation}
 Since
  $f \in \mrc^1(\msm)$ and by \Cref{ass:stability}-\ref{item_3_lem:stab_c1c}, $g_F$ is
  measurable, for all $s \in \ccint{0,T}$, $y \mapsto g_F(s,y)$ is continuously
  differentiable on $\msm$ and satisfies for all $T' \in \ccint{0,T}$,
  $\msk' \subset \msm$ compact,
  $\sup_{s \in \ccint{0,T'}, y \in \msk'} \{\abs{g_F(s,y)} +
  \norm{\diff_x g_F(s,y)}\} < \plusinfty$. 
Denote by  $\mrgD$ the set of measurable functions $g :\rset_+ \times \msm \to \rset_+$ satisfying 
  for all $s \in \ccint{0,T}$, $y \mapsto g(s,y)$ is continuously differentiable
  on $\msm$ and for all $T' \in \ccint{0,T}$, $\msk' \subset \msm$
  compact,
  $\sup_{s \in \ccint{0,T'}, y \in \msk'} \abs{g(s,y)} +
  \norm{\diff_x g(s,y)} < \plusinfty$. Then, if we show for any $g \in \mrgD$ that (1) for all
 $t \in \ccint{0,T}$ and $x \in \msk$, $y \mapsto Q^{(n)} g(t,y)$ is differentiable at $x$; (2) there exists $C \geq 0$ such that for all $n \in \nsets$,
  $x \in \msk$ and $t \in \ccint{0,T}$,
  \begin{equation}
    \label{eq:objectif_g}
    \abs{Q^{(n)} g(t,x)} + \norm{\diff_x Q^{(n)} g(t,x)} \leq C^n/(n!) \eqsp,
  \end{equation}
\eqref{eq:F_n_g_F_lem_prof_tech} and \eqref{eq:bound_F_0} show that \eqref{eq:estimate_gradient_n} holds and the proof is completed. 

  Note the following relation between $Q^{(n-1)}$ and $Q^{(1)}$ which will be essential to the proof: for all
$g : \rset_+ \times \msm \to \rset$, bounded on all compact of
$\rset_+ \times \msm$ and measurable, $t \in \ccint{0,T}$ and
$x \in \msm$ 
  \begin{multline}
\label{eq:q_n_prop_q_1}
    Q^{(n)} g (t,x) = \int_{0}^t \rmd t_1  \exp\parenthese{-\int_0^{t_1} \lambda\defEns{\varphi_{s_1}(x)} \rmd s_1} \int_{\msm } \opK(\varphi_{t_1}(x),\rmd x_1) Q^{n-1}g(t-t_1,x_1) \\
    = Q^{(1)} [Q^{(n-1)}g](t,x)\eqsp.
  \end{multline}
  
  First, we make an induction on $i \in \iint{1,n}$ to show that
  for all $i \in \iint{1,n}$ and $g \in \mrg$, that
  $(s,x) \mapsto Q^{(i)} g(s,x) \in \mrgD(\ccint{0,T}\times \msm)$, which will show that for
  all $t \in \ccint{0,T}$,  $y\mapsto Q^ng(t,y)$
  is continuously differentiable. For $i =1$, note that for all
  $s \in \ccint{0,T}$ and $y \in \msm$,
  \begin{equation}
  \label{eq:def_q_1}
    Q^{(1)}g(s,y) = \int_0^s \rmd t_1 \exp\parenthese{-\int_0^{t_1} \lambda\defEns{\varphi_{s_1}(y)} \rmd s_1}  \int_{\msm } \opK(\varphi_{t_1}(y),\rmd y_1) g(s-t_1 ,y_1) \eqsp.
  \end{equation}
  Let $T' \in \ccint{0,T}$, $\msk' \subset \msm$ be compact and $\tilde \msk'$ given by \Cref{def:compact_comp} associated with
  $\msk'$ and $T'$. Then for all $s \in \ccint{0,T'}$, $t_1\in\ccint{0,s}$, $\varphi_{t_1}(y) \in \tilde \msk '$ for all $y \in \msk'$. Therefore we have 
 by assumption on $g$, \Cref{ass:stability}-\ref{item_2_lem:stab_c1c}-\ref{item_3_lem:stab_c1c},
  $\opK g(s-t_1, \cdot)(\varphi_{t_1}(\cdot)) : y \mapsto  \int_{\msm } \opK(\varphi_{t_1}(y),\rmd y_1) g(s-t_1 ,y_1)$ for all $s \in \ccint{0,T'}$, $t_1 \in \ccint{0,s}$, is differentiable and there exists $C \geq 0$ such that for all $s \in \ccint{0,T'}$, $t_1 \in \ccint{0,s}$
  \begin{equation*}
    \sup_{y \in \msk'} \defEns{ \abs{\opK g(s-t_1 ,\cdot)(\varphi_{t_1}(y))} + \norm{\diff_x  \opK g(s-t_1, \cdot)(\varphi_{t_1}(y))}} < C \eqsp
  \end{equation*}
  By \eqref{eq:def_q_1}, we get then the result for $i=1$. The result for $i \in \iint{2,n}$ is then a straightforward consequence of \eqref{eq:q_n_prop_q_1} and the case $i=1$.

  We now show that for all $g \in \mrgD$ that  there exists $C \geq 0$ such that for all
  $x \in \msk$ and $t \in \ccint{0,T}$, \eqref{eq:objectif_g}
  holds. % Note that for all $j \in \iint{1,n-1}$, by \eqref{eq:def_q_n} we have for all
  % $s \in \rset_+$ and $x \in \msm$,
  % \begin{multline*}
  %   Q^n_tg(x) =  \int_0^t \rmd t_1 \exp\parenthese{-\int_0^{t_1} \lambda\defEns{\phi_{s_1}(x)} \rmd s_1} \int_{\msm } \opK(\phi_{t_1}(x),\rmd x_1) \\
  %   \cdots \int_{\msm } \opK(\phi_{t_{j}}(x_{j-1}),\rmd x_j)  Q_{t-\sum_{i=1}^{j} t_i}^{n-j} g(x_j) \eqsp.
  % \end{multline*}
  By an induction on $N \in \iint{1,n}$, we show that for all $t \in \ccint{0,T}$,
  $(t_i)_{i \in \iint{1,n-N}} \in \rset_+^{n-N}$, $\sum_{j=1}^{n-N} t_j \leq t$, there exists $(\msk_{i})_{i \in \iint{0,n-N}}$ satisfying
  \ref{def:compact_comp_item_1}-\ref{def:compact_comp_item_2} in \Cref{def:compact_comp} with respect to $\msk$, $T$, $\tmsk$, $(t_i)_{i \in \iint{1,n-N}}$  and the following bound holds
  \begin{multline}
    \label{eq:stab_cont_rec}
    \sup_{x_{n-N} \in \msk_{n-N}} \abs{Q^N g\parenthese{t-\sum_{j=1}^{n-N}t_j,x_{n-N}} } + \norm{\diff Q^N g\parenthese{t-\sum_{j=1}^{n-N}t_j,x_{n-N}} } \\
    \leq C^{N}_1 \defEns{\sup_{s \in \ccint{0,T}, y \in \tmsk} \abs{g(s,y)} +\sup_{s \in \ccint{0,T}, y \in \tmsk} \norm{\diff_x g(s,y)} }   \left. \parenthese{t-\sum_{j=1}^{n-N}t_j}^N \middle /(N!) \right .\eqsp,
  \end{multline}
  where 
  \begin{equation}
    \label{eq:def_C_1C_2}
    C_1 = \norm{\lambda}_{\infty,\tmsk} + C_2 \eqsp, \qquad C_2 = \sup_{s \in \ccint{0,T}} \norm{\diff_x \varphi_s(x)}_{\infty, \tmsk} (T \norm{\diff \lambda}_{\infty,\tmsk} + \norm{\Psi}_{\infty,\tmsk}) \eqsp.
  \end{equation}
Then, the result for $N=n$ will conclude the proof.
  
  For $N=1$, let $t \in \ccint{0,T}$,
  $(t_i)_{i \in \iint{1,n-1}} \in \rset_+^{n-1}$, $\sum_{j=1}^{n-1} t_j \leq t$. Note that for all $y \in \msm$, setting $u_{n-1} = \sum_{i=1}^{n-1} t_i$,
    \begin{equation*}
%  \label{eq:def_q_1}
    Q^{(1)}g\parenthese{t-u_{n-1},y} = \int_0^{t - u_{n-1}} \rmd t_{n} \rme^{-\int_0^{t_n} \lambda\defEns{\varphi_{s_n}(y)} \rmd s_n}  \int_{\msm } \opK(\varphi_{t_n}(y),\rmd y_n) g(t-u_{n-1}-t_n ,y_n) \eqsp.
  \end{equation*}
For all $t_n \in \rset_+$, such that
  $\sum_{i=1}^n t_i < t$, by \Cref{ass:stability}-\ref{item_1_lem:stab_c1c}, there exists $(\msk_i)_{i \in \iint{0,n}}$ satisfying \ref{def:compact_comp_item_1}-\ref{def:compact_comp_item_2} in \Cref{def:compact_comp}. In particular, $(\msk_i)_{i \in \iint{0,n-1}}$ only depends on  $(t_i)_{i \in \iint{1,n-1}}$.  Then, using \Cref{ass:stability}, we get for all 
  $x_{n-1} \in \msk_{n-1}$,
\begin{align*}
&  \abs{ Q^{(1)}g\parenthese{t-u_{n-1},x_{n-1}}} \\
&  \leq \norm{\lambda}_{\infty,\tmsk}  \int_0^{t - u_{n-1}} \rmd t_{n}    \sup\defEns{\abs{g}(t-u_{n-1}-t_n ,y_n)\, : \, y_n \in \supp\{ Q(\varphi_{t_n}(x_{n-1}),\rmd y_n)\}} \\
                                          & \qquad  \leq  \norm{\lambda}_{\infty,\tmsk}  \int_0^{t - u_{n-1}} \rmd t_{n}    \sup\defEns{\abs{g}(t-u_{n-1}-t_n ,y_n)\, : \, y_n \in \msk_n} \eqsp,
\end{align*}
and
\begin{align*}
&   \norm{\nabla_x  Q^{(1)}g\parenthese{t-u_{n-1},x_{n-1}}} \\
 &\leq C_2 \int_0^{t - u_{n-1}} \hspace{-0.5cm}\rmd t_{n}    \sup\defEns{\abs{g}(t-u_{n-1}-t_n ,y_n) \, : \, y_n \in \supp\{ Q(\varphi_{t_n}(x_{n-1}),\rmd y_n)\}}\\
& +  C_2 \int_0^{t - u_{n-1}} \hspace{-0.5cm} \rmd t_{n}    \sup\defEns{\norm{\nabla_x g (t-u_{n-1}-t_n ,y_n)} \, : \, y_n \in \supp\{ Q(\varphi_{t_n}(x_{n-1}),\rmd y_n)\}}\\
&\leq C_2 \int_0^{t - u_{n-1}}\hspace{-0.5cm} \rmd t_{n}    \sup\defEns{\abs{g}(t-u_{n-1}-t_n ,y_n^1)+ \norm{\nabla_x g (t-u_{n-1}-t_n ,y_n^2)} \, : \, y_n^1,y_n^2 \in \msk_n} 
  \eqsp,
\end{align*}
where $C_2$ is given by \eqref{eq:def_C_1C_2}. Combining these two results and using that $\msk_n \subset \tmsk$ for all $t_n$, $\sum_{i=1}^n t_i < t$ give~\eqref{eq:stab_cont_rec} for $N=1$.

Now assume that the result holds for $N \in \iint{1,n-1}$ and let $(t_i)_{i \in \iint{1,n-N-1}} \in \rset_+^{n-N-1}$. By induction hypothesis, for all $t_{n-N} \in \rset_+$, such that
  $\sum_{i=1}^{n-N} t_i < t$, there exists  $(\msk_{i})_{i \in \iint{0,n-N}}$ satisfying
  \ref{def:compact_comp_item_1}-\ref{def:compact_comp_item_2} in \Cref{def:compact_comp} with respect to $\msk$, $T$, $\tmsk$, $(t_i)_{i \in \iint{1,n-N}}$.  Then, using \Cref{ass:stability} and \eqref{eq:q_n_prop_q_1}, we get for all 
  $x_{n-N-1} \in \msk_{n-N-1}$, setting $u_{n-N-1} =\sum_{i=1}^{n-N-1} t_i $ and $\msa_{n-N} = \supp\{ Q(\varphi_{t_{n-N}}(x_{n-N-1}),\rmd y_{n-N})\}$
\begin{align*}
&  \abs{ Q^{(N)}g\parenthese{t-u_{n-N-1},x_{n-N-1}}} \\
& \leq \norm{\lambda}_{\infty,\tmsk}  \int_0^{t - u_{n-N-1}} \hspace{-0.5cm} \rmd t_{n-N}    \sup\defEns{\abs{Q^{(N-1)} g}(t-u_{n-N-1}-t_{n-N} ,y_{n-N})\, : \, y_{n-N} \in \msa_{n-N}} \\
                                          &   \leq  \norm{\lambda}_{\infty,\tmsk}  \int_0^{t - u_{n-N-1}}  \hspace{-0.5cm} \rmd t_{n-N}    \sup\defEns{\abs{Q^{(N-1)}g}(t-u_{n-N-1}-t_{n-N} ,y_{n-N})\, : \, y_{n-N} \in \msk_{n-N}} \eqsp,
\end{align*}
and
\begin{align*}
&   \norm{\nabla_x  Q^{(N)}g\parenthese{t-u_{n-N-1},x_{n-N-1}}} \\
   &\leq C_2 \int_0^{t - u_{n-N-1}} \hspace{-0.5cm} \rmd t_{n-N}    \sup\defEns{\abs{Q^{(N-1)} g}(t-u_{n-N-1}-t_{n-N} ,y_{n-N}) \, : \, y_{n-N} \in \msa_{n-N}}\\
& +  C_2 \int_0^{t - u_{n-N-1}} \hspace{-0.5cm} \rmd t_{n-N}    \sup\defEns{\norm{\nabla_x Q^{(N-1)} g (t-u_{n-N-1}-t_{n-N} ,y_{n-N})} \, : \, y_{n-N} \in \msa_{n-N}}\\
   &\leq C_2 \int_0^{t - u_{n-N-1}} \rmd t_{n-N}    \sup\defEns{\abs{Q^{(N-1)}g}(t-u_{n-N-1}-t_{n-N} ,y_{n-N}) \, : \, y_{n-N} \in \msk_{n-N}}\\
& +  C_2 \int_0^{t - u_{n-N-1}} \hspace{-0.5cm} \rmd t_{n-N}    \sup\defEns{\norm{\nabla_x Q^{(N-1)} g (t-u_{n-N-1}-t_{n-N} ,y_{n-N})} \, : \, y_{n-N} \in \msk_{n-N}}
  \eqsp,
\end{align*}
where $C_2$ is given by \eqref{eq:def_C_1C_2}. Combining these two results using  $\msk_{n-N} \subset \tmsk$ for all $t_{n-N}$, $\sum_{i=1}^{n-N} t_i < t$ and the induction hypothesis conclude the proof of \eqref{eq:stab_cont_rec}.

\end{proof}

\begin{remark}
\label{rem:generalization_stab_clc}
  \Cref{lem:stab_c1c} can be generalized  under the condition that
for some $k \in \nsets$,  the  characteristics $(\varphi,\lambda,Q)$ satisfy
\begin{enumerate}[label = (\roman*)]
\item \label{item_1_lem:stab_c1c_b} the flow  $\varphi$ and the Markov kernel $Q$ are compactly compatible;
\item \label{item_2_lem:stab_c1c_b} $\lambda\in\mrc^{k}(\msm)$ and for all $f \in \mrc^k(\msm)$, $\lambda Qf \in \mrc^k(\msm)$ and there exists a locally bounded function $\Psi :\msm \to \rset_+$ such that for all %compact set $\msk \subset \msm$, $\sup_{y \in \msk} \Psi(y) < \plusinfty$ and for all
 $x \in \msk$, $i \in\iint{1,k}$,
  \begin{equation*}
    \label{eq:1}
  \norm{\diffD^i (\lambda Qf)(x)} \leq    \norm{\Psi}_{\infty,\msk} \sup \ensemble{ \norm{\diffD^i f(y)}}{y \in \supp \{Q(x,\cdot)\} \eqsp,i \in\iint{1,k} };
  \end{equation*}
\item \label{item_3_lem:stab_c1c_b} we have $(t,x) \mapsto \varphi_t(x)$ is $k$-times continuously differentiable on $\rset_+ \times \msm$  and for all compact $\msk \subset \msm$ and $t \geq 0$,
  \begin{equation*}
    \sup \ensemble{\norm{\diffD^i \varphi_s(x)}}{s \in \ccint{0,t}, x \in \msk\eqsp, i \in \iint{1,k}} < \plusinfty \eqsp.
  \end{equation*}
\end{enumerate}
Then for all function $f \in \mrc^k(\msm)$, and $T \in \rset_+$,
$P_Tf \in \mrc^k(\msm)$. In addition, for all compact set
$\msk \subset\msm$, and $T \in \rset_+$, there exists $C \geq 0$ such
that for all $t \in \ccint{0,T}$,
$\sup_{x \in \msk} \absLigne{P_t f}(x) + \sup_{x \in \msk, i
  \in\iint{1,k}} \norm{\diffD^i P_t f(x)} \leq C$.
\end{remark}

% Combining \Cref{lem:compact_sup}-\Cref{propo:c_1_c_sgenerator}-\Cref{lem:stab_c1c} and \cite[Proposition 3.3, Chapter 1]{ethier:kurtz:1986}, we get the following result.
We say that a Markov semigroup $(P_t)_{t \geq 0}$ is Feller if $\mrc_0(\msm) \subset \mrb_0(\msm)$, where $\mrb_0(\msm)$ is defined by \eqref{eq:def_mrb_0}, and for any $f \in \rmc_0(\msm)$, and $t \geq 0$, $P_t f \in \rmc_0(\msm)$, where $\rmc_0(\msm)$ is equipped with the uniform norm $f \mapsto \norm{f}_{\infty}$
We refer to \Cref{sec:gener-contr-semigr} for the definition of a core of a closed operator.
\begin{corollary}
\label{coro:core_generator}
    Let $(P_t)_{t\geqslant 0}$ be a non explosive PDMP semigroup on $\msm$ with corresponding characteristics $(\varphi,\lambda,Q)$ satisfying
 \Cref{ass:compact_sup}-\Cref{ass:stability}. Then, $(P_t)_{t \geq 0}$ is Feller and $\rmc^1_c(\msm)$ is a core for the strong generator of $(P_t)_{t \geq 0}$ seen as a semigroup on $\rmc_0(\msm)$.
\end{corollary}
\begin{proof}
  By \Cref{propo:c_1_c_sgenerator}-\ref{propo:c_1_c_sgenerator_a} and \Cref{lem:stab_c1c}, $(P_t)_{t \geq 0}$ is Feller. The second statement is a consequence of  \Cref{lem:stab_c1c} and \cite[Proposition 3.3, Chapter 1]{ethier:kurtz:1986}.
\end{proof}

\begin{example*}
For the BPS process, $\lambda$ is not in $\rmc^1(\msm)$ so that we cannot apply the previous theory. One aim of the following section will then be to introduce a framework where we may overpass this limitation.
\end{example*}
%%% Local Variables:
%%% mode: latex
%%% TeX-master: "main"
%%% End:

%% file: invariant_1.tex
\section{Core and invariant distributions  of PDMPs}
\label{eq:invariant_main}

\tcr{Building upon the two previous sections,} the main purpose of this section is to provide a practical conditions on characteristics $(\varphi,\lambda,Q)$ such that if a probability measure $\mu$ on $(\msm,\mcb{\msm})$ satisfies for all $f \in \mrc^1_c(\msm)$, $\int_{\msm} \generator f(x) \rmd \mu(x) = 0$, where $\generator$ is the extended generator of the semigroup $(P_t)_{t \geq 0}$ associated with $(\varphi,\lambda,Q)$, then $\mu$ is invariant for $(P_t)_{t \geq 0}$. Since the extended generator of $(P_t)_{t \geq 0}$ is an extension of the strong generator $\sgenerator$, it is enough to show that $\mrc^1_c(\msm)$ is a core for $\sgenerator$ by \cite[Proposition 9.2]{ethier:kurtz:1986}. To do so, we introduce the following condition on $(P_t)_{t \geq 0}$. 

\begin{definition}
  We say that the PDMP semi-group $(P_t)_{t \geq 0}$ with   characteristics $( \varphi,\lambda,Q)$ is smoothly and compactly approximable if
  for all $\varepsilon>0$ there exist  characteristics  $( \varphi,\lambda^\varepsilon,Q^\varepsilon)$ satisfying \Cref{ass:compact_sup}-\Cref{ass:stability} and 
\begin{equation}
\label{eq:def_sm_comp_approx}
  \sup_{\substack{x \in \msm\\ \msa \in \mcbb(\msm)}} \defEns{\lambda^\varepsilon (x) \wedge \lambda (x) \abs{  Q^\varepsilon(x,\msa)  -  Q (x,\msa) } + \abs{\lambda^\varepsilon (x) - \lambda (x)} } \leq \veps \eqsp.
\end{equation}
\end{definition}

We may now give the main result of this section.

\begin{theorem}
  \label{theo:core_smooth_approx}
    Let $(P_t)_{t \geq 0}$ be a non explosive PDMP semigroup with
  characteristics $(\varphi,\lambda,Q)$ which is smoothly and
  compactly approximable. Then $(P_t)_{t \geq 0}$ is a Feller semigroup and  $\mrc^1_c(\msm)$ is a core for the
  strong generator $\sgenerator$ of $(P_t)_{t \geq 0}$ seen as an operator on $\rmc_0(\msm)$.
\end{theorem}

\begin{proof}
Suppose for the moment that $(P_t)_{t \geq 0}$ is Feller. 
For
  any $\veps >0$, consider a PDMP semi-group
  $(P^{\veps}_s)_{s \geq 0}$ with characteristics
  $(\varphi,\lambda^{\veps}, Q^{\veps})$ satisfying
  \Cref{ass:compact_sup}-\Cref{ass:stability} and
  \eqref{eq:def_sm_comp_approx}. By \Cref{lem:stab_c1c} and \Cref{coro:core_generator},  for
  any $n\in \nsets$, $(P^{1/n}_s)_{s \geq 0}$ is Feller and for any $t \geq 0$, $P_t^{1/n}\rmc^1_c(\msm) \subset \rmc^1_c(\msm)$. In addition, 
  \Cref{prop:synchrone} implies that  for all $x \in \msm$, $n \in \nsets$ and $t \geq 0$, $    \tvnorm{\updelta_x P_t  - \updelta_x P^{1/n}_t  } \leq 2(1-\rme^{-t/n}) \leq 2t/n $,   and therefore, for any $f \in \rmc_0(\msm)$, 
  \begin{equation}
    \label{eq:proof_sc_approx_feller_0}
\normLigne{P_tf - P_t^{1/n}f}_{\infty} \leq 2t \norm{f}_{\infty} /n \eqsp.
  \end{equation}
 Then, all the conditions of \Cref{propo:core_approximation} in \Cref{sec:gener-contr-semigr} are satisfied and we can conclude that $\rmc_c^1(\msm)$ is a core for $(P_t)_{t \geq 0}$. 

 We now prove  that  $(P_t)_{t \geq 0}$ is a Feller semigroup.  We first show that for any $f \in \rmc_0(\msm)$ and $t \geq 0$, $P_t f \in \rmc_0(\msm)$. Indeed,  \eqref{eq:proof_sc_approx_feller_0} gives that for any $t \geq 0$ and $f \in \rmc_0(\msm)$, $P_tf$ is the uniform limit of $( P_t^{1/n}f)_{n \in \nsets}$, therefore since $\mrc_0(\msm)$ is closed for the uniform topology, $P_tf \in \rmc_0(\msm)$. Finally, by \eqref{eq:proof_sc_approx_feller_0} applied with $n =1$, we have for any $f \in \rmc_0(\msm)$ and $t >0$, $\normLigne{P_tf -f }_{\infty} \leq 2 t + \normLigne{P_t^{1}f -f }_{\infty}$. Since  $(P^1_t)_{t \geq 0}$ is Feller, we get $\lim_{t \downarrow 0} \normLigne{P_tf -f }_{\infty}=0$, which implies that $\rmc_0(\msm) \subset \mrb_0(\msm)$ and completes the proof.

% for any $n \in \nsets$,     which is equivalent to 
%   \begin{equation}
%     \label{eq:proof_sc_approx_feller_1}
% \text{    for any $f \in \rmc_c(\msm)$ and
%   $t \geq 0$, $P_tf \in \rmc_0(\msm)$} \eqsp.
%   \end{equation}
%   Indeed, consider for the moment that this results holds and let $t \geq 0$ and $f\in \rmc_0(\msm)$. Since $\rmc_c(\msm)$ is
%   dense in $\rmc_0(\msm)$ for the topology induced by the uniform norm
%   on $\msm$,  there exists a sequence   $(f_m)_{m \in \nset}$,  $f_m \in \rmc_c(\msm)$ for any $m \in \nset$, such that $\lim_{m \to \plusinfty}\norm{f-f_m}_{\infty} = 0$. Then, since for any bounded and measurable function $g: \msm \to \rset$, $\norm{P_t g}_{\infty} \leq \norm{g}_{\infty}$, we obtain that $\lim_{m \to \plusinfty} \norm{P_t f - P_t f_m}_{\infty} =0$, which implies by \eqref{eq:proof_sc_approx_feller_1}  since  $(\rmc_0(\msm),\norm{\cdot}_{\infty})$ is closed, that $P_t f \in \rmc_0(\msm)$.  We proceed to show \eqref{eq:proof_sc_approx_feller_1}. Let $t \geq 0$ and $f \in \rmc_c(\msm)$.  By
% \eqref{eq:proof_sc_approx_feller_0}, we get for any  $x,y \in \msm$ and $n \in \nsets$,
%   \begin{equation*}
%     \abs{P_tf(x) - P_tf(y)} \leq 2t/n +    \abs{P_t^{1/n}f(x) - P_t^{1/n}f(y)} \eqsp.
%   \end{equation*}
%   It follows that for any $x \in \msm$ and $n \in \nsets$,  $\limsup_{y \to x}    \abs{P_tf(x) - P_tf(y)} \leq 2t/n$ since  $(P_s^{1/n})_{s \geq 0}$ is Feller, and taking $n \to \plusinfty$ completes the proof of \eqref{eq:proof_sc_approx_feller_1}.

\end{proof}
\begin{corollary}\label{lem:critere-invariant}
  Let $(P_t)_{t \geq 0}$ be a non explosive PDMP semigroup with
  characteristics $(\varphi,\lambda,Q)$ and $\mu$ be  a measure on
  $(\msm,\mcbb(\msm))$. Assume
  $\int_{\msm} \abs{\lambda} (x) \mu(\rmd x) < \plusinfty$ and that
  $(P_t)_{t \geq 0}$ is smoothly and compactly approximable. In
  addition, suppose that the generator $\generator$ associated with
  $(P_t)_{t \geq 0}$ satisfies for all $f \in \mrc^1_c(\msm)$,
  \begin{equation}
    \label{eq:critere_inv}
    \int_{\msm} \generator f(x) \mu (\rmd x) = 0 \eqsp.
  \end{equation}
  Then $\mu$ is invariant for $(P_t)_{t \geq 0}$. 
  % Let $\generator$ be the generator of the semigroup $(P_t)_{t\geqslant 0}$ of a homogeneous PDMP with characteristics $( \varphi,(\lambda_i,Q_i)_{i\in \iintLigne{1,n}})$, and let $\mu\in \mcp(\msm)$. Suppose the PDMP is nicely approximable  and $\int_{\msm} \lambda_i \rmd \mu < \infty$ for all $i\in \iintLigne{1,n}$.
%that satisfies
%\begin{eqnarray*}
%\sum_{i=1}^n \int_{\msm} |\lambda_i^\varepsilon - \lambda_i| \rmd \mu & \underset{\varepsilon\rightarrow 0}\longrightarrow & 0\eqsp .
%\end{eqnarray*}
% Suppose moreover that
% \begin{eqnarray*}
% \int_{\msm} \generator f(x) \mu( \rmd x) & = & 0
% \end{eqnarray*}
% for all $f\in \mrc^{1}_c(\msm)$. Then $\mu$ is invariant for $(P_t)_{t\geqslant 0}$.
\end{corollary}

With this notion of smoothly approximable semigroup, we will be able to consider the BPS process.

\begin{example*}

\begin{proposition}\label{prop:approx-BPS}
  Let $U\in\rmc^2(\rset^d)$, $\rate>0$ and $\loiy\in\mcp(\msy)$. Assume that $\msy$ is bounded. Then the associated BPS on $\rset^d \times \msy$, given by \Cref{ex:BPS}, is smoothly and compactly  approximable.
\end{proposition}

\begin{proof}
  Let $(\varphi,\lambda,Q)$ be the  characteristics of the BPS given by \Cref{ex:superposition:bps}.
  %We first consider the case where $\msy$ is bounded.
  Let $\veps > 0$, $\lambda^{\veps}: \msm \to \rset_+$ and $Q^{\veps}$ be  a Markov kernel on $(\rset^d\times \msy,\mcbb(\rset^d\times \msy))$ defined for all $(x,y)\in \rset^d \times \msy$ and $\msa \in \mcbb(\rset^d)$ by
  \begin{equation}
    \label{eq:def_lambda_eps_bps}
    \lambda^{\veps}(x,y) = \lambda_1^{\veps}(x,y) + \rate \eqsp, \, \lambda_1^{\veps}(x,y) = (\ps{y}{\nabla U(x)}-\veps)^2_+ / (\veps + (\ps{y}{\nabla U(x)}-\veps)_+) \eqsp,
  \end{equation}
  \begin{equation*}
    Q^{\veps}((x,y),\msa) =    (1/\lambda^{\veps}(x,y))\defEns{\lambda_1^{\veps}(x,y) \updelta_{(x,\Refl(x,y))}(\msa) + \rate (\updelta_x \otimes \loiy)(\msa)} \eqsp,
  \end{equation*}
  where $\Refl$ is defined by \eqref{eq:def_refl} for $g = \nabla U$. Then, similarly to the BPS process, $(\varphi,\lambda^{\veps},Q^{\veps})$ defines a non explosive semi-group $(P_t^{\veps})_{t\geq 0}$ on $(\rset^d \times \msy) \times \mcbb (\rset^d \times \msy)$. In addition we have
     \begin{equation}
     \label{eq:bound_distance_lambda_lambda_eps}
    \sup_{(x,y) \in \rset^d \times \msy  } \abs{ \lambda^{(\varepsilon)}(x,y) -
    \lambda(x,y) } \leq 2 \varepsilon  \eqsp.
   \end{equation}
   Therefore, using \Cref{rem:sync}, we get
   \begin{multline*}
\sup_{\substack{(x,y) \in \rset^d \times \msy\\ \msa \in \mcbb(\rset^d \times \msy)}} \defEns{\lambda^\varepsilon (x,y) \wedge \lambda (x,y) \abs{  Q^\varepsilon((x,y),\msa)  -  Q ((x,y),\msa) } + \abs{\lambda^\varepsilon (x,y) - \lambda (x,y)} } \\ \leq 2 \sup_{\substack{(x,y) \in \rset^d \times \msy\\ \msa \in \mcbb(\rset^d \times \msy)}} \abs{\lambda^\varepsilon (x,y)   Q^\varepsilon((x,y),\msa)  -  \lambda (x,y)  Q ((x,y),\msa) } \leq 4 \veps   \eqsp,
\end{multline*}
which shows \eqref{eq:def_sm_comp_approx}.

Since $\msy$ is assumed to be  bounded, $\msy \subset\cball{0}{M_{\msy}}$, with $M_\msy \in \rset_+$. Therefore by definition,  for all $t \in \rset_+$ and $(x,y) \in \rset^d \times \msy$, we have $P^{\veps}_t((x,y),\cball{x}{t M_\msy} \times \msy) = 1$ and $P^{\veps}_t$ satisfies \Cref{ass:compact_sup}.

Finally, we show that \Cref{ass:stability} is satisfied. \Cref{ass:stability}-\ref{item_3_lem:stab_c1c_b} trivially holds by definition of $\varphi$. 

For all closed ball, $\cball{0}{M} \subset \rset^d$, $M \in \rset_+$, $Q^{\veps}((x,y),\cball{0}{M} \times \msy) = 1$ for all $x \in \cball{0}{M}$. For all compact set $\msk \subset \ball{0}{M} \times \msy \subset \rset^d \times \msy$, $M \geq 0$,  and $T \in \rset_+$, define $\tilde{K} = \ball{0}{M+TR} \times \msy$, with  $R = \sup_{y \in \msy} \norm{y}$. Then, for all $n \in \nsets$, $(t_i)_{i \in \iint{1,n}} \in \rset_+^n$,
$\sum_{i=1}^{n} t_i \leq T$, conditions \ref{def:compact_comp_item_1}-\ref{def:compact_comp_item_2} of \Cref{def:compact_comp} are satisfied with  $\tilde{K}$, and
% for all $i \in \iint{1,N+1}$ and $s_i \in \ccint{0,t_i}$ setting
% $t_0=0$,
$(\msk_i)_{i \in \iint{1,n}}$ given by
\begin{equation*}
  %\msk_0 = \cball{0}{M} \times \msy\eqsp, \,
  \msk_i =  \cball{0}{R_i } \times \msy \eqsp, \, R_i =M+R \sum_{j=0}^i t_j \eqsp.
\end{equation*}
 Thus, $\varphi$ and $Q^{\veps}$ are compactly compatible.

    Then note that for all $\veps>0$,  $\lambda^{\veps}$ is
  continuously differentiable on $\msm$ since
  $t \mapsto (t-\veps)_+^2/(\veps+(t-\veps)_+)$ is on $\rset$; and its gradient is given for all $x \in \msm$ by
  \begin{equation*}
%    \label{eq:gradient_lambda_eps}
    \nabla(\lambda^{(\varepsilon)})(x,y) =
    \begin{cases}
      \frac{(\ps{y}{\nabla U(x)}-\varepsilon)(\ps{y}{\nabla U(x)}+\varepsilon)}{\ps{y}{\nabla U(x)}^2} \begin{pmatrix}  \nabla^2 U(x) y , \nabla U(x)
      \end{pmatrix} & \text{ if $\ps{y}{\nabla U(x)} \geq \varepsilon$}\\
      0 & \text{ otherwise} \eqsp.
    \end{cases}
  \end{equation*}
  In addition, for all continuously differentiable function $f: \rset^d \times \msy \to \rset$,  $(x,y) \in \rset^d \times \msy$, we have
  $    \lambda^{\veps}(x,y)  Q^{\veps}f(x,y) = A_1(x,y) + A_2(x,y)$  where
  \begin{equation*}
    A_1(x,y) = \lambda_1^{\veps}(x,y) f(x,\Refl(x,y)) \eqsp, \, A_2(x,y) =  \rate \int_{\msy} f(x,\tilde y) \loiy(\rmd \tilde y) \eqsp.
  \end{equation*}
  We show that $A_1$ and $A_2$ are continuously differentiable and satisfy for $i=1,2$, for all compact set $\msk \in \rset^d \times \msy$, for all $(x,y) \in \msk$,
\begin{equation}
    \label{eq:bound_bps_A_1_2} 
  \norm{\diff A_i(x,y)} \leq \sup_{(w,z) \in \msk} \{\Psi_i(w,z)\}\sup\ensemble{\abs{f}(x,\tilde{y}) + \norm{\diff f(x,\tilde{y})}}{\ty \in \msy}   \eqsp,
\end{equation}
  where $\Psi_i : \rset^d \times \msy \to \rset_+$, $i=1,2$, are
  bounded on compact sets of $\rset^d \times \msy$. Note that if we
  show \eqref{eq:bound_bps_A_1_2}, since for all
  $(x,y) \in \rset^d \times \msy$,
  $\supp \{Q((x,y),\cdot)\} = \{x\} \times \msy$, this result concludes the proof that \Cref{ass:stability}-\ref{item_2_lem:stab_c1c_b} holds.
  
  First, for all $(x,y) \in \rset^d \times \msy$, $\Refl$ is
  continuously differentiable at $(x,y)$ is
  $\ps{y}{\nabla U(x)} \not =0$. Since $f$, $\lambda_1^{\veps}$ are
  continuously differentiable and $\lambda_1^{\veps}(x,y) = 0$ if
  $\ps{y}{\nabla U(x)} \leq \veps$, $A_1$ is continuously
  differentiable and satisfies for all $(x,y) \in \rset^d \times y$,
  $\ps{y}{\nabla U(x)} \geq \veps$
  \begin{equation}
    \label{eq:bound_bps_A_1}
\norm{\diff A_1(x,y)} \leq  \norm{\diff \lambda_1^{\veps}(x,y)} \abs{f}(x,\Refl(x,y)) + \lambda_1^{\veps}(x,y) (1+\norm{\diff \Refl(x,y)})\norm{\diff f(x,\Refl(x,y))} \eqsp. 
\end{equation}
Regarding $A_2$, we have for all $(x,y) \in \rset^d \times \msy$, since $\diff f$ is bounded on all compact sets of $\rset^d \times \msy$ and the Lebesgue dominated convergence theorem,
\begin{equation*}
\norm{\diff A_2(x,y)} \leq \rate \int_\msy \norm{\diff_x f(x,\tilde y)} \loiy(\tilde y) \leq \rate \sup_{y \in \msy}\norm{\diff f(x,\tilde y)}\eqsp,
\end{equation*}
where $\diff_x$ is the differential operator with respect to the $x$-variable. Combining this result and \eqref{eq:bound_bps_A_1}, we get that \eqref{eq:bound_bps_A_1_2} holds and therefore \Cref{ass:stability}-\ref{item_2_lem:stab_c1c_b} as well.

\end{proof}

\begin{corollary}\label{prop:invarince-BPS}
Consider the BPS characteristics $(\varphi,\lambda_1,Q_1,\lambda_2,Q_2)$  defined in \Cref{ex:BPS}, and let $(P_t)_{t\geqslant 0}$ be the
  corresponding  semigroup. Assume
  that $\loiy$ is rotation invariant,~\ie~for all
  $O \in \rset^{d\times d}$, $O^{\transpose}O = \Id$,
  $\loiy(O \msa) = \loiy(\msa)$, for all $\msa \in\mcbb(\msy)$. In addition, suppose that
  \begin{equation}
    \label{eq:prop:invarince-BPS}
    \int_{\rset^d} ( 1+\norm{\na U(x)} ) \rme^{-U(x)}\rmd x \ < \ \infty \eqsp,\quad \int_\msy \norm{y} \loiy(\rmd y)\ < \ \infty \eqsp. 
  \end{equation}
Then $\tpi = \pi \otimes \loiy$ is invariant for $(P_t)_{t\geqslant 0}$, where $\pi$ is the probability measure on $(\rset^d,\mcbb(\rset^d))$ with density with respect to the Lebesgue measure proportional to $x \mapsto \rme^{-U(x)}$.  
\end{corollary}

\begin{proof}
The case where $\msy$ is bounded is a direct consequence of \Cref{theo:core_smooth_approx}. 
  
  In the case where $\msy$ is not bounded, consider the conditional distribution associated with $\loiy$ defined for all $\msa \in \mcb{\msy}$, 
\begin{equation*}
%  \label{eq:loiy_R}
  \loiy^{R}(\msa) = \loiy(\ball{0}{R}\cap \msa) / \loiy(\ball{x_0}{R}) \eqsp,
\end{equation*}
for  $R$ large enough such that $\loiy(\ball{0}{R}) \not = 0$.
Then, for any $\msa \in \mcb{\msy}$, we get $ \absLigne{  \loiy(\msa)  - \loiy^R(\msa) } \leq 2 \loiy(\msy \setminus \ball{0}{R})$, which implies that for
any $f : \rset^{d} \times \msy \to \rset$, bounded and measurable, $t \geq 0$, we
have for all $R \geq 0$ large enough,
\begin{align}
  \nonumber
  & \abs{\int_{\rset^{d} \times \msy} P_t f(x,y) \rmd \tilde{\pi}(x,y) - \int_{\rset^{d} \times \msy} f(x,y) \rmd \tilde{\pi}(x,y)} \\
  \nonumber
& \, \leq 2 \norm{f}_\infty  \loiy\parenthese{ \msy \setminus \ball{0}{R}}
 +   \abs{\int_{\rset^{d} \times \msy} P_t f(x,y) \rmd \tilde{\pi}^{R}(x,y) - \int_{\rset^{d} \times \msy} P_t^{R} f(x,y) \rmd \tilde{\pi}^{R}(x,y) } \eqsp.
\end{align}
Since $\lim_{R \to \plusinfty} \loiy\parentheseLigne{\msy  \setminus \ball{0}{R}} = 0$, it
remains to show that the last term in the right-hand side goes to $0$
as $R \to \plusinfty$ and the proof will be finished. Besides, note
that this result holds if for $t >0$, we  show that for all
$(x,y) \in \rset^{d} \times \msy$,
\begin{equation}
  \label{eq:invariance_non_bornee_2}
\lim_{R \to \plusinfty}   \tvnorm{ \updelta_{(x,y)} P_t  -  \updelta_{(x,y)} P_t^R } = 0 \eqsp.
\end{equation}
But by \Cref{prop:synchrone} and definition of  characteristics of $(P_s)_{s \geq 0}$ and $(P_s^R)_{s \geq 0}$, we get for all $t \geq 0$, $(x,y) \in \rset^d \times \msy$, 
\begin{equation*}
  \tvnorm{ \updelta_{(x,y)} P_t  -  \updelta_{(x,y)} P_t^R }  \leq 2 \{1-\exp(-2\rate \mu[\msy \setminus \ball{0}{R}])\} \eqsp,
\end{equation*}
which shows that \eqref{eq:invariance_non_bornee_2} holds. 
\end{proof}

\end{example*}
%%% Local Variables:
%%% mode: latex
%%% TeX-master: "main"
%%% End:

%% file: invariant_2.tex
\section{Stability of invariant measure and jump rate}
\label{sec:invariant_2}
We conclude this work with an asymptotic counterpart of the comparison theorems established in  \Cref{sec:synchrone}. % One may link it to the result for example by Rudolf and Schweizer \cite{rudolf:schweizer:2018} for stability results for Markov chains (in Wasserstein distance).
For a measurable function $V:\msm\rightarrow [1,+\infty)$ and $\nu_1$, $\nu_2\in\mcp(\msm)$, $\nu_1(V), \nu_2(V) < \plusinfty$, define the $V$-norm between $\nu_1$ and $\nu_2$ by
\begin{equation*}
\norm{\nu_1 - \nu_2}_V  =   \sup\left\{ \abs{\int_{\msm}  f \rmd \nu_1 - \int_{\msm}  f \rmd \nu_2 },\  \norm{f/V}_\infty \leq 1\right\}\eqsp.
\end{equation*}
We say that a semi-group on $\msm$ with invariant probability measure $\mu$ is $V$-uniformly geometrically ergodic with constants $C$, $\rho>0$ if for all $\nu\in \mcp(\msm)$, $\nu(V)< \plusinfty$ and $t\geqslant 0$,
\begin{equation*}
\norm{\mu  - \nu P_t}_V \leq C \rme^{-\rho t} \nu(V)\eqsp.
\end{equation*}

%Note that $(P_t)_{t \geq 0}$ is a contraction semigroup on $\mrb(\msm)$, \ie~for all $s,t \in \rset_+$, $P_{s+t} = P_t P_s$ and for all function $f \in \mrb(\msm)$, $\norm{P_tf}_{\infty} \leq  \norm{f}_{\infty}$. In addition, define the subset $\mrb_0(\msm) \subset \mrb(\msm)$ by
%\begin{equation*}
%  \mrb_0(\msm) = \ensemble{f \in \mrb(\msm)}{\lim_{t \to \plusinfty}\norm{P_t f-f}_{\infty} = 0} \eqsp.
%\end{equation*}
%By \cite[p.28-29]{davis1993markov}, $\mrb_0(\msm)$ is a closed subspace of $\mrb(\msm)$ and a Banach space for the uniform norm. Then by definition, $(P_t)_{t \geq 0}$ is a strongly continuous semigroup on $\rmb_0(\msm)$, \ie~for all $f \in \mrb_0(\msm)$, $\lim_{t \to \plusinfty}\norm{P_t f-f}_{\infty} = 0$.
%
%
%Define $(\sgenerator, \rmD(\sgenerator))$ the strong generator of $(P_t)_{t \geq 0}$ by
%\begin{align*}
%  \rmD(\sgenerator) &= \ensemble{f \in \mrb_0(\msm)}{\text{there exists $g: \msm \to \rset$ } \lim_{t \to \plusinfty} \norm{t^{-1}(P_tf-f)-g}_{\infty} = 0 } \eqsp,\\
%  \sgenerator f &= g \text{ for all $f \in \mrb_0(\msm)$} \eqsp. 
%\end{align*}

\begin{theorem}\label{prop:invariant_2}
Let $(P_t^1)_{t\geqslant 0}$  and $(P_t^2)_{t\geqslant 0}$ be two non-explosive homogeneous PDMP semi-group with characteristics $(\varphi,\lambda_1,Q_1)$ and $(\varphi,\lambda_2,Q_2)$ respectively. Let $\mu_1,\mu_2 \in \mcp(\msm)$ be invariant for $(P_t^1)_{t \geq 0}$ and $(P_t^2)_{t \geq 0}$ respectively. Suppose that $(P_t^1)_{t \geq 0}$ is $V$-uniformly geometrically ergodic with constants $C,\rho \in \rset_+^*$ for a function $V:\msm\rightarrow [1,+\infty)$ such that $\mu_2(V)<+\infty$.
Assume in addition that 
\begin{equation*}
 \rmc^1_c(\msm) \subset  \mrb_0^1(\msm) = \ensemble{f \in \mrb(\msm)}{\lim_{t \to 0}\norm{P_t^1 f-f}_{\infty} = 0} \eqsp,
\end{equation*}
 and that for all $t \in \rset_+$ and $f \in \rmc^1_c(\msm)$, $x \mapsto \int_0^t P_s^1 f(x) \rmd s \in \mrD(\sgenerator_2)$, where $(\sgenerator_2, \mrD(\sgenerator_2))$ is the strong generator of $(P^2_t)_{t \geq 0}$. 
Then
\begin{equation*}
  \norm{\mu_1  - \mu_2}_V \leq
C\rho^{-1}\sup\ensemble{\int_{\msm}    \abs{\lambda_1  Q_1h  - \lambda_2  Q_2h + (\lambda_2   - \lambda_1 )h }    \rmd \mu_2}{\norm{h/V}_\infty \leq 1}  \eqsp.
\end{equation*}
\end{theorem}

\begin{proof}
  By density, it is sufficient to bound $\absLigne{\mu_1(f) -\mu_2(f)}$ for all $f\in \mrc^1_c(\msm)$ with $\norm{f/V}_{\infty} \leq 1$. Let $f\in \mrc^1_c(\msm)$ with $\norm{f/V}_{\infty} \leq 1$ and, for $t\geqslant 0$, let $g_t = \int_0^t P_s^1 (f-\mu_1( f)) \rmd s$. %Since $\rmc^1_c(\msm) \subset \mrb_0(\msm)$, by \cite[Proposition 14.10]{davis:1993},
 According to \cite[Proposition 14.10]{davis:1993}, for all $t \geq 0$, $g_t \in \rmD(\sgenerator_1)$ as a sum of a constant function and of $\int_0^t P_s f \rmd s$ with $f\in \mrb_0^1(\msm)$, where $(\sgenerator_1, \mrD(\sgenerator_1))$ is the strong generator of $(P^1_t)_{t \geq 0}$; moreover it holds
  % , where $(\sgenerator^1,\mrD(\sgenerator^1))$ is the strong generator of $(P^1_s)_{s \geq 0}$ and
  \begin{equation*}
    \sgenerator_1 g_t =  \sgenerator_1 \int_0^t P_s f \rmd s  = P_t f -f \eqsp. 
  \end{equation*}
Using that   $(P^1_t)_{t \geq 0}$ is $V$-uniformly geometrically ergodic and that $g_t \in \rmD(\sgenerator_2)$, we obtain for all $t \geq 0$ and $x\in\msm$ that
\begin{align*}
|\mu_1( f)-f(x)-\sgenerator_2 g_t(x)| & \leq  \norm{\updelta_x P_t - \mu_1 }_V + |P_t f(x) - f(x)-\sgenerator_2 g_t(x)| \\
&\leq  C \rme^{-\rho t} V(x) + |\sgenerator_1 g_t(x) - \sgenerator_2 g_t(x)|\eqsp.
\end{align*}
% Therefore using $g_t \in \rmD(\sgenerator^2)$, for all $t \geq 0$ and $x \in \msm$,
%  \begin{equation*}
%    \abs{(\sgenerator^1-\sgenerator^2) g_t}(x) + \sgenerator^2g_t(x) \geq P_t^1 f(x) -f(x) \eqsp,
%  \end{equation*}
%  and since $(P^1_s)_{s \geq 0}$ is $V$-uniformly geometrically ergodic, we obtain for all $t \geq 0$ and $x \in \msm$,
%  \begin{equation}
%    \label{eq:diff_invariant_meas_1}
%    \abs{(\sgenerator^1-\sgenerator^2) g_t}(x) + \sgenerator^2g_t(x) + C \rme^{-\rho t} V(x) \geq \mu^1( f)-f(x) \eqsp,
%  \end{equation}
%  Similarly, we have for all $t \geq 0$ and $x \in \msm$,
%  \begin{equation}
%    \label{eq:diff_invariant_meas_2}
%    \abs{(\sgenerator^2-\sgenerator^1) g_t}(x) - \sgenerator^2g_t(x) + C \rme^{-\rho t} V(x) \geq f(x)-\mu^1( f) \eqsp,
%  \end{equation}
  In addition, by definition of $(\sgenerator_2,\rmD(\sgenerator_2))$
  and since $\mu_2$ is invariant for $(P_t^2)_{t \geq 0}$, then
  $\mu_2( \sgenerator_2 g_t) = 0$ for all $t\geqslant0$, so that% and taking the integral with respect to $\mu^2$ in \eqref{eq:diff_invariant_meas_1}-\eqref{eq:diff_invariant_meas_2}, we get $t \geq 0$ 
  \begin{equation*}
    \abs{\mu_1(f)-\mu_2(f)}  = \abs{\mu_2 \{ \mu_1(f) - f - \sgenerator_2 g_t\}} \leq \mu_2\po \abs{(\sgenerator_1-\sgenerator_2) g_t}\pf +  C \rme^{-\rho t} \mu_2(V)\eqsp.
    % \int_{\msm}  \defEns{   \abs{(\sgenerator^1-\sgenerator^2) g_t}(x) + C \rme^{-\rho t} V(x)} \rmd \mu^2(x)\eqsp.
  \end{equation*}
 Since $\rmD(\sgenerator_i) \subset \rmD(\generator_i)$, $i=1,2$, for all $t \geq 0$, $(\sgenerator_1-\sgenerator_2) g_t = (\generator_1-\generator_2) g_t$. Finally,   $(P^1_t)_{t \geq 0}$ being $V$-uniformly geometrically ergodic, then for all
  $x \in \msm$, $g_t(x) \leq (C/\rho) V(x)$.
%   Since $\rmD(\sgenerator^i) \subset \rmD(\generator^i)$, $i=1,2$, for all $t \geq 0$,
%  \begin{multline*}
%\abs{    \mu^1(f)-\mu^2(f)}  \leq C \rme^{-\rho t} \mu^2(V) \\
%    +(C/\rho) \sup_{h \in \mrb(\msm), \norm{h/V}_\infty \leq 1} \int_{\msm}  \defEns{ \abs{\lambda^1(x) Q^1h(x) - \lambda^2(x) Q^2h(x)+ \lambda^2(x) h(x) - \lambda^1(x)h(x)}(x) } \rmd \mu^2(x)\eqsp.
%  \end{multline*}
  The proof is then concluded taking $t \to \plusinfty$.
\end{proof}

\begin{example*}

Let us apply this result in the case of the Bouncy Particle Sampler.  Ergodicity of the BPS is studied in \cite{DurmusGuillinMonmarche:bouncy}, to which we will refer for details on this matter in the following. For the sake of simplicity, we will work under restrictive conditions.

\begin{proposition}\label{prop:bps-invariant2}
Consider the BPS  with characteristics $(\varphi,\lambda_1,Q_1,\lambda_2,Q_2)$  defined in \Cref{ex:BPS}, with $U \in \mrc^2(\msm)$ and $\msy \subset \cball{0}{1}$, and $(P_t)_{t\geqslant 0}$ the corresponding semigroup. For $M>0$, let $(P_t^M)_{t\geqslant0}$ be the PDMP semi-group with characteristics $(\varphi,\lambda_1\wedge M,Q_1,\lambda_2,Q_2)$.
  
   Assume that $\loiy$ is rotation invariant and \eqref{eq:prop:invarince-BPS} holds. In addition, assume that there exist $R>0$, $W\in\rmc^2(\rset^d)$ and $F\in\rmc^2(\rset)$ such that $U(x) = F(W(x))$ for $x\notin \ball{0}{R}$, $\norm{\diff W}_\infty + \norm{\diff^2 W}_\infty < +\infty$, $\int_\msm \exp(-W(x))\rmd x <\infty$, $\lim_{\norm{x} \to \plusinfty}W(x)= +\infty$  and $\lim_{w \in \plusinfty} F'(w)= +\infty$.
   Then there exists $C>0$ and $\tilde{M}>0$ such that for all $M>\tilde{M}$, $(P_t^M)_{t\geqslant 0}$ admits a unique invariant measure $\tpi_M$ that satisfies
 \begin{equation*}%\label{eq:bps-invariant2}
   \norm{\tpi  - \tpi_M}_{\rme^W} \ \leq \  C \int_{\msm}  \po \norm{\nabla U(x)} - M\pf_+ \rme^{W(x) - U(x)}  \rmd x\eqsp,
 \end{equation*}
 where $\tpi = \pi \otimes \loiy$ where $\pi$ is the probability measure on $(\rset^d,\mcbb(\rset^d))$ with density with respect to the Lebesgue measure proportional to $x \mapsto \rme^{-U(x)}$.  
\end{proposition}
   
For example, if $U(x) = \ps{x}{Ax}$ for some definite positive matrix $A$ outside a ball, then these conditions are satisfied by $W(x) = \sqrt{1+U(x)}$ and $F(w)=w^2-1$. In this case, \Cref{prop:bps-invariant2} implies that there exist $M_0,C,c>0$ such that for all $M>M_0$,
 \begin{equation*}%\label{eq:bps-invariant2}
   \norm{\tpi  - \tpi_M}_{\rme^W} \ \leq \  C \rme^{-c M^2}\eqsp.
 \end{equation*}

 Note that for all $M>0$, $(P_t^M)_{t\geqslant 0}$ can be sampled by thinning procedures, see \cite{lewis:shedler:1979,Thieullen2016}.

\begin{proof}

 For all $\varepsilon>0$, in a similar way as in \Cref{prop:approx-BPS}, we can construct a semi-group $(P_t^{M,\varepsilon})_{t\geqslant 0}$ with characteristics $(\varphi,\lambda_{M,\varepsilon},Q_1,\lambda_2,Q_2)$ where $\lambda_{M,\varepsilon}$ is such that $(P_t^{M,\varepsilon})_{t\geqslant 0}$ satisfies \Cref{ass:compact_sup}-\Cref{ass:stability} and that 
 \begin{equation}
   \label{eq:bound_lambda_1_eps_M}
   \sup_{(x,y)\in\rset^d \times \msy} |\lambda_1(x,y) \wedge M - \lambda_{M,\varepsilon}(x,y)| \ \leq \ \varepsilon\eqsp.
 \end{equation}
 Let us show that we can apply \Cref{prop:invariant_2} twice, with each time $P_t^1 = P_t^{M, \varepsilon}$ and $P_t^2$ equal either to $P_t^M$ or $P_t$. Consider the Lyapunov function defined for all $(x,y)\in \rset^d \times \msy$ by $V(x,y) = \exp( W(x)) \phi(\ps{y}{\na W(x)})$, where $\phi\in\rmc^2(\rset)$ is an increasing function, with $\phi(r)=1$ for $r\leq - 2$ and $\phi(r) \leq  3$ for $r\geqslant 1$.

 We show first that $(P_t)_{t \geq 0}$, $( P_t^{M})_{t \geq 0}$ and $ ( P_t^{M, \varepsilon})_{t \geq 0}$ are $V$-uniformly geometrically ergodic, which will imply that all these semi-groups admit a unique stationary measure for which $V$ is integrable.  For $h>0$, let $\generator_{hW}$ be the generator of the BPS semi-group with potential $hW$, refreshment rate $\rate>0$ and refreshment law $\loiy$. Following \cite[Section~3.2]{DurmusGuillinMonmarche:bouncy}, there exist $\phi: \rset \to \ccint{1,3}$ and $h_0 >0$  such that there exist $\alpha,C>0$ satisfying 
 \[\generator_{h_0 W} V \ \leq \ -\alpha V + C\eqsp.\]
Moreover, for all $M>0$, denoting $\generator_M$ the generator of $(P_t^M)_{t\geqslant 0}$,
\begin{equation*}
  (\generator_M - \generator_{h_0 W}) V(x,y)   =    \parentheseDeux{ \lambda_1(x,y)\wedge M - h_0\ps{y}{\na W(x)}_+}  \parentheseDeux{ \phi(-\ps{y}{\na W(x)}) - \phi(\ps{y}{\na W(x)})} \rme^{W(x)}\eqsp.
\end{equation*}
Note that, since $\msy \subset \cball{0}{1}$, $W$ is Lipschitz and $F'$ and $W$ going to infinity at infinity, then for $M_0=h_0\norm{\na W}_\infty$ and some $R_0>R$ large enough, for all $M\geqslant M_0$, $x\notin \ball{0}{R_0}$ and $y\in\msy$,
\begin{align*}
  \lambda_1(x,y)\wedge M - h_0\ps{y}{\na W(x)}_+  & \geqslant  \po \po F'(W(x))-h_0\pf \ps{y}{\na W(x)}_+\pf \wedge \po M - h_0\norm{\na W}_\infty\pf \\
                                                  & \geqslant  0\eqsp.
\end{align*}
Besides, $\phi$ being increasing, $\phi(-r) - \phi(r) \leq 0$ for all $r\geqslant 0$, so that for all $M\geqslant M_0$, $x\notin \ball{0}{R_0}$ and $y\in\msy$,
\[\generator_M V(x,y) \ \leq \ \generator_{h_0 W} V(x,y) \ \leq \ -\alpha V(x,y) + C \eqsp.\]
Hence, for $M\geqslant M_0$ and all $(x,y)\in\rset^d\times\msy$,
\[\generator_M V(x,y) \ \leq \ -\alpha V(x,y) + C + \sup_{(x,y) \in \ball{0}{R_0}\times\msy} \abs{\generator_M V(x,y) + \alpha V(x,y)}\ \leq \ -\alpha V(x,y) + C'\eqsp,\]
for some $C'$ that does not depend on $M>M_0$, since $\lambda_1$ is bounded on $\ball{0}{R_0}\times\msy$.
By a similar argument, denoting $\generator_{M,\varepsilon}$ the generator of $(P_t^{M,\varepsilon})_{t\geqslant 0}$, 
\[\generator_{M,\varepsilon} V  =  \generator_{M} V + \po \generator_{M,\varepsilon} -\generator_{M}\pf V
 \leq  - \alpha V + 6 \varepsilon \rme^{W} +C  \leq  -(\alpha - 12\varepsilon) V + C\eqsp.\]
Hence, we have obtained $M_0,\varepsilon_0,\alpha',C'>0$ such that for all $M>M_0$, all $\varepsilon \in [0,\varepsilon_0]$ then
\begin{equation}\label{eq:bps-lyapunov}
\generator_{M,\varepsilon} V   \leq  - \alpha' V + C'\eqsp.
\end{equation}

Moreover, following \cite[Section~3.3]{DurmusGuillinMonmarche:bouncy}, for all compact set $\msk \subset \rset^d\times\msy$, there exist $\eta,t_0>0$ such that, for all $M>M_0$ and $\varepsilon\in[0,\varepsilon_0]$, for all $(x,y),(x',y')\in\msk$ and all $t\geqslant t_0$,
\begin{equation}\label{eq:bps-couplage}
\tvnorm{\updelta_{x,y} P_t^{M,\varepsilon} -\updelta_{x',y'} P_t^{M,\varepsilon}}  \leq  2(1-\eta)\eqsp.
\end{equation}
Note that, indeed, $\eta$ and $t_0$ do not depend on $M$ or $\varepsilon$ since their construction only involves the supremum of $\lambda_{M,\varepsilon}$ over some compact set, which is smaller than $\varepsilon_0$ plus the supremum of $\lambda_1$ over the same compact (see \cite[Section~3.3]{DurmusGuillinMonmarche:bouncy} for details).

By \cite[Theorem 6.1]{meyn:tweedie:1993:III}, \eqref{eq:bps-lyapunov} together with \eqref{eq:bps-couplage} implies that $(P_t^{M,\varepsilon})_{t\geqslant 0}$  admits a unique invariant measure $\tpi_{M,\varepsilon}$, $\tpi_{M,\varepsilon}(V) < \plusinfty$ and is $V$-uniformly ergodic for all $M>M_0$ and $\varepsilon\in[0,\varepsilon_0]$ with some constants that does not depend on $M$ nor $\varepsilon$ (see \cite[Section~3]{DurmusGuillinMonmarche:bouncy}  for details).
More precisely for all $M>M_0$ and $\varepsilon\in[0,\varepsilon_0]$, there exists $C \geq 0$ and $\rho >0$ such that for all initial distribution $\nu_0$, $\nu_0(V) < \plusinfty$,
\begin{equation}
  \label{eq:distance_inv_measure}
  \tvnorm{\nu_0 P_t^{M,\varepsilon} - \tpi_{M,\varepsilon}} \leq C \rho^t \nu_0(V) \eqsp.
\end{equation}

The case $\varepsilon=0$ in particular implies that $(P_t^{M})_{t\geqslant 0}$ admits an invariant measure $\tpi_M$ that satisfies $\tpi_M( V) < +\infty$. Besides, note that for $R_1>R_0$ large enough to ensure that $F(W(x)) \geqslant 2W(x)$ for all $x\notin \ball{0}{R_1}$,
\[\int_\msm \rme^{W(x) - U(x)}\rmd x \ \leq \ \int_{\ball{0}{R_1}} \rme^{W(x) - U(x)}\rmd x  + \int_{\rset^d \setminus \ball{0}{R_1}} \rme^{-W(x)}\rmd x \ < \ +\infty\eqsp,\]
so that $\tpi(V) <+\infty$. Finally, note that $V \leq \rme^W \leq 2 V$, so that the associated $V$-norms are equivalent.

In order to apply \Cref{prop:invariant_2}, it remains to check the regularity conditions. Since $(P_t^{M,\varepsilon})_{t\geqslant 0}$ satisfies  \Cref{ass:compact_sup}-\Cref{ass:stability}, by \Cref{lem:stab_c1c}-\Cref{lem:compact_sup}, $\int_0^t P_s^{M,\varepsilon} f \rmd s \in \rmc^1_c(\rset^d\times\msy)$ for all $f \in \rmc^1_c(\rset^d\times\msy)$, $t\geqslant 0$, $M,\varepsilon>0$.  By \Cref{propo:c_1_c_sgenerator}, for all $M,\varepsilon>0$, $\rmc^1_c(\rset^d\times\msy) \subset \mrD(\sgenerator_{M,\varepsilon} )\cap\mrD(\sgenerator_{M} )\cap \mrD(\sgenerator) $, where $\sgenerator$ is the strong generator of $(P_t)_{t\geqslant 0}$. As a consequence, we can apply twice \Cref{prop:invariant_2} which implies, combining with \eqref{eq:distance_inv_measure}, that there exists $C_1 \geq 0$ satisfying for any $M>M_0$  and  $\varepsilon\in\ccint{0,\varepsilon_0}$, 
\begin{align*}
\norm{\tpi - \tpi_M}_{\rme^W} %& \leq & \norm{\tpi - \tpi_M}_V \\
                               & \leq  \norm{\tpi - \tpi_{M,\varepsilon}}_{\rme^W}  + \norm{\tpi_{M,\varepsilon} - \tpi_M}_{\rme^W} \\
 & \leq  C_1 \sup\ensemble{  \int_{\rset^d\times\msy}   \abs{\lambda_1 - \lambda_{M,\varepsilon}} (|Q_1h| + |h|) \rmd \tpi}{\norm{h\rme^{-W}}_\infty \leq 1} \\
 &   +\  C_1 \sup\ensemble{\int_{\rset^d\times\msy}   \abs{\lambda_1\wedge M - \lambda_{M,\varepsilon}} (|Q_1h| + |h|) \rmd \tpi_M }{\norm{h\rme^{-W}}_\infty \leq 1} \eqsp,
\end{align*}
Using that $Q_1 \rme^W = \rme^W$, \eqref{eq:bound_lambda_1_eps_M}, $\tpi(\rme^{W})$ and $\tpi_M(\rme^{W}) < \plusinfty$, we obtain that there exists $C_2 \geq 0$ such that 
\[\norm{\tpi - \tpi_M}_{\rme^W}   \leq C_2\parentheseDeux{ \int_{\rset^d\times\msy}   \defEns{ \lambda_{1}-\lambda_1\wedge M } \rme^W \rmd \tpi + \varepsilon}\eqsp. \]
 Finally, the proof is concluded taking $\varepsilon \to 0$ and upon noting that  for all $(x,y) \in\rset^d\times\msy$,
\begin{align*}
& \lambda_{1}(x,y) -  \lambda_1(x,y) \wedge M   = (\ps{y}{\na U(x)}_+-M)\1_{ \coint{M, \plusinfty}}(\ps{y}{\na U(x)}_+)\\
 & \leq  (\norm{\na U(x)}-M)\1_{ \coint{M, \plusinfty}}(\ps{y}{\na U(x)}_+)\ \leq \ (\norm{\na U(x)}-M)_+\eqsp.
\end{align*}
\end{proof}
\end{example*}
 
%%% Local Variables:
%%% mode: latex
%%% TeX-master: "main"
%%% End:

%% file: core.tex
\section{On  generators of contraction semigroups   and their core}
\label{sec:gener-contr-semigr}
In this section, we detail how to establish that a
particular subspace is a core for the generator of a contraction
semigroup.  In all this section, we consider a semigroup
$(T_t)_{t \geq 0}$ on a Banach space $\msl$ equipped with the norm $\normL{\cdot}$. We assume the following condition on $(T_t)_{t \geq 0}$:
\begin{assumptionC}
  \label{assC:semigroup}
The semigroup $(T_t)_{t \geq 0}$ is a strongly continuous contraction on $\msl$ \ie~for any $f \in \msl$, $\lim_{s \to 0}\normL{T_s f -f} = 0$ and for any $t \geq 0$, $\normL{T_t f} \leq \normL{f}$.
\end{assumptionC} 
Next, we define the generator $\generatorT$ and its domain of $(T_t)_{t \geq 0}$ defined by
\begin{align*}
  \domain(\generatorT) &= \ensemble{f \in \msl}{\text{ there exists } g \in \msl\, , \, \lim_{t \to \plusinfty} \normL{t^{-1}(T_tf-f) -g} = 0} \\
  \generatorT f &= g \in \msl\, \eqsp, \text{ such that } \, \lim_{t \to \plusinfty} \normL{t^{-1}(T_tf-f) -g} = 0  \eqsp.
\end{align*}

Note that by the Hille-Yosida theorem \cite[Theorem
2.6]{ethier:kurtz:1986}, $\domain(\generatorT)$ is dense in $\msl$,
$\range(\lambda \Id -\generatorT)$ is dense in $\msl$ for some
$\lambda >0$, where $\range(\mathcal{C})$ stands for the range of an
operator $\mathcal{C}$ in $\msl$, and $\generatorT$ is dissipative, \ie~for any $f \in \domain(\generatorT)$ and $\lambda >0$, $\normL{(\lambda \Id - \generatorT)f} \geq \lambda \normL{f}$.

We are now interested in identifying a core for $\generatorT$. Recall
that a core $\msc$ of a closed operator $\generatorT$ with domain
$\domain(\generatorT)$ on $\msl$ is a subspace of
$\domain(\generatorT)$ such that the closure  of the restriction
$\restr{\generatorT}{\msc}$ of $\generatorT$ to $\msc$ is equal to
$\generatorT$. Our main tool is the following result which is in
essence a reformulation of \cite[Proposition 3.7]{ethier:kurtz:1986}.

\begin{proposition}
  \label{propo:core_approximation}
  Assume \Cref{assC:semigroup} and that there exists a sequence
  $\{(T_t^{n})_{t \geq 0} \, : \, n \in \nsets\}$ of semigroups
  satisfying \Cref{assC:semigroup} and for any $t \geq 0$,
  $n \in \nset$ and $f \in \mrl$,
  \begin{equation}
    \label{eq:condition_core_approx}
    \normL{T_tf - T_t^{n}f} \leq \varepsilon_n t \normL{f} \eqsp,
  \end{equation}
  for a
  sequence $(\varepsilon_n)_{n \in \nsets}$ such that
  $\lim_{n \to \plusinfty} \varepsilon_n =0$. Suppose in addition that
  there exists a dense subspace $\msc$ of $\msl$ such that $\msc \subset \domain(\generatorT)$ and that $T_t^nf \in \msc$ for any $t \geq 0$, $n \in \nsets$, $f \in \msc$. Then $\msc$ is a core for $(T_t)_{t \geq 0}$. 
\end{proposition}
\begin{proof}
  We show that
  \begin{equation}
    \label{eq:goal_core_propo}
    \msc \in \closure(\range(\restr{\{\lambda \Id -
    \generatorT\}}{\msc}))
  \end{equation}
  where $\closure(\msd)$ is the closure of $\msd \subset \msl$ in $\msl$. Indeed, if this statement holds and since $\msc$ is dense, then \cite[Proposition 3.1]{ethier:kurtz:1986} completes the proof.

  First, note that $\msc \subset \domain(\generatorT^n)$ for any $n \in \nsets$, where $\generatorT^n$ is the generator associated with $(T_t^n)_{t \geq 0}$  since $\msc$ is a core for $\generatorT^n$ by \cite[Proposition 3.3]{ethier:kurtz:1986}. We now show that for any $n \in \nsets$ and  $f \in \msc \subset \domain(\generatorT) \cap \domain(\generatorT^n)$, $\normL{\generatorT f - \generatorT^n f} \leq \varepsilon_n$. Indeed, we have by \eqref{eq:condition_core_approx}, 
  \begin{equation}
\label{eq:condition_core_approx_2}    
 \normL{\generatorT f - \generatorT^n f} \leq  \limsup_{s \to 0} s^{-1} \normL{T_s f - T_s^nf} \leq \varepsilon_n\eqsp.
  \end{equation}
  
We now turn in showing \eqref{eq:goal_core_propo}.  Let $f \in \msc$ and  for any $n \in \nsets$ and $p \in \nsets$, define
  \begin{equation*}
    f^{n}_p = p^{-1} \sum_{k=1}^{p^2} \rme^{-\lambda k /p} T_{k/p}^n f \eqsp. 
  \end{equation*}
  Note that for any $n,p \in \nsets$, $f^{n}_p \in \msc$ by assumption on $\{(T_t^{n})_{t \geq 0} \, : \, n \in \nsets\}$ and $\normL{f^n_p} \leq \lambda^{-1} C\normL{f}$ for some constant $C\geq 0$ independent of $n,p$  and  by \cite[Proposition 2.1, Proposition 1.5-(b)]{ethier:kurtz:1986} for any $n \in \nsets$,
  \begin{align*}
    \lim_{p \to \plusinfty}     (\lambda \Id - \generatorT^n)f_p^n & = (\lambda \Id - \generatorT^n) \int_{0}^{\plusinfty} \rme^{-\lambda s} T_s^n f \rmd s =  \int_{0}^{\plusinfty} \rme^{-\lambda s} T_s^n (\lambda \Id - \generatorT^n)f \rmd s \\
    &= (\lambda \Id - \generatorT^n)^{-1} (\lambda \Id - \generatorT^n) f = f \eqsp. 
  \end{align*}
  Then for any $n \in \nsets$, there exists $p_n \in \nsets$, such that $\normL{f - (\lambda \Id - \generatorT^n)f_{p_n}^n} \leq n^{-1}$ and using \eqref{eq:condition_core_approx_2}, we get for any $n \in \nsets$,
  \begin{equation*}
    \normL{f - (\lambda \Id - \generatorT)f_{p_n}^n} \leq      \normL{f - (\lambda \Id - \generatorT^n)f_{p_n}^n} + \normL{\generatorT f_{p_n}^n - \generatorT^nf_{p_n}^n} \leq n^{-1} + \varepsilon_n \eqsp.
  \end{equation*}
Taking the limit $n \to \plusinfty$ concludes the proof of \eqref{eq:goal_core_propo}.
\end{proof}

%%% Local Variables:
%%% mode: latex
%%% TeX-master: "main"
%%% End: